\documentclass[12pt]{amsart}

\usepackage{etex}
\usepackage[english]{babel}
\usepackage[cp1250]{inputenc}
\usepackage{amssymb,amsthm,amsfonts,amsmath}
\usepackage{mathrsfs}
\usepackage{verbatim}
\usepackage{enumitem}

\usepackage{url}
\usepackage{amsopn}
\usepackage{graphicx}
\usepackage[usenames, dvipsnames]{color}
\DeclareGraphicsExtensions{.pdf,.png,.jpg}

\usepackage{enumitem,titletoc}
\definecolor{darkblue}{rgb}{0.0, 0.0, 0.55}
\usepackage[colorlinks,linkcolor=BrickRed,citecolor=OliveGreen,urlcolor=darkblue,hypertexnames=true]{hyperref}
\usepackage[a4paper,margin=2.5cm]{geometry}
\renewcommand{\qedsymbol}{\rule[.12ex]{1.2ex}{1.2ex}}
\def\ds{\displaystyle}

\DeclareMathOperator{\symm}{Sym}
\DeclareMathOperator{\Span}{span}
\DeclareMathOperator{\re}{re}
\DeclareMathOperator{\im}{im}

\DeclareMathOperator{\Pos}{Pos}
\DeclareMathOperator{\Sq}{Sos}
\DeclareMathOperator{\Lf}{Lf}
\DeclareMathOperator{\Vol}{Vol}
\DeclareMathOperator{\conv}{conv}
\DeclareMathOperator{\codim}{codim}
\DeclareMathOperator{\SO}{SO}
\DeclareMathOperator{\SU}{SU}
\DeclareMathOperator{\U}{U}
\DeclareMathOperator{\CP}{CP}
\DeclareMathOperator{\pr}{pr}
\DeclareMathOperator{\sq}{sq}
\DeclareMathOperator{\gr}{gr}

\DeclareMathOperator{\Tr}{tr}

\numberwithin{equation}{section}

\newcommand{\x}{{\tt x}}
\newcommand{\z}{{\tt z}}
\newcommand{\w}{{\tt w}}
\newcommand{\Sym}{\mathbb{S}}

\newcommand{\RR}{\mathbb R}
\newcommand{\FF}{\mathbb F}
\newcommand{\HH}{\mathbb H}
\newcommand{\NN}{\mathbb N}

\newcommand{\CC}{\mathbb C}
\newcommand{\QQ}{\mathbb Q}

\newcommand{\y}{{\tt y}}
\newcommand{\e}{{\tt e}}
\newcommand{\f}{{\tt f}}

\newcommand{\cC}{\mathcal C}
\newcommand{\cU}{\mathcal U}
\newcommand{\cH}{\mathcal H}
\newcommand{\cK}{\mathcal K}
\newcommand{\cL}{\mathcal L}
\newcommand{\PP}{\mathbb P}
\newcommand{\dd}{{\rm d}}
\newcommand{\cS}{\mathcal S}
\newcommand{\cT}{\mathcal T}

\newcommand{\cM}{\mathcal M}

\newtheorem{theorem}{Theorem}[section]
\newtheorem{corollary}[theorem]{Corollary}
\newtheorem{lemma}[theorem]{Lemma}
\newtheorem{proposition}[theorem]{Proposition}

\newtheorem{algorithm}{Algorithm}[section]

\theoremstyle{definition}

\newtheorem{remark}[theorem]{Remark}
\newtheorem{example}[theorem]{Example}

\title[There are many more positive maps than completely positive maps]{There are many more positive maps than\\[1mm]completely positive maps}

\author[I. Klep]{Igor Klep${}^{1}$}
\address{Igor Klep, Department of Mathematics,
The University of Auckland, New Zealand}
\email{igor.klep@auckland.ac.nz}
\thanks{${}^1$Supported by the Marsden Fund Council of the Royal Society of New Zealand. Partially supported by the Slovenian Research Agency grants P1-0222, L1-6722 and JI-8132.}

\author[S. McCullough]{Scott McCullough${}^2$}
\address{Scott McCullough, Department of Mathematics\\
  University of Florida\\ Gainesville 
   }
   \email{sam@math.ufl.edu}
\thanks{${}^2$Research supported by the NSF grant DMS-1361501}

\author[K. \v Sivic]{Klemen \v Sivic${}^3$}
\address{Klemen \v Sivic, Faculty of Mathematics and Physics, University of Ljubljana, Slovenia}
\email{klemen.sivic@fmf.uni-lj.si}
\thanks{${}^3$Partially supported by the Slovenian Research Agency grants P1-0222, L1-6722 and JI-8132.}

\author[A. Zalar]{Alja\v z Zalar${}^4$}
\address{Alja\v z Zalar, Faculty of Mathematics and Physics, University of Ljubljana, Slovenia}
\email{aljaz.zalar@fmf.uni-lj.si}
\thanks{${}^4$Supported by the Slovenian Research Agency grant JI-8132.}


\date{\today}

\begin{document}

\setcounter{tocdepth}{3}
\contentsmargin{2.55em}
\dottedcontents{section}[3.8em]{}{2.3em}{.4pc}
\dottedcontents{subsection}[6.1em]{}{3.2em}{.4pc}
\dottedcontents{subsubsection}[8.4em]{}{4.1em}{.4pc}

\makeatletter
\newcommand{\mycontentsbox}{%
{
\addtolength{\parskip}{.3pt}
\tableofcontents}}
\def\enddoc@text{\ifx\@empty\@translators \else\@settranslators\fi
\ifx\@empty\addresses \else\@setaddresses\fi
\newpage\mycontentsbox\newpage\printindex}
\makeatother

\begin{abstract}
	A $\ast$-linear map $\Phi$ between matrix spaces is positive if it maps positive semidefinite matrices to positive semidefinite ones, and is called completely positive if all its ampliations $I_n\otimes \Phi$ are positive. In this article quantitative bounds on the fraction of positive maps that are completely positive are proved. A main tool is the real algebraic geometry techniques developed by Blekherman to study the gap between positive polynomials and sums of squares. Finally, an algorithm to produce positive maps that are not completely positive is given.
\end{abstract}

\maketitle

\section{Introduction}\label{introd}

	For $\FF\in \{\RR,\CC\}$ and $n\in \NN$, let $M_n(\FF)$ be the vector space of $n\times n$ matrices over $\FF$
	equipped with the \textbf{involution} $\ast$ which is conjugate transposition for $\FF=\CC$ and transposition for $\FF=\RR$.
	Let $\HH_n$ (resp.\ $\Sym_n$) stand for its subspace $\left\{ A\in M_n(\FF)\colon A^\ast=A\right\}$ of hermitian (resp.\ real symmetric) matrices.
	A matrix $A\in \HH_n$ (resp.\ $A\in \Sym_n$) is \textbf{positive semidefinite (psd)}
	if and only if all of its eigenvalues are nonnegative;
	equivalently, $v^\ast A v\geq 0$ for all $v\in \FF^n$.
	We write $A\succeq 0$.
	A linear map $\Phi:\cS\to \cT$ between matrix spaces is $\ast$-linear if 
	$\Phi(A^\ast)=\Phi(A)^\ast$ for all
	$A\in \cS$. It is 
	\textbf{positive} if
	$\Phi(A)\succeq 0$ for every $A\succeq 0$ in its domain $\cS$.
	For $k\in \NN$, a $\ast$-linear map $\Phi: \cS\to \cT$ induces a $\ast$-linear map
		$$	\Phi^{(k)}:M_k(\FF)\otimes \cS=M_{k}(\cS)\to M_k(\FF)\otimes \cT=M_{k}(\cT),\quad M\otimes 			A\mapsto M\otimes \Phi(A)$$
	where $\otimes$ stands for the Kronecker tensor product of matrices.
	A $\ast$-linear map $\Phi$ 
	is \textbf{$k$-positive} if
	$\Phi^{(k)}$ is positive. If $\Phi$ is $k$-positive for every $k\in \NN$, then $\Phi$ is
	\textbf{completely positive (cp)}.
	Obviously, every cp map is positive, and
	the transpose map $M_2(\FF)\to M_2(\FF)$
	is positive but not $2$-positive and thus not cp.

Positive maps occur frequently in
matrix theory \cite{Hog,Wor76} and functional
analysis (e.g., positive linear functionals).
Cp maps are ubiquitous in
quantum physics (where they are called quantum channels or
 operations) \cite{NC10}, and operator algebra
\cite{PAU}. Both types of maps are also
 important topics in random matrix theory
and  free probability \cite{VDN92}. 
In quantum information theory  cp maps
are used to describe the quantum mechanical generalization of a noisy channel. The Stinespring representation theorem
\cite[Theorem 4.1]{PAU} provides the justification for their physical interpretation as reduction of a unitary evolution to a subsystem.
Positive maps that are not cp do not possess such physical realizability, since they may fail to preserve positivity on entangled states.
However, they do preserve positivity on separable states, and thus are
of great importance for detecting  entanglement of a system. We refer to \cite{AS06,ASY14,HSR03,P-GWPR06,SWZ11} for a small sample of the vast
quantum information theory literature on entanglement breaking maps; see also \cite{JKPR11,Sto08,PTT11}.
Verifying whether a linear map is positive is computationally intractable; numerical algorithms, based on 
Lasserre's \cite{Las09} polynomial sum of squares relaxations for detecting positivity are given in \cite{NZ+}.

Recently
Collins, Hayden,  Nechita \cite{CHN} studied
entanglement breaking maps from a
free probability viewpoint \cite{VDN92} using von Neumann algebras.  Among other
results they present random techniques for constructing $k$-positive
maps that are not $k+1$-positive
in large dimensions
\cite[Theorem 4.2]{CHN}.
The gap between positive and cp maps was also
investigated by Arveson \cite{Arv09-1,Arv09-2},
and Aubrun, Szarek, Werner, Ye, \. Zyczkowski \cite{SWZ08,ASY14}.
Arveson used operator algebra to
establish:

\begin{theorem}[Arveson \cite{Arv09-1}]\label{thm:arv}
Let $n,m\geq 2$.  Then the probability $p$ that a  positive map $\varphi:M_n(\CC)\to M_m(\CC)$ is cp satisfies $0<p<1.$
\end{theorem}

\begin{remark}
Theorem \ref{thm:arv} is established by considering the dual problem to estimating the probability that a positive map $\varphi:M_n(\CC)\to M_m(\CC)$ is cp, which is to
estimate the probability that a random state on $M_{n}(\CC)\otimes M_m(\CC)$ is separable. Now we briefly explain the probability distribution on the state space
from \cite{Arv09-1}.  Arveson introduces a compact Riemannian manifold $\mathcal V$ of dimension $n^2(2m-1)$ on which the unitary group 
$\U(nm)$ acts as a transitive group of isometries and induces a probability measure on $\mathcal V$. The state space can be parametrized as the orbit space of the subgroup 
$\{ \left[\lambda_{ij} I_m\right]_{i,j=1}^n\colon \lambda_{ij}\in \CC \}$ of $\U(nm)$ where $I_m$ stands for the identity $m\times m$ matrix, and as such inherits the probability measure from $\mathcal V$ which is the underlying measure in Theorem \ref{thm:arv}.
\end{remark}

Szarek, Werner and \. Zyczkowski 
use classical convexity and
geometry of Banach spaces to
improve upon Arveson's results by providing quantitative bounds on the probability $p$
(in the case where $n=m$) and establish 
its asymptotic behavior, see \cite[Theorem 5]{SWZ08}.\looseness=-1

In this paper we
investigate the gap  between positive and
completely positive maps 
by translating the problem into
the language of real algebraic geometry \cite{BCR98}.

\subsection{Main results and reader's guide} \label{sec-main-res}

The contribution of this paper is threefold. First, 
we will
study
nonnegative biquadratic biforms that
are not sums of squares
by
estimating volumes of appropriate cones of positive polynomials.
The study of positive polynomials
is one of the pillars of real algebraic geometry,
starting with
Artin's solution of Hilbert's 17th problem, cf.~\cite{Mar08,Lau09,Rez95,Put93,Sce09,Scw03,KS10,Pow11,Cim12,Oza13}.
To estimate the ratio between compact base sections of the cones of 
sums of squares biforms and 
nonnegative biquadratic biforms we shall employ
powerful techniques,
based on harmonic analysis
and classical convexity,
developed by
Blekherman \cite{Blek1} and Barvinok-Blekherman \cite{Barv-Blek}.

Let $\RR[\x,\y]$ be the vector space of real polynomials in the variables $\x:=(x_1,\ldots,x_n)$ and $\y:=(y_1,\ldots,y_m)$. Let $\RR[\x,\y]_{k_1, k_2}$ be the subspace of \textbf{biforms of bidegree} $(k_1,k_2)$, i.e., polynomials from $\RR[\x,\y]$ that are homogeneous of degree $k_1$ in $\x$ and of degree $k_2$ in $\y$.
Note that the dimension of $\RR[\x,\y]_{k_1, k_2}$ is equal to $\binom{n+k_1-1}{k_1}\binom{m+k_2-1}{k_2}$.
Let
\begin{align}
	\Pos^{(n,m)}_{(2k_1,2k_2)}& = \left\{ f\in \RR[\x,\y]_{2k_1,2k_2}\colon f(\x,\y)\geq 0\quad \text{for all }(\x,\y)\in\RR^n\times \RR^m \right\}, \label{pos-def}\\
	\Sq^{(n,m)}_{(2k_1,2k_2)} &= \left\{ f\in \RR[\x,\y]_{2k_1,2k_2} \colon f=\sum_{i}f_i^2\quad \text{for some }f_i\in \RR[\x,\y]_{k_1,k_2}\right\}, \label{sos-def}
\end{align}

\noindent be the cone of nonnegative biforms and the cone of sums of squares biforms;
respectively.
In all but a few stray cases the cone of sums of squares biforms is strictly contained in the cone of nonnegative biforms.

\begin{theorem}[Choi, Lam, Reznick \protect{\cite[Theorem 8.4]{CLR}}] \label{n,m>3-intro}
	Let $n,m\geq 2$. Then $\Pos^{(n,m)}_{(2k_1,2k_2)}=\Sq^{(n,m)}_{(2k_1,2k_2)}$ if and only if 
	$n=2$ and $k_2=1$ or $m=2$ and $k_1=1$.
\end{theorem}


\noindent 
We shall estimate the gap between the cones $\Pos^{(n,m)}_{(2k_1,2k_2)}$ and 
$\Sq^{(n,m)}_{(2k_1,2k_2)}$ by comparing volumes of compact sections of these cones obtained by intersecting each with a suitably chosen affine hyperplane
$\mathcal{H}^{(n,m)}_{(2k_1,2k_2)}\subset \RR[\x,\y]_{2k_1,2k_2}.$
Let $T:=S^{n-1}\times S^{m-1}$
and consider the product measure $\sigma=\sigma_1\times\sigma_2$ on $T$,
where $S^{n-1}\subseteq \RR^n$, $S^{m-1}\subseteq \RR^m$ are the unit spheres and
$\sigma_1$, $\sigma_2$ are the normalized Lebesgue measures on $S^{n-1}$ and $S^{m-1}$, respectively.
The $L^p$ norm  of a biform $f\in \RR[\x,\y]_{2k_1,2k_2}$ on $T$ is given by
	\begin{equation*}\label{measure}
		\left\| f \right\|_p^p= \int_{T} |f|^p\ \dd\sigma=
		\int_{x\in S^{n-1}}\left(\int_{y\in S^{m-1}} |f(x,y)|^p \;\dd\sigma_2(y)\right) \;\dd	
		\sigma_1(x),
	\end{equation*} 
while the supremum norm by
	$$\|f\|_{\infty}:=\max_{(x,y)\in T} |f(x,y)|.$$
\noindent Let $\mathcal{H}^{(n,m)}_{(2k_1,2k_2)}$ be the hyperplane
 of biforms from $\RR[\x,\y]_{2k_1,2k_2}$ of average 1 on $T$, i.e.,
	$$\mathcal H^{(n,m)}_{(2k_1,2k_2)}=\left\{ f\in \RR[\x,\y]_{2k_1,2k_2}\colon \int_T f\; \dd\sigma=1 \right\}.$$
Let $\left(\Pos^{(n,m)}_{(2k_1,2k_2)}\right)'$ and $\left(\Sq^{(n,m)}_{(2k_1,2k_1)}\right)'$
be the sections
of the cones $\Pos^{(n,m)}_{(2k_1,2k_2)}$ and $\Sq^{(n,m)}_{(2k_1,2k_2)}$,
	\begin{eqnarray*}
		\left(\Pos^{(n,m)}_{(2k_1,2k_2)}\right)' &=& \Pos^{(n,m)}_{(2k_1,2k_2)}\;\bigcap\; \cH^{(n,m)}_{(2k_1,2k_2)},\\
		\left(\Sq^{(n,m)}_{(2k_1,2k_2)}\right)'	&=& \Sq^{(n,m)}_{(2k_1,2k_2)}\;\bigcap\; \cH^{(n,m)}_{(2k_1,2k_2)}.
	\end{eqnarray*}
Thus
$\left(\Pos^{(n,m)}_{(2k_1,2k_2)}\right)'$ and $\left(\Sq^{(n,m)}_{(2k_1,2k_2)}\right)'$ are convex and compact full-dimensional sets in the finite dimensional hyperplane $\cH^{(n,m)}_{(2k_1,2k_2)}$.
For technical reasons we translate these sections by subtracting the polynomial
$(\sum_{i=1}^n x_i^2)^{k_1}(\sum_{j=1}^m y_j^2)^{k_2}$, i.e.,
\[
\begin{split}
	\widetilde{\Pos}^{(n,m)}_{(2k_1,2k_1)}
			&=	\left\{ f\in \RR[\x,\y]_{2k_1,2k_2}\colon  f+(\sum_{i=1}^n x_i^2)^{k_1}(\sum_{j=1}^m y_j^2)^{k_2}
				\in \left(\Pos^{(n,m)}_{(2k_1,2k_2)}\right)'\right\},\\
					\widetilde{\Sq}^{(n,m)}_{(2k_1,2k_2)} &= \left\{ f\in \RR[\x,\y]_{2k_1,2k_2}\colon  f+(\sum_{i=1}^n x_i^2)^{k_1}(\sum_{j=1}^m y_j^2)^{k_2}
				\in \left(\Sq^{(n,m)}_{(2k_1,2k_2)}\right)' \right\}.
\end{split}
\]
Let $\cM:=\cM^{(n,m)}_{(2k_1,2k_2)}$  be the hyperplane  of
biforms from $\RR[\x,\y]_{2k_1,2k_2}$ with average 0 on $T$,
	\begin{equation} \label{M-def}
		\cM=\left\{ f\in \RR[\x,\y]_{2k_1,2k_2}\colon \int_T f\; \dd\sigma=0 \right\}. 
	\end{equation}
Notice that
	$$\widetilde{\Pos}^{(n,m)}_{(2k_1,2k_1)}\subseteq \cM
		\quad\text{and}\quad
		\widetilde{\Sq}^{(n,m)}_{(2k_1,2k_1)}\subseteq \cM.$$
The natural $L^2$ inner product in $\RR[\x,\y]_{2k_1,2k_2}$ is defined by 
	$$\langle f,g\rangle=\int_{T}fg\; d\sigma.$$
With this inner product $\cM$ is a Hilbert subspace of $\RR[\x,\y]_{2k_1,2k_2}$ of dimension $D_{\cM}$ 
and so it is isomorphic to $\RR^{D_\cM}$ as a Hilbert space.
Let $S_{\cM}$, $B_\cM$ be the unit sphere and the unit ball in $\cM$, respectively. 
Let $\psi:\RR^{D_\cM}\to \cM$ be a unitary isomorphism and $\psi_{\ast}\mu$ the pushforward of the
Lebesgue measure $\mu$ on $\RR^{D_{\cM}}$ to $\cM$, i.e., $\psi_\ast\mu(E):=\mu(\psi^{-1}(E))$ for every
Borel measurable set $E\subseteq \cM$.

\begin{lemma} \label{unique-pushforward}
	The measure of a Borel set $E\subseteq \cM$ does not depend on the choice of the unitary isomorphism $\psi$, i.e.,
	if $\psi_1:\RR^{D_\cM}\to \cM$ and $\psi_2:\RR^{D_\cM}\to \cM$ are unitary isomorphisms, then
	$(\psi_1)_\ast\mu(E)=(\psi_2)_\ast\mu(E)$.
\end{lemma}

\begin{proof}
	We have
		\begin{eqnarray*}
			(\psi_2)_\ast\mu(E)
				&=&
					\mu(\psi_2^{-1}(E))=\mu((\psi_2^{-1}\circ\psi_1\circ\psi_1^{-1})(E))=
					\mu((\psi_2^{-1}\circ\psi_1)(\psi_1^{-1}(E)))\\
				&=& \mu(\psi_1^{-1}(E))=(\psi_1)_\ast\mu(E),
		\end{eqnarray*}
	where the first equality in the second line holds since $\psi_2^{-1}\circ\psi_1$ is a linear isometry and $\mu$ is the Lebesgue measure.
\end{proof}

The bounds for the volume of the section of nonnegative biforms are as follows.

\begin{theorem}\label{psd-intro}
	For $n,m\in \NN$ we have
		$$c_{2k_1,2k_2}\leq \left(\frac{\Vol \widetilde{\Pos}^{(n,m)}_{(2k_1,2k_2)}}{\Vol B_{\cM}}\right)^{\frac{1}{D_{\cM}}}\leq
		2\left(\min\left(\frac{2k_1^2}{2k_1^2+n},  \frac{2k_2^2}{2k_2^2+m}\right)\right)^{\frac{1}{2}},$$
	where
		$$c_{2k_1,2k_2}=\ds
			\left\{\begin{array}{lr}
				3^{3}\cdot 10^{-\frac{20}{9}} \max(n,m)^{-\frac{1}{2}},& \text{if } k_1=k_2=1\\[1mm]
				\exp(-3) \left(2\lceil\max(n,m)\ln(2\max(k_1,k_2)+1)\rceil\right)^{-\frac{1}{2}} ,& \text{otherwise.}
			\end{array}\right.$$
\end{theorem}

Next we give bounds for the volume of the section of sums of squares biforms.

\begin{theorem} \label{squares-intro}
	For integers $n,m\geq 3$ we have
		$$d_{2k_1,2k_2}
			\leq
			\left(\frac{\Vol \widetilde{\Sq}^{(n,m)}_{(2k_1,2k_2)}}{\Vol B_\cM}\right)^{\frac{1}{D_{\cM}}}
			\leq e_{2k_1,2k_2},$$
	where
		\begin{eqnarray*}
		d_{2k_1,2k_2} &=&
			\left\{\begin{array}{lr}
				 2^{-8}\cdot 6^{-\frac{1}{2}} \cdot \frac{\sqrt{nm+n+m}}{(n+4)(m+4)},& \text{if } k_1=k_2=1\\
				\frac{(k_1!\ k_2!)^\frac{3}{2}}{2\sqrt{6}\cdot4^{2k_1+2k_2}\cdot \sqrt{(2k_1)!\ (2k_2)!}} \frac{  n^{\frac{k_1}{2}}  m^{\frac{k_2}{2}}  }{(\frac{n}{2}+2k_1)^{k_1} (\frac{m}{2}+2k_2)^{k_2}},&
						 \text{otherwise}
			\end{array}\right.,\\
		e_{2k_1,2k_2}&=&
			\left\{\begin{array}{lr}
				2^{10}\sqrt{6} \cdot \frac{1}{\sqrt{nm+n+m}},& \text{if } k_1=k_2=1\\
				2\sqrt{6}\cdot 4^{2k_1+2k_2}\cdot \sqrt{\frac{(2k_1)!\ (2k_2)!}{k_1!\ k_2!}} \cdot
					n^{-\frac{k_1}{2}} m^{-\frac{k_2}{2}},& \text{otherwise}
			\end{array}\right..
		\end{eqnarray*}
\end{theorem}

Combining the previous two theorems we obtain:

\begin{corollary}\label{cor:ratio}
For integers $n,m\geq 3$ we have
		$$f_{2k_1,2k_2}
			\leq
			\left(\frac{\Vol \widetilde{\Sq}^{(n,m)}_{(2k_1,2k_2)}}{\Vol \widetilde{\Pos}^{(n,m)}_{(2k_1,2k_2)}}\right)^{\frac{1}{D_{\cM}}}
			\leq g_{2k_1,2k_2},$$
	where
		\begin{eqnarray*}
		f_{2k_1,2k_2} &=&
			\left\{\begin{array}{lr}
				 \frac{3\sqrt{3}}{2^{10} 7^2\sqrt{\min(n,m)}},& \text{if } k_1=k_2=1\\
				C_{2k_1,2k_2}\cdot \frac{  \left(n^{k_1}  m^{k_2} \big(2+\max\big(\frac{n}{k_1^2},
					\frac{m}{k_2^2}\big)\big)\right)^{\frac{1}{2}} }
					{(\frac{n}{2}+2k_1)^{k_1} (\frac{m}{2}+2k_2)^{k_2}},&
						 \text{otherwise,}
			\end{array}\right.\\
		g_{2k_1,2k_2}&=&
			\left\{\begin{array}{lr}
				\frac{2^{12}\cdot 5^2 \cdot 6^{\frac{1}{2}} \cdot 10^{\frac{2}{9}} }{3^3\cdot \sqrt{\min(n,m)+1}},& \text{if } k_1=k_2=1\\
				D_{2k_1,2k_2}
					\cdot (n^{k_1-1} m^{k_2-1} \min(n,m))^{-\frac{1}{2}},& \text{otherwise,}
			\end{array}\right.
		\end{eqnarray*}
	and
	\begin{eqnarray*}
		C_{2k_1.2k_2} &=& \frac{2\cdot (k_1!\ k_2!)^\frac{3}{2}}{\sqrt{3}\cdot 4^{2(k_1+k_2+1)}\cdot \sqrt{(2k_1)!\ (2k_2)!}
				\cdot \min(k_1,k_2)},\\
		D_{2k_1,2k_2} &=& \sqrt{3}\cdot e^3 \cdot 4^{2k_1+2k_2+1}\cdot \sqrt{\frac{(2k_1)!\ (2k_2)!}{k_1!\ k_2!}   
					\lceil \ln(2\max(k_1,k_2)+1)  \rceil
					}.
	\end{eqnarray*}
\end{corollary}

\begin{proof}
	For the case $k_1=k_2=1$ see Theorem \ref{cor-1}. For the other constants
	$f_{2k_1,2k_2},g_{2k_1,2k_2},$ $C_{2k_1,2k_2},D_{2k_1,2k_2}$
	combine Theorems \ref{psd-intro} and \ref{squares-intro} together with the estimate
	$\lceil N t\rceil\leq N\lceil t\rceil$, used for $N=\max(n,m)$
	and $t=\ln(2\max(k_1,k_2)+1)$.
\end{proof}

Section \ref{sec:estimate} is devoted to
\noindent Theorems \ref{psd-intro} and \ref{squares-intro}.
The proof of Theorem \ref{psd-intro} is given in 
Subsection \ref{lower-bound}, while for the proof of Theorem \ref{squares-intro} we need
some preliminary results about the apolar inner product on the vector space of biforms introduced and studied in Subsection \ref{dif-inn-prod}.
Finally, Theorem \ref{squares-intro} is established in Subsection \ref{sos-subsec}.

The second contribution of the paper is the estimate on the gap between the cones of positive and completely positive maps. By converting the problem into the language of real algebraic geometry, the following estimate will follow
from Theorems \ref{psd-intro} and \ref{squares-intro} by choosing $k_1=k_2=1$.

\begin{corollary}	\label{verjetnost-intro}
	For integers $n,m\geq 3$ the probability $p_{n,m}$	that a 
	 positive map
		$\Phi:\Sym_{n}\to \Sym_{m}$
	is completely positive, is bounded by
		$$\left(\frac{3\sqrt{3}}{2^{10} \cdot 7^2\cdot \sqrt{\min(n,m)}}\right)^{D_\cM}< p_{n,m}<
			\left(\frac{2^{12}\cdot 5^2\cdot 6^{\frac{1}{2}}\cdot 10^{\frac{2}{9}}}{3^3\cdot \sqrt{\min(n,m)+1}}\right)^{D_\cM},$$
	where $D_\cM= 
	\binom{n+1}2 \binom{m+1}2-1.$
	In particular, if $\min(n,m)\geq \frac{2^{25}\cdot 5^4 \cdot 10^{\frac{4}{9}}}{3^5}$, then \[\lim_{\max(n,m)\to\infty}p_{n,m}=0.\]
\end{corollary}

\noindent Here, the probability $p_{n,m}$ is defined as the ratio between the volumes of the sections $\widetilde{\Sq}_{(2,2)}^{(n,m)}$ and $\widetilde{\Pos}_{(2,2)}^{(n,m)}$ in $\cM$.

\begin{remark}\mbox{}\par
\begin{enumerate}[label={\rm(\arabic*)}]
\item 
Szarek, Werner, and \. Zyczkowski in \cite{SWZ08}
provide bounds similar to those in Corollary \ref{verjetnost-intro} in the case 
of complex matrix algebras with $n=m$. 
However, their
normalization is different from ours.
We normalize using $\Tr(\Phi(I_n))=nm$ 
(see Proposition \ref{hyperplane-equiv}),
whereas in \cite{SWZ08} the compact
cross-section is obtained by fixing 
$\Tr(\Phi(I_n))=n$.
	\item
We note that the  normalized probability
${p_{n,m}}^{\frac1{D_\cM}}$
(as in Corollary \ref{verjetnost-intro}) 
does not go to 0 if $\min(m,n)$ is bounded and $\max(m,n)\to\infty$.
\end{enumerate}
\end{remark}

Section \ref{pos-maps-biforms}
converts the positive--cp gap problem
into the language of real algebraic geometry \cite{BCR98}. To each linear map $\Phi:\Sym_n\to\Sym_m$ we associate the biquadratic biform $p_\Phi\in\RR[\x,\y]$, $p_\Phi=\y^\ast\Phi(\x\x^\ast)\y$. Then $\Phi$ is positive if and only if $p_\Phi$
is nonnegative on $\RR^{n+m}$, and $\Phi$ is cp if and only if
$p_\Phi$ is a sum of squares of polynomials, see Proposition \ref{map-poly} below.
Therefore positive maps that are not cp correspond exactly to
nonnegative biquadratic biforms that
are not sums of squares biforms.
We note that a different connection between (completely) positive
maps and real algebraic geometry was introduced and investigated in
\cite{HKM1,HKM+}.

The third contribution of the paper is the construction
(from random input data)
of positive maps
$\Phi:\Sym_n\to\Sym_m$
($n,m\geq3$)
that are not completely positive, see
Section \ref{sec:algo}. Again, by Proposition \ref{map-poly}, it suffices to construct  nonnegative
biquadratic biforms that are not sums of squares biforms.
This construction is done in Algorithm \ref{algo}
by specializing the \cite{BSV} algorithm to our context.
Algorithm \ref{algo} depends on semidefinite programming \cite{WSV00}, so produces a floating point output. We discuss implementation and rationalization, i.e.,
producing exact output, in Subsection \ref{ssec:rat}.

\subsubsection{Positive but not completely positive maps on full matrix algebras $M_n(\FF)$}

The counterpart of Corollary \ref{verjetnost-intro} that gives the upper bound for the probability $p_{n,m}^\FF$ that a random positive map $\Phi:M_n(\FF)\to M_m(\FF)$ is completely positive is the following.

\begin{theorem}\label{asimptotika-F-intro}
	For integers $\displaystyle n,m\geq 3$, the probability $p_{n,m}^\FF$
	that a random positive map
	$\Phi:M_n(\FF)\to M_m(\FF)$ is completely
	positive, is bounded above by
		$$p_{n,m}^\FF<
			\left( 
					\frac{C}{\min(n,m)-\frac{1}{2}}  \right)^{\frac{D_{\cM_{\cC_\FF}}}{2}},$$
	where $C:=\frac{\left(2^{28-\dim_{\RR}\FF} \right) \cdot 5^4 \cdot 10^{\frac{4}{9}}}{3^{5}}$ and
		$D_{\cM_{\cC_{\FF}}}=\left\{  
		\begin{array}{lc} 
			n^2m^2-1,&\quad \text{if }\FF=\CC,\\
			\frac{nm(nm+1)}{2}-1,&\quad  \text{if }\FF=\RR.
		  \end{array} \right.$
\end{theorem}

Theorem \ref{asimptotika-F-intro} is established in Subsection \ref{ext-to-real-or-complex} as a corollary of the extensions (in Subsection \ref{ext-to-sym-multiforms}) 
of the special case $k_1=k_2=1$ of Theorems \ref{psd-intro} and \ref{squares-intro} from real biforms to
\textbf{symmetric multiforms of 
multidegree $(1,1,1,1)$}, i.e., 
	$$\symm\FF[\z,\overline{\z},\w,\overline{\w}]_{1,1,1,1}:=
		\left\{\sum_{i,j=1}^n\sum_{k,\ell=1}^m a_{ijk\ell} \overline{z_i} z_j \overline{w_k} w_\ell\colon a_{ijk\ell}\in \FF,\ 
		a_{ijk\ell}=\overline{a_{ji\ell k}} \text{ for all }i,j,k,\ell\right\}.$$

By extending a positive map $\Phi:\Sym_n\to \Sym_m$ that is not completely positive with a linear map $\Psi:\mathbb K_n\to \mathbb K_m$ where 
$\mathbb{K}_n$ stands for the vector space $\left\{A\in M_n(\RR) \mid A^\ast=-A\right\}$ of real antisymmetric $n\times n$ matrices, 
one obtains a positive map $\Gamma:=\Phi\oplus \Psi:M_n(\RR)\to M_m(\RR)$ that is not completely positive. 
The complexification $(\Phi\oplus \mathbf 0)^\CC:M_n(\CC)\to M_m(\CC)$ where $\mathbf 0:\mathbb K_n\to \mathbb K_m$ stands for the trivial map, i.e., $\ker \mathbf 0=\mathbb K_n$, is a positive map that is not completely positive. 
Thus, the algorithm from Section \ref{sec:algo} can also be used to produce positive maps $\Phi:M_n(\FF)\to M_m(\FF)$ that are not completely positive.

\subsubsection{Comparison with the original work of Blekherman}
In \cite{Blek1}
Blekherman established
  estimates on the volumes of compact sections of the cones of nonnegative forms and sums of squares forms. If the degree is fixed and the number of variables goes to infinity the ratio between the volumes goes to 0. We restrict ourselves to special subcones of these cones, i.e., the cones of nonnegative biforms and sums of squares biforms.
  It is not clear how to directly apply the estimates from \cite{Blek1} to this special case.
In fact,  \cite{BR} gives the example
of symmetric nonnegative forms vs sums of squares,
where the ratio between the corresponding
volumes behaves differently, i.e., does not tend to $0$.
Regarding biforms as tensor products of forms we establish
 estimates for biforms following the techniques of \cite{Blek1}.
In \cite{BSV} there is an explicit construction of nonnegative quadratic forms on special projective varieties that are not sums of squares forms. We specialize their construction to the context of biquadratic biforms to produce nonnegative biforms that are not sums of squares biforms.

Recently, Erg\"ur posted the preprint \cite{Erg+} on arXiv. There he extends some of Blekherman's volume estimates to biforms; like our results in Section \ref{sec:estimate} his results readily generalize to multiforms.
While there is certain overlap with our results, we explicitly compute all constants appearing in the
estimates. Furthermore, some of our estimates are strictly better than the ones of \cite{Erg+}; cf.\  Theorem \ref{lower-bound} and \cite[Section 3]{Erg+}.

\subsection*{Acknowledgments}
The authors thank Greg Blekherman for many
inspiring discussions and for bringing the preprint \cite{Erg+} to their attention. 
Thanks to Benoit Collins for helpful suggestions. 
We also thank the anonymous referees whose useful comments and interesting questions led to marked improvement in the manuscript.

\section{Blekherman type estimates for biforms}\label{sec:estimate}

	In this section we extend the estimates on the volumes of compact sections of the cones of nonnegative forms and sums of squares forms established in \cite{Blek1} to biforms.
	Our proofs  borrow heavily from \cite{Blek1} and
	to a lesser extent  from \cite{Barv-Blek}.
	For clarity and completeness of exposition
	we give proofs with all the details, even if some
of the reasoning repeats arguments from
	\cite{Blek1}. 

	At various places we will regard the vector space $\RR[\x,\y]_{2k_1,2k_2}$ of biforms of bidegree $(2k_1,2k_2)$ as a module over the product $\SO(n)\times \SO(m)$
	of special orthogonal groups with the action given by rotating the coordinates, i.e.,
	for $(A,B)\in \SO(n)\times \SO(m)$ and $f\in \RR[\x,\y]_{2k_1,2k_2}$	 we define
		\begin{equation} \label{rotation}
			(A,B)\cdot f(\x,\y)=f(A^{-1}\x,B^{-1}\y).
		\end{equation}
\noindent Note that the cones $\Pos^{(n,m)}_{(2k_1,2k_2)}$ and $\Sq^{(n,m)}_{(2k_1,2k_2)}$, 
and the sections $\widetilde{\Pos}^{(n,m)}_{(2k_1,2k_2)}$ and $\widetilde{\Sq}^{(n,m)}_{(2k_1,2k_2)}$
are invariant under this action.

\subsection{Nonnegative biforms }

In this subsection we establish bounds for the volume of the section of nonnegative biforms. The main result is the following.

\begin{theorem}\label{lower-bound}
	For $n,m\in \NN$ we have:
		$$c_{2k_1,2k_2}\leq \left(\frac{\Vol \widetilde{\Pos}^{(n,m)}_{(2k_1,2k_2)}}{\Vol B_{\cM}}\right)^{\frac{1}{D_{\cM}}}\leq
			2 \left(\min\left(\frac{2k_1^2}{2k_1^2+n},  \frac{2k_2^2}{2k_2^2+m}\right)\right)^{\frac{1}{2}},$$
	where
		$$c_{2k_1,2k_2}=\ds
			\left\{\begin{array}{lr}
				3^3\cdot 10^{-\frac{20}{9}}\max(n,m)^{-\frac{1}{2}},& \text{if } k_1=k_2=1\\[1mm]
				\exp(-3) \left(2\lceil\max(m,n)\ln(2\max(k_1,k_2)+1)\rceil\right)^{-\frac{1}{2}} ,& \text{otherwise.}
			\end{array}\right.$$
\end{theorem}

The proof of Theorem \ref{lower-bound} occupies the next two subsections. It is inspired by Blekherman's proof of \cite[Theorem 4.1]{Blek1}.

Let $V$ be a real vector space. Recall that, for a convex body $\cK$ with the origin in its interior,
the \textbf{gauge} $G_{\cK}$ is defined by
	$$G_{\cK}:V\to\RR,\quad
		G_{\cK}(p)=\inf\left\{\lambda>0\colon p\in \lambda\cdot \cK\right\}.$$

\begin{lemma}\label{binom-lem}
	Let $p,q\in \NN$ be natural numbers such that $p>q.$ For every natural number $n\in \NN$ we have
		$$\binom{pn}{qn}^{\frac{1}{qn}}<\frac{p}{q}\left(\frac{p}{p-q}\right)^{\frac{p-q}{q}}.$$
\end{lemma}


\begin{proof}
	Using Stirling's approximation \cite[inequality (9.14)]{Fel68}
		$$\sqrt{2\pi}\cdot n^{n+\frac{1}{2}}\cdot \exp\left(-n+\frac{1}{12n+1}\right)<n!<\sqrt{2\pi}\cdot n^{n+\frac{1}{2}}\cdot \exp\left(-n+\frac{1}{12n}\right)$$
	in $\binom{pn}{qn}$, we obtain 
	\begin{equation}\label{stirling1}
		\binom{pn}{qn}
		=\frac{(pn)!}{(qn)!((p-q)n)!}<
		\left(\frac{p}{2\pi q(p-q)n}\right)^{\frac{1}{2}}\cdot \exp\left(f(p,q,n)\right)
				\cdot \left(\frac{p^{p}}{q^q(p-q)^{p-q}}\right)^n,
	\end{equation}
	where $$f(p,q,n):=\frac{1}{12pn}-\frac{1}{12qn+1}-\frac{1}{12(p-q)n+1}.$$
	
	\vspace{0.5cm}
	\noindent \textbf{Claim:} Let $p,q,n\in \NN$ be natural numbers with $p>q$. Then 
		$$\exp\left(f(p,q,n)\right)<1\quad \text{and}\quad \frac{p}{2\pi q(p-q)n}<1.$$	
	 
	Note that $f(p,q,n)<0$ and hence $\exp\left(f(p,q,n)\right)<1$.
	To prove the other inequality in the claim first notice that it suffices to assume that $n=1$ and then we have that
		\begin{equation}\label{stirling}
			\frac{p}{2\pi q(p-q)}<1 \quad \Leftrightarrow \quad p<2\pi q(p-q) \quad \Leftrightarrow \quad 2\pi q^2 < p(2\pi q -1).
		\end{equation}
	Now since $q+1\leq p$ it follows that 
		$$2\pi q^2<2\pi q^2 +2\pi q-(q+1)=(q+1)(2\pi q-1)\leq p(2\pi q-1).$$
	Using this together with the equivalences (\ref{stirling}) concludes the proof of the claim.\\
	
	Using the Claim in the inequality (\ref{stirling1}) it follows that
		$$\binom{pn}{qn}< \left(\frac{p^{p}}{q^q(p-q)^{p-q}}\right)^n=\left(\frac{p}{q}\right)^{qn} \left(\frac{p}{p-q}\right)^{(p-q)n},$$
	which proves the lemma.
\end{proof}

\subsubsection{Proof of the lower bound in Theorem \ref{lower-bound}}
	We denote $\cK=\widetilde{\Pos}^{(n,m)}_{(2k_1,2k_2)}$. 
	Note that $\cK$ is a convex body in $\cM$
	with origin in its interior
	and the boundary of $\cK$
	consists of biforms with minimum $-1$ on $T$.
	Indeed, it is easy to see that $\cK$ consists exactly of biforms from $\cM$
	with minimum at least $-1$ on $T$ and that every biform with minimum $-1$ on $T$
	belongs to its boundary. However, if $f\in \cK$ satisfies
	 $m_f:=\min_{(x,y)\in T} f(x,y)>-1$, then the ball
		$$B(f,m_f+1):=\left\{ g\in \RR[\x,\y]_{2k_1,2k_2} \colon
			\|f-g\|_\infty<m_f+1\right\}$$
	also belongs to $\cK$ and hence $f$ belongs to the interior of $\cK.$
	Therefore the gauge
	$G_{\cK}:\cM\to\RR$ of $\cK$ in $\cM$
	 is given by
		$$G_{\cK}(f)=|\min_{v\in T} f(v)|\quad\text{for }f\in \cM.$$
	Let $\widetilde\mu$ be the rotation invariant probability measure on $S_{\cM}$.
	By \cite[p.\ 91]{14}, 
		\begin{equation*} \label{expression-gauge}
			\left( \frac{\Vol \cK}{\Vol B_{\cM}} \right)^{\frac{1}{D_{\cM}}}=
			\left( \int_{S_{\cM}}G_{\cK}^{-D_{\cM}} \dd\widetilde\mu   \right)^{\frac{1}{D_{\cM}}}.
		\end{equation*}
	\noindent By H\"older's inequality we have
			$$\left( \int_{S_{\cM}}G_{\cK}^{-D_{\cM}} \dd\widetilde\mu   \right)^{\frac{1}{D_{\cM}}}\geq
				\int_{S_\cM} G_{\cK}^{-1}\dd\widetilde\mu,$$
	and so
		$$\left(\frac{\Vol \cK}{\Vol B_\cM}\right)^{\frac{1}{D_\cM}}\geq
			\int_{S_\cM} G_{\cK}^{-1}\dd\widetilde\mu.$$
	By Jensen's inequality
	(applied to the convex function $y=\frac{1}{x}$ on $\RR_{>0}$),
		$$\int_{S_\cM} G_{\cK}^{-1}\dd\widetilde\mu \geq \left(\int_{S_\cM} G_{\cK}\dd\widetilde\mu\right)^{-1}.$$
	Since $\left\|f\right\|_\infty=\max_{v\in T} |f(v)|\geq |\min_{v\in T} f(v)|$, it follows that
		\begin{equation*}\label{inequality}
			\left( \frac{\Vol \cK}{\Vol B_{\cM}} \right)^{\frac{1}{D_{\cM}}}\geq
				\left(\int_{S_{\cM}} \left\|f\right\|_\infty \dd\widetilde\mu\right)^{-1}.
		\end{equation*}
	The proof of the lower bound in Theorem \ref{lower-bound} now reduces to proving the following claim.
	\\

	\noindent \textbf{Claim 1:} $\ds\int_{S_{\cM}}\left\|f\right\|_\infty \dd\widetilde\mu \leq \frac{1}{c_{2k_1,2k_2}}$.\\

	To prove this claim we will use \cite[Corollary 2]{Barv1}.
	Write $G=\SO(n)\times \SO(m)$ and consider
	the tensor product
	 $(\RR^n)^{\otimes 2k_1}\otimes (\RR^m)^{\otimes 2k_2}$.
	Let  $e_1\in \RR^n$, $f_1\in \RR^m$ be standard unit vectors
	and let $w$ be the tensor
		$$w:=(e_1)^{\otimes 2k_1}\otimes (f_1)^{\otimes 2k_2}\in
			(\RR^n)^{\otimes 2k_1}\otimes (\RR^m)^{\otimes 2k_2}.$$
	We also define
		$$v:=w-q,\quad\text{where }q=\int_{(g,h)\in G} (g,h)w\; \dd(g,h),$$
	and we integrate w.r.t.\ the Haar measure on $G$.
	Similarly as in \cite[Example 1.2]{Barv-Blek} we proceed as follows:
	\begin{enumerate}[label={\rm(\arabic*)}]
		\item We identify the vector space of biforms from $\RR[\x,\y]_{2k_1,2k_2}$
			with the vector space $V_1$ of the restrictions of linear functionals
				$\ell:(\RR^n)^{\otimes 2k_1}\otimes (\RR^m)^{\otimes 2k_2}\to \RR$
			to the orbit
				 $\left\{(g,h)w\colon (g,h)\in G\right\}.$
		\item We identify the vector space of biforms from $\cM$
			with the vector space $V_2$ of the restrictions of linear functionals
				$\ell:(\RR^n)^{\otimes 2k_1}\otimes (\RR^m)^{\otimes 2k_2}\to \RR$
			to
				$$B=\conv((g,h) v\colon (g,h)\in G).$$
		\item We introduce an inner product on
			$V_2$
			by defining
				$$\langle \ell_1,\ell_2\rangle:=\int_G \ell_1((g,h) v)\cdot \ell_2((g,h) v)\;\dd(g,h).$$
			This inner product also induces the dual inner product on $V_2^\ast\cong V_2$ which we
			also denote by $\langle\cdot,\cdot\rangle$.
	\end{enumerate}
	By  \cite[Corollary 2]{Barv1},
		$$\left\|f\right\|_\infty\leq (D_k)^{\frac{1}{2k}} \cdot \left\|f\right\|_{2k},$$
	where
		$D_k=\dim\Span\{(g,h) w^{\otimes k}\colon (g,h)\in G\}.$ Clearly,
		\[\begin{split}
			D_k &= \dim\Span\{g e_1^{\otimes 2k_1k}\colon g\in \SO(n)\}\cdot \dim\Span\{h f_1^{\otimes 2k_2k}\colon h\in \SO(m)\}\\
				&= \binom{2k_1 k+n-1}{2k_1 k} \binom{2k_2k+m-1}{2k_2 k},
				\end{split}
\]
	where the second equality follows as in \cite[p.\ 404]{Barv1}. 
		\begin{equation} \label{eq:D_k}
		D_k=\binom{2k_1 k+n-1}{2k_1 k} \binom{2k_2k+m-1}{2k_2 k}\leq \binom{2\max(k_1,k_2) k+\max(n,m)-1}{2\max(k_1,k_2) k}^2.
		\end{equation}
	We now distinguish two cases. \\

	\noindent \textbf{Case 1:} $k_1=k_2=1.$\\

	If $\max(n,m)$ is odd, we let $2k_0=9(\max(n,m)-1)$. Otherwise take
	$2k_0=9\max(n,m)$ to get
		$$D_{k_0}^{\frac{1}{2k_0}}\leq \binom{\frac{20}{9}k_0}{2k_0}^{\frac{1}{k_0}}.$$
	Since $2k_0=9\ell_0$ for some $\ell_0\in \NN$ we get 
		$$D_{k_0}^{\frac{1}{2k_0}}\leq \binom{10\ell_0}{9\ell_0}^{\frac{2}{9\ell_0}}\leq \left(\frac{10}{9}\cdot 10^{\frac{1}{9}}\right)^2,$$
	where we used Lemma \ref{binom-lem} in the last inequality. \\

	\noindent \textbf{Case 2:} $k_1>1$ or $k_2>1.$\\

	\noindent \textbf{Claim 2:} For $k_0\geq \lceil \max(m,n)\ln(2\max(k_1,k_2)+1)\rceil$,
	  $$D_{k_0}^{\frac{1}{2k_0}}\leq \exp(3).$$

	We define the function
		$$H(x)=-x\ln(x)-(1-x)\ln(1-x)\quad\text{for }x\in (0,1).$$
	For $\lambda\in \left(0,\frac{1}{2}\right]$ we have the estimate
		\begin{eqnarray*}
		1&=&(\lambda+(1-\lambda))^n=
			\sum_{i=0}^n \binom{n}{i} \lambda^i (1-\lambda)^{n-i}=
			\sum_{i=0}^n \binom{n}{i} (1-\lambda)^n
				\left(\frac{\lambda}{1-\lambda}\right)^i\\
		 &>& \sum_{i=0}^{\lfloor\lambda n\rfloor} \binom{n}{i} (1-\lambda)^n
				\left(\frac{\lambda}{1-\lambda}\right)^{\lambda n}=
			\sum_{i=0}^{\lfloor \lambda n\rfloor} \binom{n}{i} \exp(-nH(\lambda)),
		\end{eqnarray*}
	where we used that $\frac{\lambda}{1-\lambda}\leq 1$ for
	$\lambda \in (0,\frac{1}{2}]$ in the inequality.
	It follows that
		$$\binom{a}{b} \leq \exp\left(aH\left(\frac{b}{a}\right)\right) \quad \text{for }
			a,b\in \NN\text{ and }
			b\leq \lfloor\frac{a}{2}\rfloor.$$
	Since $\binom{a}{b}=\binom{a}{a-b}$ and
	$H\left(\frac{b}{a}\right)=H\left(\frac{a-b}{a}\right)$, we conclude
		\begin{equation}\label{binom-ineq}
			\binom{a}{b} \leq \exp\left(aH\left(\frac{b}{a}\right)\right) \quad \text{for }
			a,b\in \NN\text{ and }
			b\leq a.
		\end{equation}
	Writing $C_1=\max(k_1,k_2)$, $C_2=\max(m,n)$ and using (\ref{eq:D_k}), (\ref{binom-ineq})
	we get
		\begin{eqnarray*}
			D_k^{\frac{1}{2k}}
			&\leq& \left(\exp\left(\left(2C_1k+C_2-1\right) H\left(\frac{2C_1k}{2C_1k+C_2-1}\right)\right)\right)^{\frac{1}{k}}\\
			&=& \exp\left(2C_1\ln\left(1+\frac{C_2-1}{2C_1k}\right)+\frac{C_2-1}{k}
				\ln\left(1+\frac{2C_1k}{C_2-1}\right)\right)\\
			&=& \exp\left(2C_1\ln\left(1+\frac{C_2-1}{2C_1k}\right)\right)
				\cdot\exp\left(\frac{C_2-1}{k}
				\ln\left(1+\frac{2C_1k}{C_2-1}\right)\right)\\
			&\leq& \exp\left(\frac{C_2-1}{k}\right)
				\cdot\exp\left(\frac{C_2-1}{k}
				\ln\left(1+\frac{2C_1k}{C_2-1}\right)\right),
		\end{eqnarray*}
	where we used $\ln(1+x)\leq x$ for $x>-1$ in the second inequality.
	Let as assume that
		$$k_0\geq \lceil C_2\ln(2C_1+1)\rceil.$$
	Then
		$$\exp\left(\frac{C_2-1}{k_0}\right)< \exp(1),$$
	since $k_0\geq C_2.$
	To prove Claim 2 it remains to establish
		\begin{equation}\label{exp-2}
		\exp\left(\frac{C_2-1}{k_0}
		\ln\left(1+\frac{2C_1k_0}{C_2-1}\right)\right)\leq \exp(2).
		\end{equation}
	Notice that (\ref{exp-2}) holds if and only if
		$$\ln\left(1+\frac{2C_1k_0}{C_2-1}\right)\leq \frac{2k_0}{C_2-1}.$$
	Now
		$$\ln\left(1+\frac{2C_1k_0}{C_2-1}\right)\underbrace{\leq}_{k_0\geq C_2}
		\ln\left(\frac{(2C_1+1)k_0}{C_2-1}\right).
		$$
	Thus it suffices to prove that
		$$\ln\left(\frac{(2C_1+1)k_0}{C_2-1}\right)\leq \frac{2k_0}{C_2-1},$$
	or equivalently
		$$
		(C_2-1)\left(\ln\left(2C_1+1\right)+\ln\left(\frac{k_0}{C_2-1}\right)\right)\leq 2k_0.
		$$
	Using $\ln(x)\leq x-1<x$ for $x>0$ we estimate the left hand side from above by
		$$(C_2-1)\ln\left(2C_1+1\right)+k_0$$
	and since
		$$(C_2-1)\ln\left(2C_1+1\right)+k_0\leq 2k_0$$
	if and only if
		$$(C_2-1)\ln\left(2C_1+1\right)\leq k_0,$$
	(\ref{exp-2}) holds.
	Hence Claim 2 follows.\\

	To prove Claim 1 it remains to estimate the average $L^{2k_0}$ norm, i.e.,
		\begin{equation}\label{def-A}
			A=\int_{S_{\cM}} \left\| f\right\|_{2k_0}\; \dd\widetilde\mu = \int_{S_{\cM}} \left(\int_{T}
				f^{2k_0}\; \dd\sigma\right)^{\frac{1}{2k_0}} \dd\widetilde\mu .
		\end{equation}
	Notice that
		\begin{equation}\label{identifikacija}
			\int_{S_{\cM}} \left(\int_{T}
				f^{2k_0}\; \dd\sigma \right)^{\frac{1}{2k_0}}\dd\widetilde\mu=
			\int_{S_{V_2}} \left(\int_G \langle c,(g,h)v\rangle^{2k_0}\dd(g,h)\right)^{\frac{1}{2k_0}}\dd c,
		\end{equation}
	where $S_{V_2}$ is the unit sphere in $V_2$ endowed with the rotation invariant probability measure $c$.
	Combining \eqref{def-A},
 \eqref{identifikacija} we obtain
		$$A=
			\int_{S_{V_2}} \left(\int_G \langle c,(g,h)v\rangle^{2k_0}\dd(g,h)\right)^{\frac{1}{2k_0}}
			\dd c\leq
			\sqrt{\frac{2k_0 \langle v,v\rangle}{D_\cM}}=\sqrt{2k_0},$$
	where we used \cite[Lemma 3.5]{Barv-Blek} for the inequality and
	\cite[Remark p.\ 62]{Barv-Blek} for the last equality.
	This equality proves Claim 1 and establishes the lower bound in Theorem \ref{lower-bound}.
	\hfill\qedsymbol

\subsubsection{Proof of the upper bound in Theorem \ref{lower-bound}}

Before proving the upper bound in Theorem \ref{lower-bound} we introduce the gradient inner products needed in the proof.
Let $f\in  \RR[\x,\y]_{2k_1,2k_2}$ be a biform. For every fixed $y\in \RR^m$ we define a form $f^{y}\in \RR[\x]_{2k_1}$ by
		$$f^{y}(\x):=f(\x,y).$$
	Recall \cite[p.\ 367]{Blek1} that the \textbf{gradient inner product} on
	$\RR[\x]_{2k_1}$ is defined by
		 $$\left\langle h_1, h_2 \right\rangle_{\gr}=
		 	\frac{1}{4k_1^2}\int_{S^{n-1}}
			\left\langle \nabla  h_1 , \nabla h_2 \right\rangle \dd\sigma_1\quad
			\text{for }h_1,h_2\in \RR[\x]_{2k_1},$$
	where
		$$\nabla  h=(\frac{\partial h}{\partial x_1},\ldots, \frac{\partial h}{\partial x_n})\;\text{for }
			h\in \RR[\x]_{2k_1}\quad\text{and}\quad
		\left\langle \nabla  h_1, \nabla h_2\right\rangle=
		\sum_{i=1}^n \frac{\partial h_1}{\partial x_i}\frac{\partial h_2}{\partial x_i}.$$
	We define the \textbf{$\x$-gradient inner product} on $\RR[\x,\y]_{2k_1,2k_2}$ by
		$$\left\langle f, g \right\rangle_{\gr_{\x}}=\int_{S^{m-1}} \left\langle f^{y}, g^{y} \right\rangle_{\gr} \dd\sigma_2(y).$$
	Note that positive definiteness follows by noticing that if $f\in \RR[\x,\y]_{2k_1,2k_2}$
	is a nonzero biform, then there exists $(x,y)\in S^{n-1}\times S^{m-1}$ such that
	$f(x,y)\neq 0$. By continuity it follows that $f^{y_0}$ is nonzero for every $y_0$ in some neighbourhood of $y$. Thus $\left\langle f^{y_0}, f^{y_0} \right\rangle_{\gr}>0$ for
	every $y_0$ in some neighbourhood of $y$. Hence 
	$\left\langle f, f \right\rangle_{\gr_{\x}}>0$.
	
	Let $\left\| f \right\|_{\gr_{\x}}$ be the \textbf{$\x$-gradient norm} of $f$ and
	let $B_{\gr_{\x}}$ be the unit ball in the $\x$-gradient norm.

\begin{proof}[Proof of the upper bound in Theorem \ref{lower-bound}]
	Let $\widetilde{\Pos^{\circ}}$ denote the polar dual of the section $\widetilde{\Pos}^{(n,m)}_{(2k_1,2k_2)}$ in
	$\cM$,
		$$\widetilde{\Pos^{\circ}}=\left\{ f\in \cM\colon \left\langle f, g \right\rangle \leq 1\quad
			\text{for all } g\in \widetilde{\Pos}^{(n,m)}_{(2k_1,2k_2)} \right\}.$$
	By the Blaschke-Santal\'o inequality \cite{MP} applied to
	$\widetilde{\Pos}^{(n,m)}_{(k_1,k_2)}$
 	we get that
		\begin{equation} \label{Blaschke}
			\Vol \left(\widetilde{\Pos}^{(n,m)}_{(2k_1,2k_2)} \right)\Vol \left(\widetilde{\Pos^{\circ}} \right)\leq \left(\Vol B_{\cM} \right)^2.
		\end{equation}
	(Note that for the validity of (\ref{Blaschke}) we used the fact that
	the origin is the Santal\'o point of $\widetilde{\Pos}^{(n,m)}_{(2k_1,2k_2)}$. This fact follows
	by observing that the origin is the unique point in the convex body $\widetilde{\Pos}^{(n,m)}_{(2k_1,2k_2)}$ fixed by the action of $\SO(n)\times \SO(m)$
	and that $\widetilde{\Pos}^{(n,m)}_{(2k_1,2k_2)}$ is also
	invariant under the action of $\SO(n)\times \SO(m)$.)
		Hence it suffices to prove that
		\begin{equation} \label{estimate-10}
			\left(\frac{\Vol \widetilde{\Pos^{\circ}} }{\Vol B_{\cM}}\right)^{\frac{1}{D_{\cM}}} \geq
				\frac{1}{2} \left(
				\max\left(\frac{2k_1^2+n}{2k_1^2},  \frac{2k_2^2+m}{2k_2^2}\right)  \right)^{\frac{1}{2}}.
		\end{equation}
	Let $B_{\infty}$ be the unit ball of the supremum norm in $\cM$.
	We notice that
		$$B_{\infty}= \widetilde{\Pos}^{(n,m)}_{(2k_1,2k_2)}  \bigcap
			-\widetilde{\Pos}^{(n,m)}_{(2k_1,2k_2)},$$
	and by taking polar duals we get
		$$B_\infty^\circ=\text{ConvexHull}\{\widetilde{\Pos^\circ},-\widetilde{\Pos^\circ}\}$$
	By a theorem of Rogers and Shephard
	\cite[Theorem 3]{RS2}, it follows that
		$$\Vol B_\infty^\circ\leq 2^{D_{\cM}}\Vol \widetilde{\Pos^\circ}.$$
	Thus
		\begin{equation}\label{estimate-11}
			\left(\frac{\Vol \widetilde{\Pos^{\circ}} }{\Vol B_{\infty}^\circ }\right)^{\frac{1}{D_{\cM}}} \geq \frac{1}{2}.
		\end{equation}
	Using (\ref{estimate-10}) and (\ref{estimate-11}) the proof of the upper bound in 
	Theorem \ref{lower-bound} reduces to establishing
		\begin{equation}\label{new-estimate}
			\left(\frac{ \Vol B_{\infty}^\circ }{\Vol B_{\cM}} \right)^{\frac{1}{D_{\cM}}}\geq \left( \max\left(\frac{2k_1^2+n}{2k_1^2},  \frac{2k_2^2+m}{2k_2^2}\right)   \right)^{\frac{1}{2}}.
		\end{equation}

	\noindent \textbf{Claim:} $\ds\left(\frac{ \Vol B_{\infty}^\circ }{\Vol B_{\cM}} \right)^{\frac{1}{D_{M}}}\geq \left( \frac{2k_1^2+n}{2k_1^2} \right)^{\frac{1}{2}}.$\\

	We estimate
		\begin{equation} \label{Kellog-est}
			\left\| f\right\|_{\infty}=
			\max_{y\in S^{m-1}} \left(\max_{x\in S^{n-1}} \left|f^{y}(x) \right|\right)
			\geq  \max_{y\in S^{m-1}} \left\|f^{y}(\x) \right\|_{\gr} \geq  \left\| f\right\|_{\gr_{\x}},
		\end{equation}
	where the first inequality follows by \cite[Theorem IV]{Kellog},
		$$\left\| \left\langle \nabla  f^{y}, \nabla f^{y}\right\rangle \right\|_{\infty}\leq
			4k_1^2 \left\| f^{y}\right\|^2_\infty.$$
	Using (\ref{Kellog-est}) we get the inclusion
		\begin{equation} \label{ball-inclusion}
			B_{\infty} \subseteq B_{\gr_{\x}} \quad \text{and hence}\quad
			B_{\gr_{\x}}^\circ\subseteq B_{\infty}^\circ,
		\end{equation}
	where $B_{\infty}^\circ$ and  $B_{\gr_{\x}}^\circ$ are the polar duals of  $B_{\infty}$
	and  $B_{\gr_{\x}}$, respectively.
	Since $B_{\gr_{\x}}$ is an ellipsoid (the $\x$-gradient norm is induced from an inner product),
	we deduce
		$$\Vol B_{\gr_{\x}}^\circ=\frac{(\Vol B_\cM)^2}{\Vol B_{\gr_{\x}}}$$
	and hence by (\ref{ball-inclusion}),
		$$\Vol B_{\infty}^\circ\geq \frac{(\Vol B_\cM)^2}{\Vol B_{\gr_{\x}}}.$$
	Therefore the proof of the Claim reduces to showing that
		$$  \left\langle f, f \right\rangle_{\gr_{\x}} \geq \frac{2k_1^2+n}{2k_1^2}  \left\langle f, f \right\rangle. $$
	We estimate
		\begin{eqnarray*}
			\left\langle f, f \right\rangle_{\gr_{\x}}
			&=& \int_{S^{m-1}} \left\langle f^{y}, f^{y} \right\rangle_{\gr} \dd\sigma_2(y)\\
			&\geq& \frac{2k_1^2+n}{2k_1^2} \int_{S^{m-1}} \left\langle f^{y}, f^{y} \right\rangle \dd\sigma_2(y)=
				\frac{2k_1^2+n}{2k_1^2} \left\langle f, f \right\rangle,
		\end{eqnarray*}
	where the inequality follows by \cite[(4.3.1)]{Blek1}.
	This proves the Claim.\\

	 By interchanging the roles of $\x$ and $\y$ in the Claim we also obtain the inequality
	 	$$\ds\left(\frac{ \Vol B_{\infty}^\circ }{\Vol B_{\cM}} \right)^{\frac{1}{D_{M}}}\geq
			 \left( \frac{2k_2^2+m}{2k_2^2} \right)^{\frac{1}{2}},$$
	which proves (\ref{new-estimate}) and
	concludes the proof of the upper bound in Theorem \ref{lower-bound}.
\end{proof}

\subsection{The apolar inner product on $\RR[\x,\y]_{2k_1,2k_2}$}

\label{dif-inn-prod} 

Before tackling the bounds for the volume of the section of sum of squares biforms we have to extend some of the results
on the apolar inner product given in \cite[\S 5]{Blek1}.

For technical reasons we identify $\RR[\x,\y]_{2k_1,2k_2}$ with $\RR[\x]_{2k_1}\otimes \RR[\y]_{2k_2}$  in the natural way.
Recall from \cite[p.\ 11]{Rez92}
that for a form
	$$r=\sum_{\alpha=(i_1,\ldots,i_n)} c_\alpha x_1^{i_1}\cdots x_{n}^{i_n}\in \RR[\x]_{2k_1},$$
the associated differential operator $D^{\x}_r$ is defined by
	$$D^{\x}_r= \sum_{\alpha=(i_1,\ldots,i_n)} c_\alpha \frac{\partial^{i_1}}{\partial{x_1^{i_1}} }\cdots \frac{\partial^{i_n}}{\partial{x_{n}^{i_n}}}.$$
The operator $D^{\x}_r$ induces the inner product on $\RR[\x]_{2k_1}$, called
the \textbf{$\x$-apolar inner product}, defined by
	$$\left\langle r,s\right\rangle_{d_{\x}}=D_r^{\x}(s)\quad \text{for }s\in \RR[\x]_{2k_1}.$$
Note that positive definiteness follows from 
$D_r^{\x}(r)=\sum_{\alpha} c_\alpha^2 \cdot i_1! \cdots i_n!.$
Analogously we define the differential operator $D^{\y}_t$ for a form $t\in \RR[\y]_{2k_2}$ and the \textbf{$\y$-apolar inner product}
$\left\langle \cdot,\cdot\right\rangle_{d_{\y}}$ on $\RR[\y]_{2k_2}$.

	To every form $f\in \RR[\x]_{2k_1}\otimes \RR[\y]_{2k_2}$,
	\begin{eqnarray*}
		f
		&=& \sum_\ell f_{\ell 1}(\x)\otimes f_{\ell 2}(\y)\\
		&=& \sum_\ell\left(
			\sum_{\alpha=(i_1,\ldots,i_n)} c_\alpha^{(\ell)} x_1^{i_1}\cdots x_{n}^{i_n} \otimes
			\sum_{\beta=(j_1,\ldots,j_m)} d_\beta^{(\ell)} y_1^{j_1}\cdots y_{m}^{j_m}\right),
	\end{eqnarray*}
we associate the differential operator $D_f$ by
	\begin{eqnarray*}
		D_f	&=&
			\sum_\ell D^{\x}_{f_{\ell 1}}\otimes D^{\y}_{f_{\ell 2}}\\
			&=&
			\sum_\ell\left(
			\sum_{\alpha=(i_1,\ldots,i_n)}
				c_\alpha^{(\ell)} \frac{\partial^{i_1}}{\partial x_1^{i_1}}\cdots
					\frac{\partial^{i_n}}{\partial x_n^{i_n}} \otimes
			\sum_{\beta=(j_1,\ldots,j_m)} d_\beta^{(\ell)} \frac{\partial^{j_1}}{\partial y_1^{j_1}}
				\cdots  \frac{\partial^{j_m}}{\partial y_m^{j_m}}\right),
	\end{eqnarray*}
and the corresponding inner product, called the \textbf{apolar inner product}, by
	$$\left\langle f,g\right\rangle_d=D_f(g)\quad \text{for }g\in \RR[\x]_{2k_1}
		\otimes \RR[\y]_{2k_2}.$$
\begin{example}
	For $f=f_1\otimes f_2\in \RR[\x]_{2k_1}\otimes \RR[\y]_{2k_2}$ and 
	$g=g_1\otimes g_2\in \RR[\x]_{2k_1}\otimes \RR[\y]_{2k_2}$, we have
		$$D_f(g)=D_{f_1}(g_1)D_{f_2}(g_2).$$
\end{example}
Note that this inner product is invariant under the action of $\SO(n)\times \SO(m)$. For a point $v=(v_1,\ldots,v_n)\in S^{n-1}$, we denote by
$v^{2k_1}$ the form
	$$v^{2k_1}:=(v_1x_1+\cdots+v_nx_n)^{2k_1}\in \RR[\x]_{2k_1}.$$
We define a linear operator $T_v:\RR[\x]_{2k_1}\to \RR[\x]_{2k_1}$ by
	\begin{equation}\label{definicija-T_v}
		T_v(r)= \int_{S^{n-1}} r(v) v^{2k_1} \dd\sigma_1(v).
	\end{equation}
Analogously for a point $u=(u_1,\ldots,u_m)\in S^{m-1}$  we denote by $u^{2k_2}$ and
$T^u$ the form and the linear operator on $\RR[\y]_{2k_2}$ given by
	$$u^{2k_2}=(\sum_{j=1}^m u_jy_j)^{2k_2}\in \RR[\y]_{2k_2}\quad\text{and}\quad
		T^u(t)= \int_{S^{m-1}} t(u) u^{2k_2} \dd\sigma_2(u).$$
Finally let
	$$T:\RR[\x]_{2k_1}\otimes \RR[\y]_{2k_2}\to \RR[\x]_{2k_1}\otimes \RR[\y]_{2k_2}$$
be the linear operator defined by
	$$T\left(\sum_{\ell} f_{\ell 1} \otimes f_{\ell 2}\right)=
		\sum_{\ell } T_v\left(f_{\ell 1}\right)\otimes T^u\left(f_{\ell 2}\right).$$
Some properties of the operator $T$ we will need are collected in the following lemma.

\begin{lemma}	\label{pomozna}
	The following statements hold:
	\begin{enumerate}[label={\rm(\arabic*)}]
		\item \label{pom-pt1} 
			The operator $T$ relates the two inner products by the following identity,
	$$\left\langle Tf,g\right\rangle_d = (2k_1)! (2k_2)! \left\langle f,g\right\rangle.$$
		\item \label{pom-pt2}
			The operator $T$ is bijective.
		\item \label{pom-pt3} The operator $T$ has eigenspace
			$\Span\left\{(\sum_{i=1}^n x_i^2)^{k_1}\otimes (\sum_{j=1}^m y_j^2)^{k_2}\right\}$, i.e.,
	$$T\left((\sum_{i=1}^n x_i^2)^{k_1}\otimes (\sum_{j=1}^m y_j^2)^{k_2}\right)=
		c\cdot \left( (\sum_{i=1}^n x_i^2)^{k_1}\otimes (\sum_{j=1}^m y_j^2)^{k_2}\right),$$
	where
	$$c=\frac{\Gamma(k_1+\frac{1}{2}) \Gamma(\frac{n}{2})  }
			{\sqrt{\pi}\Gamma(k_1+\frac{n}{2})} 
		\frac{\Gamma(k_2+\frac{1}{2}) \Gamma(\frac{m}{2})  }
			{\sqrt{\pi} \Gamma(k_2+\frac{m}{2})}.$$
	\end{enumerate}
\end{lemma}

\begin{proof}
		By  bilinearity it suffices to prove Lemma \ref{pomozna} \ref{pom-pt1}
	only for elementary tensors
		$f=f_1\otimes f_2 ,$ $g= g_1\otimes g_2\in \RR[\x]_{2k_1}\otimes \RR[\y]_{2k_2}$. Since
			\begin{eqnarray*}
				\left\langle Tf,g\right\rangle_d &=& \left\langle T_v f_1,g_1\right\rangle_{d_{\x}} \left\langle T^u f_2,g_2\right\rangle_{d_{\y}},\\
				\left\langle f,g\right\rangle &=& \left\langle f_1,g_1\right\rangle \left\langle f_2,g_2\right\rangle,
			\end{eqnarray*}
		Lemma \ref{pomozna} \ref{pom-pt1} follows by \cite[Lemma 5.1]{Blek1}.
		
		Since $T$ maps from the finite-dimensional vector space into itself, 
		to prove Lemma \ref{pomozna} \ref{pom-pt2}, it suffices to prove that the kernel 
		of $T$ is trivial.	
		Let us assume that $Tf=0$ for some 
		$f\in \RR[\x]_{2k_1}\otimes \RR[\y]_{2k_2}$. 
		By Lemma \ref{pomozna} \ref{pom-pt1}
		it follows that $\left\langle f,f\right\rangle=0$. 
		Hence $f=0$ and the kernel of $T$ is trivial.
		
		Finally, Lemma \ref{pomozna} \ref{pom-pt3} follows
		by
		\begin{align*}
	 	T\left((\sum_{i=1}^n x_i^2)^{k_1}\otimes (\sum_{j=1}^m y_j^2)^{k_2}\right)=
		T_v\left((\sum_{i=1}^n x_i^2)^{k_1}\right) \otimes T^u\left((\sum_{j=1}^m y_j^2)^{k_2}\right)\\
		=\left(\frac{\Gamma(k_1+\frac{1}{2}) \Gamma(\frac{n}{2})  }{\sqrt{\pi}\Gamma(k_1+\frac{n}{2})}
		(\sum_{i=1}^n x_i^2)^{k_1}\right)\otimes
		\left(\frac{\Gamma(k_2+\frac{1}{2}) \Gamma(\frac{m}{2})  }{\sqrt{\pi} \Gamma(k_2+\frac{m}{2})}
		(\sum_{j=1}^m y_j^2)^{k_2}\right),
	\end{align*}
	where the second equality follows by \cite[p.\ 371]{Blek1} used for $T_v$ and $T^u$.
\end{proof}

	Let $\cL$ be a full-dimensional cone
	in $\RR[\x]_{2k_1}\otimes \RR[\y]_{2k_2}$ containing
		$(\sum_{i=1}^n x_i^2)^{k_1}\otimes (\sum_{j=1}^m y_j^2)^{k_2}$
	in its interior, and satisfying $\int_T f \dd\sigma>0$ for all non-zero $f\in \cL$.
	Let $\widetilde{\cL}$ be the subset of $\cM$ defined by
		$$ \widetilde{\cL}=\left\{ f\in \cM\colon f+(\sum_{i=1}^n x_i^2)^{k_1}\otimes (\sum_{j=1}^m y_j^2)^{k_2}\in \cL\right\}.$$
Let $\cL^{\ast}$ be the dual cone of $\cL$ w.r.t.\ the $L^2$ inner product and $\cL^\ast_d$ the dual cone of $\cL$ in the apolar inner product,
	\begin{eqnarray*}
		\cL^{\ast} &=& \left\{  f\in \RR[\x]_{2k_1}\otimes \RR[\y]_{2k_2}\colon \left\langle f,g\right\rangle\geq 0 \quad \text{for all } g\in \cL   \right\},\\
		\cL^{\ast}_d &=& \left\{  f\in \RR[\x]_{2k_1}\otimes \RR[\y]_{2k_2}\colon \left\langle f,g\right\rangle_d \geq 0 \quad \text{for all } g\in \cL   \right\}.
	\end{eqnarray*}

\begin{proposition}
		The biform $(\sum_{i=1}^n x_i^2)^{k_1}\otimes (\sum_{j=1}^m y_j^2)^{k_2}$ 
	belongs to the interiors of $\cL^\ast$ and $\cL^\ast_d$.
\end{proposition}

\begin{proof}
		The biform
	$r_\x^{k_1}\otimes s_{\y}^{k_2}:=
	(\sum_{i=1}^n x_i^2)^{k_1}\otimes (\sum_{j=1}^m y_j^2)^{k_2}$ is in 
	the interior of $\cL^\ast$ (resp.\ $\cL^\ast_d$) if and only if
	$\left\langle r_\x^{k_1}\otimes s_{\y}^{k_2},g\right\rangle>0$ 
	(resp.\ $\left\langle r_\x^{k_1}\otimes s_{\y}^{k_2},g\right\rangle_d>0$) 
	is true for all $g\in \cL$. 
	Since
		$$\left\langle r_\x^{k_1}\otimes s_{\y}^{k_2},g\right\rangle=\int_T
			(r_\x^{k_1}\otimes s_{\y}^{k_2})\cdot g \dd\sigma=
			\int_T g \dd\sigma,
		$$
	(resp.\
	\begin{eqnarray*}
			\left\langle r_\x^{k_1}\otimes s_{\y}^{k_2},g\right\rangle_d
			&=&\frac{1}{(2k_1)!(2k_2)!} 
			\left\langle T^{-1} \left(r_\x^{k_1}\otimes s_{\y}^{k_2}\right),g\right\rangle
			= \frac{1}{(2k_1)!(2k_2)! c} 
			\left\langle \left(r_\x^{k_1}\otimes s_{\y}^{k_2}\right),g\right\rangle\\
			&=&
			\frac{1}{(2k_1)!(2k_2)! c} \int_T g \dd\sigma,
	\end{eqnarray*}
	where $c$ is defined as in Lemma \ref{pomozna} \ref{pom-pt3}, 
	and the first equality follows by Lemma \ref{pomozna} \ref{pom-pt1}, 
	\ref{pom-pt2}, while the second one by  Lemma \ref{pomozna} \ref{pom-pt3}),
	this is true by definition of $\cL$.
\end{proof}

	Let $\widetilde{\cL^\ast}$ and $\widetilde{\cL^\ast_d}$ be defined by
	\begin{eqnarray*}
		\widetilde{\cL^\ast}
		&=&	\left\{ f\in \cM\colon f+(\sum_{i=1}^n x_i^2)^{k_1}\otimes
			(\sum_{j=1}^m y_j^2)^{k_2}\in \cL^\ast\right\},\\
		\widetilde{L^\ast_d}
		&=&	\left\{ f\in \cM\colon f+(\sum_{i=1}^n x_i^2)^{k_1}\otimes
			(\sum_{j=1}^m y_j^2)^{k_2}\in \cL^\ast_d\right\}.
	\end{eqnarray*}

The following is an analog of \cite[Lemma 5.2]{Blek1}.

\begin{lemma} \label{cone-L}
		Let $\cL$ be a full dimensional cone in $\RR[\x]_{2k_1}\otimes \RR[\y]_{2k_2}$ such that the polynomial $(\sum_{i=1}^n x_i^2)^{k_1}\otimes (\sum_{j=1}^m y_j^2)^{k_2}$
	is the interior point of $\cL$ and $\int_T f \dd\sigma>0$ for all non-zero $f\in \cL$. Then we have the following
	relationship between the volumes of $\widetilde{\cL^\ast}$ and $\widetilde{\cL^\ast_d}$:
		$$\frac{k_1!}{(\frac{n}{2}+2k_1)^{k_1}}\frac{k_2!}{(\frac{m}{2}+2k_2)^{k_2}}\leq \left(\frac{\Vol \widetilde{\cL^\ast_d}}{\Vol \widetilde{\cL^\ast}}\right)^{\frac{1}{D_{\cM}}}
			\leq \left(\frac{k_1!}{(\frac{n}{2}+k_1)^{k_1}}\frac{k_2!}{(\frac{m}{2}+k_2)^{k_2}}\right)^{\alpha_{2k_1,2k_2}},$$
	where $$\alpha_{2k_1,2k_2}=1-\left(\frac{2k_1-1}{n+2k_1-1}\right)^2-\left(\frac{2k_2-1}{m+2k_2-1}\right)^2+
			\left(\frac{2k_1}{n+2k_1-2} \frac{2k_2}{m+2k_2-2}\right)^2.$$
\end{lemma}

\begin{proof}
	From Lemma \ref{pomozna} \ref{pom-pt1} it follows that
for all $f,g\in \RR[\x]_{2k_1}\otimes \RR[\y]_{2k_2}$,
	$\left\langle f, g\right\rangle \geq 0$ if and only if $\left\langle Tf, g\right\rangle_d\geq 0$. Therefore, $T$ maps $\cL^\ast$ to $\cL^\ast_d$.
By Lemma \ref{pomozna} it follows that
for all $f,g\in \RR[\x]_{2k_1}\otimes \RR[\y]_{2k_2}$,
	$\left\langle f, g\right\rangle_d \geq 0$ if and only if 
	$\left\langle T^{-1} f, g\right\rangle\geq 0$, where $T^{-1}$ is the inverse of $T$.
	Therefore, $T$ maps $\cL^\ast$ onto $\cL^\ast_d$,
		\begin{equation} \label{action-of-T}
			T(\cL^{\ast})=\cL^{\ast}_d.
		\end{equation}
	Let $\Delta_{\x}=\sum_{i=1}^n \frac{\partial^2}{\partial x_i^2}$
	(resp.\ $\Delta_{\y}=\sum_{j=1}^m \frac{\partial^2}{\partial y_j^2}$) be the Laplace differential
	operator on $\RR[\x]$ (resp.\ $\RR[\y]$).
	Then $\RR[\x]_{2k_1}$ (resp.\ $\RR[\y]_{2k_2}$) splits into irreducible $\SO(n)$-modules
	(resp.\ $\SO(m)$-modules) \cite[Chapter IX \S 2]{Vil},
		$$\RR[\x]_{2k_1}=\oplus_{i=0}^{k_1} r_{\x}^{i} H_{n,2k_1-2i},\quad (\text{resp.\ } \RR[\y]_{2k_2}=\oplus_{j=0}^{k_2} s_{\y}^{j} H_{m,2k_2-2j}),$$
	where
	$$r_{\x}=\sum_{i=1}^n x_i^2\quad \text{and}\quad
		H_{n,2i}=\left\{ r\in \RR[\x]_{2i} \colon \Delta_{\x} r= 0 \right\}$$
	(resp.\ $s_{\y}=\sum_{j=1}^m y_j^2$ and
		$H_{m,2j}=\left\{ s\in \RR[\y]_{2j} \colon \Delta_{\y} s= 0 \right\}$).
	Then the $\SO(n)\times \SO(m)$-module $\RR[\x]_{2k_1}\otimes \RR[\y]_{2k_2}$ splits into submodules as follows:
	 	$$\RR[\x]_{2k_1}\otimes \RR[\y]_{2k_2}=\oplus_{i=0}^{k_1} \oplus_{j=0}^{k_2} (r_{\x}^{i} H_{n,2k_1-2i}\otimes s_{\y}^{j} H_{m,2k_2-2j}).$$
	By Lemma \ref{pomozna} \ref{pom-pt3},
		$$T\left( r_\x^{k_1}\otimes s_\y^{k_2}\right)=
			c\cdot \left( r_\x^{k_1}\otimes s_\y^{k_2}\right),$$
	where
		$$c=\frac{\Gamma(k_1+\frac{1}{2}) \Gamma(\frac{n}{2})  }{\sqrt{\pi}
		\Gamma(k_1+\frac{n}{2})}
		\frac{\Gamma(k_2+\frac{1}{2}) \Gamma(\frac{m}{2})  }{\sqrt{\pi} \Gamma(k_2+\frac{m}{2})}.$$
	Since $\frac{1}{c}T$ commutes with the action of $\SO(n)\times \SO(m)$ and fixes
	$r_\x^{k_1}\otimes s_\y^{k_2}$, it also fixes the orthogonal complement of 
	$ r_\x^{k_1}\otimes s_\y^{k_2}$, which is the hyperplane
	of all biforms with integral 0 on $T$. Using this and $(\ref{action-of-T})$, we conclude that
	$\frac{1}{c}T$ maps $\widetilde{\cL^\ast}$
	to $\widetilde{\cL_d^\ast}$. Applying \cite[Lemma 7.4]{Blek3} componentwise for $\frac{1}{c}T$ we have that
		\begin{eqnarray*}
			\frac{1}{c}T\left(\sum_\ell f_{\ell 1}\otimes f_{\ell 2}\right)
			&=&	\sum_\ell \left(\sum_{j=0}^{k_1}\sum_{k=0}^{k_2} c_{jk} \ell^{\x}_{2j}(f_{\ell 1})\otimes \ell^{\y}_{2k}(f_{\ell 2})\right)\\
			&=&	\sum_{j=0}^{k_1}\sum_{k=0}^{k_2} c_{jk} \left(\sum_\ell \ell^{\x}_{2j}(f_{\ell 1})\otimes \ell^{\y}_{2k}(f_{\ell 2})\right),
		\end{eqnarray*}
	where
		$$c_{jk}=\frac{k_1!\ \Gamma(k_1+\frac{n}{2})}{(k_1-j)!\ \Gamma(k_1+j+\frac{n}{2})} \frac{k_2!\ \Gamma(k_2+\frac{m}{2})}{(k_2-k)!\ \Gamma(k_2+k+\frac{m}{2})}$$
	and $\ell^{\x}_{2j}(f_{\ell 1})$ (resp.\ $\ell^{\y}_{2k}(f_{\ell 2})$)  denotes the orthogonal projection of $f_{\ell 1}$ to $r_{\x}^{k_1-j} H_{n,2j}$ (resp.\ $f_{\ell 2}$ to $s_{\y}^{k_2-k} H_{m,2k}$).
	Note that
	$c_{k_1k_2}$ is the smallest among the coefficients $c_{jk}$ and 
	the lower bound on the
	change in volume is
	 	$$\left(\frac{ \Vol \widetilde{\cL^\ast_d}}{\Vol \widetilde{\cL^\ast }} \right)^{\frac{1}{D_{\cM}}} \geq
			\frac{k_1!\ \Gamma(k_1+\frac{n}{2})}{\Gamma(2k_1+\frac{n}{2})}\frac{k_2!\ \Gamma(k_2+\frac{m}{2})}{\Gamma(2k_2+\frac{m}{2})}=c_{k_1k_2}.$$
	Estimate
		$$\frac{k_1!\ \Gamma(k_1+\frac{n}{2})}{\Gamma(2k_1+\frac{n}{2})}\frac{k_2! \ \Gamma(k_2+\frac{m}{2})}{\Gamma(2k_2+\frac{m}{2})}\geq
			\frac{k_1!}{(\frac{n}{2}+2k_1)^{k_1}}\frac{k_2!}{(\frac{m}{2}+2k_2)^{k_2}}.
		$$
	This proves the lower bound in Lemma \ref{cone-L}.

	To prove the upper bound in Lemma \ref{cone-L} observe that the largest coefficient of contraction occurs in the submodule $H_{n,2k_1}\otimes H_{m,2k_2}$ which has dimension
		\begin{eqnarray*}
			D_H
			&=& (\dim \RR[\x]_{2k_1}-\dim \RR[\x]_{2k_1-2})
				(\dim \RR[\y]_{2k_2}-\dim \RR[\y]_{2k_2-2})\\
			&=& \Big(\binom{n+2k_1-1}{2k_1}-\binom{n+2k_1-3}{2k_1-2}\Big)  \cdot
				\Big(\binom{m+2k_2-1}{2k_2}-\binom{m+2k_2-3}{2k_2-2}\Big).
		\end{eqnarray*}
	The dimension $D_\cM$ of the ambient space $\cM$ is
		$$D_{\cM}=\binom{n+2k_1-1}{2k_1} \binom{m+2k_2-1}{2k_2}-1.$$
	We have
		$$D_H
			= \binom{n+2k_1-1}{2k_1}\binom{m+2k_2-1}{2k_2} \cdot C,$$
	where
 		$$C=\left(1-\frac{2k_1-1}{n+2k_1-2}\frac{2k_1}{n+2k_1-1}\right)
			\left(1-\frac{2k_2-1}{m+2k_2-2}\frac{2k_2}{m+2k_2-1}\right).$$
	Thus
		$$D_H=D_{\cM}\cdot C+ C \underbrace{<}_{C<1} D_{\cM}\cdot C+ 1.$$
	If $k_1>1$ or $k_2>1$, then
		\begin{eqnarray*}
		\frac{D_{H}}{D_{\cM}}< C+\frac{1}{D_{\cM}}
		&<& C+\frac{1}{\binom{n+2k_1-1}{2} \binom{m+2k_2-1}{2}}\\
		&=&
		C+\frac{4}{(n+2k_1-1)(n+2k_1-2)(m+2k_2-1)(m+2k_2-2)}\\
		&<& \alpha_{2k_1,2k_2},
		\end{eqnarray*}
	where $\alpha_{2k_1,2k_2}$ is as in the statement of Lemma \ref{cone-L}. On the other hand, if $k_1=k_2=1$,
	then
	  \begin{equation*}
	  \frac{D_{H}}{D_{\cM}}<C+\frac{4}{n(n+1)m(m+1)-4}<C+\frac{8}{n(n+1)m(m+1)}<\alpha_{2,2}.
	  \end{equation*}
	Estimating $c_{k_1k_2}$ from above gives
		$$c_{k_1k_2}=\frac{k_1!\ \Gamma(k_1+\frac{n}{2})}{\Gamma(2k_1+\frac{n}{2})}\frac{k_2!\ \Gamma(k_2+\frac{m}{2})}{\Gamma(2k_2+\frac{m}{2})}\leq
			\frac{k_1!}{(\frac{n}{2}+k_1)^{k_1}}\frac{k_2!}{(\frac{m}{2}+k_2)^{k_2}},
		$$
	which concludes the proof of the lemma.
\end{proof}

\begin{lemma}\label{lemma-sq}
		The dual cone $\Sq_d^\ast$ to the cone of sums of squares
		$\Sq^{(n,m)}_{(2k_1,2k_2)}$ in the apolar inner product is contained in the cone of sums of
		squares $\Sq^{(n,m)}_{(2k_1,2k_2)}$,
			$$\Sq_d^\ast\subseteq \Sq^{(n,m)}_{(2k_1,2k_2)}.$$
\end{lemma}

\begin{proof}
 	Let $W$ be the space of quadratic forms on $\RR[\x]_{k_1}\otimes \RR[\y]_{k_2}$.
	 For $A,B\in W$ with the corresponding symmetric matrices
	$\mathscr M_A$ and $\mathscr M_B$ with respect to an orthonormal basis 
	 for the apolar differential inner product, we define the inner product of $A,B$ by
		$$\left\langle A,B \right\rangle = \text{tr }(\mathscr M_A \mathscr M_B).$$
	For $q\in \RR[\x]_{k_1}\otimes \RR[\y]_{k_2}$, let $A_q$ be the rank one quadratic form given by
		$$A_q(p)=\left\langle p,q \right\rangle_d^2\quad \text{for }
			p\in \RR[\x]_{k_1}\otimes \RR[\y]_{k_2}.$$
	For any $B\in W$ we have
		$$\left\langle A_q,B \right\rangle=B(q).$$
	For $f\in \RR[\x]_{2k_1}\otimes \RR[\y]_{2k_2}$, let $H_f$ be the quadratic form on
	$\RR[\x]_{k_1}\otimes \RR[\y]_{k_2}$ given by
		\begin{equation} \label{def-H}
			H_f(p)=\left\langle p^2,f \right\rangle_d\quad \text{for }
			p\in \RR[\x]_{k_1}\otimes \RR[\y]_{k_2}.
		\end{equation}
	If $f\in \Sq_d^\ast$, then $H_f$ is positive semidefinite by definition. Therefore it can be written as a finite nonnegative linear combination of forms of rank 1,
		$$H_f=\sum_{k} A_{q_k}\quad \text{for some } q_k\in \RR[\x]_{k_1}\otimes \RR[\y]_{k_2}.$$
	Let $W_1$ be the subspace of $W$ given by the linear span of the forms $H_g$ with
	$g\in \RR[\x]_{2k_1}\otimes \RR[\y]_{2k_2}$.
	Let $P_{W_1}$
	be the orthogonal projection
	onto $W_1$. We claim that
		\begin{equation}\label{projekcija}
			P_{W_1}(A_{q_k})=\binom{2k_1}{k_1}^{-1}\binom{2k_2}{k_2}^{-1} H_{q_k^2}.
		\end{equation}
	From
		$$\left\{h\in \RR[\x]_{2k_1}\colon h(v)=0\text{ for all }v\in \RR^n\right\}=
			\left\{0\right\},$$
	it follows by \cite[equality (1.9)]{Rez92} that 
		$$\Span\left\{v^{2k_1}\colon v\in \RR^n\right\}=\RR[\x]_{2k_1}.$$
	Analogously 
	$\Span\left\{u^{2k_2}\colon u\in \RR^m\right\}=\RR[\y]_{2k_2}.$
	Thus 
		\begin{equation} \label{powers-of-lin-forms}
			\Span\left\{v^{2k_1}\otimes u^{2k_2}\colon v\in \RR^n, u\in \RR^m\right\}=				\RR[\x]_{2k_1}\otimes \RR[\y]_{2k_2}.
		\end{equation}
	To establish (\ref{projekcija}) it suffices to show that
		$A_{q_k}-\binom{2k_1}{k_1}^{-1}\binom{2k_2}{k_2}^{-1} H_{q_k^2}$
	is orthogonal to the forms
		$H_{v^{2k_1}\otimes u^{2k_1}}$
	since these span $W_1$.
	We observe that
		\begin{eqnarray*}
			H_{v^{2k_1}\otimes u^{2k_1}}(p)
				&=&(2k_1)!\ (2k_2)!\ p(u,v)^2=
				\frac{(2k_1)!\ (2k_2)!\ A_{v^{k_1}\otimes u^{k_2}}(p)}{(k_1!\ k_2!)^2}\\
				&=&\binom{2k_1}{k_1}\binom{2k_2}{k_2} A_{v^{k_1}\otimes u^{k_2}}(p).
		\end{eqnarray*}
	Therefore
		\begin{align*}
		\left\langle A_{q_k}-\binom{2k_1}{k_1}^{-1}\binom{2k_2}{k_2}^{-1} H_{q_k^2}, H_{v^{2k_1}\otimes u^{2k_1}}\right\rangle  &
		= H_{v^{2k_1}\otimes u^{2k_1}}(q_k)-
		\left\langle H_{q_k^2}, A_{v^{k_1}\otimes u^{k_1}}\right\rangle \\
		& = H_{v^{2k_1}\otimes u^{2k_1}}(q_k)-
		H_{q_k^2}(v^{k_1}\otimes u^{k_1})=
		0.
		\end{align*}
	Hence,
		$$H_f=P_{W_1}(\sum_k A_{q_k})\stackrel{\eqref{projekcija}}{=}
			\sum_k \binom{2k_1}{k_1}^{-1}\binom{2k_2}{k_2}^{-1} H_{q_k^2}=
		H_{	\binom{2k_1}{k_1}^{-1}\binom{2k_2}{k_2}^{-1}  \sum_k q_k^2},$$
	and 
		$$H_{f-\binom{2k_1}{k_1}^{-1}\binom{2k_2}{k_2}^{-1}  \sum_k q_k^2}\equiv 0.$$
	From (\ref{def-H}) it follows that
		\begin{equation} \label{def-H-used}
			\left\langle p^2, f-\binom{2k_1}{k_1}^{-1}\binom{2k_2}{k_2}^{-1}  \sum_k q_k^2 \right\rangle_d=0\quad \text{for all }
			p\in \RR[\x]_{k_1}\otimes \RR[\y]_{k_2}.
		\end{equation}
	In particular, by the equality (\ref{powers-of-lin-forms}), the linear span of the squares of
	forms from
	$\RR[\x]_{k_1}\otimes \RR[\y]_{k_2}$ is the whole space
	$\RR[\x]_{2k_1}\otimes \RR[\y]_{2k_2}$. Therefore (\ref{def-H-used})
	implies that 
		$$
			\left\langle g, f-\binom{2k_1}{k_1}^{-1}\binom{2k_2}{k_2}^{-1}  \sum_k q_k^2 \right\rangle_d=0\quad \text{for all }
			g\in \RR[\x]_{2k_1}\otimes \RR[\y]_{2k_2},$$
	and hence 
		$f$ is a sum of squares,
		$$f=\binom{2k_1}{k_1}^{-1}\binom{2k_2}{k_2}^{-1}  \sum_k q_k^2.$$
	This concludes the proof of Lemma \ref{lemma-sq}.
\end{proof}

\subsection{Sums of squares biforms}

\label{sos-subsec} 

In this subsection we establish the bounds for the volume of the section of sums of squares biforms. The main result is as follows.

\begin{theorem} \label{squares}
	For integers $n,m\geq 3$ we have
		$$d_{2k_1,2k_2}
			\leq
			\left(\frac{\Vol \widetilde{\Sq}^{(n,m)}_{(2k_1,2k_2)}}{\Vol B_\cM}\right)^{\frac{1}{D_{\cM}}}
			\leq e_{2k_1,2k_2},$$
	where
		\begin{eqnarray*}
		d_{2k_1,2k_2} &=&
			\left\{\begin{array}{lr}
				 \frac{1}{2^8\sqrt{6}} \frac{\sqrt{nm+n+m}}{(n+4)(m+4)},& \text{if } k_1=k_2=1\\
				\frac{(k_1!\ k_2!)\cdot \sqrt{k_1!\ k_2!}}{2\sqrt{6}\cdot4^{2k_1+2k_2}\cdot \sqrt{(2k_1)!\ (2k_1)!}} \frac{  n^{\frac{k_1}{2}}  m^{\frac{k_2}{2}}  }{(\frac{n}{2}+2k_2)^{k_1} (\frac{m}{2}+2k_2)^{k_2}},&
						 \text{otherwise,}
			\end{array}\right.\\
		e_{2k_1,2k_2}&=&
			\left\{\begin{array}{lr}
				2^{10}\sqrt{6} \cdot \frac{1}{\sqrt{nm+n+m}},& \text{if } k_1=k_2=1\\
				2\sqrt{6}\cdot 4^{2k_1+2k_2}\cdot \sqrt{\frac{(2k_1)!\ (2k_2)!}{k_1!\ k_2!}} \cdot
					n^{-\frac{k_1}{2}} m^{-\frac{k_2}{2}},& \text{otherwise,}
			\end{array}\right.
		\end{eqnarray*}
\end{theorem}

Blekherman \cite[Theorem 6.1]{Blek1} established
volume bounds for sum of squares forms.
Our proof freely borrows from his ideas.
An important ingredient in the proof will also be the
following version of the Reverse H\"older inequality.

\begin{lemma} \label{bask-prod}
	For a biform $g\in \RR[\x,\y]_{k_1,k_2}$ of bidegree $(k_1,k_2)$ we have
		$$\left(  \int_T g^2 \; \dd\sigma \right)^{\frac{1}{2}}=
		\left\| g \right\|_2\leq 4^{k_1+k_2}\left\|g\right\|_1 =4^{k_1+k_2} \left( \int_T g\; \dd\sigma\right).$$
\end{lemma}

\begin{proof}
	By definition,
		\begin{equation}\label{rev-hol-1}
			\int_T g^2  \dd\sigma=\int_{x\in S^{n-1}}
			\left(\int_{y\in S^{m-1}} g^2(x,y) \;\dd\sigma_2(y)\right) \;\dd\sigma_1(x).
		\end{equation}
	For every fixed $x\in S^{n-1}$, $g^2(x,\y)$ is a polynomial in $\y$ of degree $2k_2$.
	By the Reverse H\"older inequality \cite[Corollary 3]{Duoa} used for $p=1$, $k=k_2$, $P_k(\y)=g(x,\y)$ we obtain
		\begin{equation}\label{rev-hol-2}
			\left(\int_{S^{m-1}}{ g^2(x,y)\; \dd\sigma_2(y)}\right)^{\frac{1}{2}}\leq
			4^{k_2}\left(\int_{S^{m-1}}{ g(x,y)\; \dd\sigma_2(y)}\right),
		\end{equation}
	for each $x\in S^{n-1}$.
	Hence using (\ref{rev-hol-2}) in (\ref{rev-hol-1}) we have
		\begin{equation} \label{rev-hol-3}
			\int_T g^2 \; \dd\sigma \leq
			4^{2k_2}\int_{x\in S^{n-1}}
			\left(\int_{y\in S^{m-1}} g(x,y) \;\dd\sigma_2(y)\right)^2 \;\dd\sigma_1(x).
		\end{equation}
	The expression $\int_{y\in S^{m-1}} g(\x,y) \;\dd\sigma_2(y)$ is a polynomial in $\x$ of degree $k_1$.
	Using the Reverse H\"older inequality \cite[Corollary 3]{Duoa} for $p=1$, $k=k_1$, $P_k(\x)=\int_{\y\in S^{m-1}} g(\x,y) \;\dd\sigma_2(y)$,
		\[
				\left(\int_{x\in S^{n-1}}\left(\int_{y\in S^{m-1}} g(x,y) \;\dd\sigma_2(y)\right)^2 \dd\sigma_1(x)\right)^{\frac{1}{2}}\leq  4^{k_1}\left( \int_{S^{n-1}}\int_{S^{m-1}} g(x,y)\:\dd\sigma_2(y)\dd\sigma_1(x) \right).
		\]
	Using this in (\ref{rev-hol-3}) we get
		\begin{eqnarray*}
			\int_T g^2 \; \dd\sigma
				&\leq&  4^{2k_1+2k_2}\left(\int_{x\in S^{n-1}} \int_{y\in S^{m-1}} g(x,y) \;\dd\sigma_2(y) \;\dd\sigma_1(x)\right)^2\\
				&=& 4^{2k_1+2k_2} \left\| g\right\|_1^2.
		\end{eqnarray*}
	Taking square roots concludes the proof of the lemma.
\end{proof}

\subsubsection{Proof of the upper bound in Theorem \ref{squares}}
	We write $\widetilde{\Sq}=\widetilde{\Sq}^{(n,m)}_{(2k_1,2k_2)}$.
	We define the support function
		$L_{\widetilde{\Sq}}$ of $\widetilde{\Sq}$ by
		$$L_{\widetilde{\Sq}}:\cM\to \RR,\quad
			L_{\widetilde{\Sq}}(f)=\max_{g\in \widetilde{\Sq}} \left\langle f,g\right\rangle.$$
	The average width $W_{\widetilde{\Sq}}$ of $\widetilde{\Sq}$ is given by
		\begin{equation}\label{width-1}
			W_{\widetilde{\Sq}}=2\int_{S_{\cM}} L_{\widetilde{\Sq}}\; \dd\widetilde\mu.
		\end{equation}
	By Urysohn's inequality \cite[p.\ 318]{Sch} applied to $\widetilde{\Sq}$ we have
		\begin{equation}\label{Urysohn-estimate}
			\left(\frac{\Vol \widetilde{\Sq}}{\Vol B_{\cM}}\right)^{\frac{1}{D_{\cM}}}\leq
			\frac{W_{\widetilde{\Sq}}}{2}.
		\end{equation}
	Let $S_\cU$ be the unit sphere in $\cU:=\RR[\x,\y]_{k_1, k_2}$ equipped with the 
	$L^2$ norm, i.e., 	
	\begin{equation*}\label{measure'}
		\left\| g \right\|_2^2= \int_{T} |g|^2\ \dd\sigma=
		\int_{x\in S^{n-1}}\left(\int_{y\in S^{m-1}} |g(x,y)|^2 \;\dd\sigma_2(y)\right) \;\dd	
		\sigma_1(x).
	\end{equation*} 
	The extreme points of $\widetilde{\Sq}$ are of the form
		$$g^2-(\sum_{i=1}^n x_i^2)^{k_1}(\sum_{j=1}^m y_j^2)^{k_2}\quad
			\text{where }\; g\in \RR[\x,\y]_{k_1, k_2}
			\quad \text{and}\quad \int_{T}g^2\dd\sigma=1.$$
	For $f\in \cM$,
		$$\left\langle f,(\sum_{i=1}^n x_i^2)^{k_1}(\sum_{j=1}^m y_j^2)^{k_2} \right\rangle=
			\int_{T}f\dd\sigma=0,$$
	and thus
		$$L_{\widetilde{\Sq}}(f)=\max_{g\in S_{\cU}} \langle f,g^2 \rangle.$$
	Let $\|\;\|_{\sq}$ be the norm on $\RR[\x,\y]_{2k_1, 2k_2}$ defined by
		$$\|f\|_{\sq}=\max_{g\in S_{\cU}} |\langle f,g^2\rangle|,$$
	Clearly,
		\begin{equation}\label{width-2}
			L_{\widetilde{\Sq}}(f)\leq \|f\|_{\sq}.
		\end{equation}
	Using (\ref{width-1}), (\ref{Urysohn-estimate}) and (\ref{width-2}) we deduce
		$$\left(\frac{\Vol \widetilde{\Sq}}{\Vol B_\cM}\right)^{\frac{1}{D_{\cM}}}\leq
			\int_{S_{\cM}} \left\|f\right\|_{\sq} \dd\widetilde\mu.$$
	To prove the upper bound in Theorem \ref{squares} it now suffices to prove the following claim.\\

	\noindent \textbf{Claim:} $\ds\int_{S_{\cM}} \left\|f\right\|_{\sq} \dd\widetilde\mu \leq e_{2k_1,2k_2}.$\\

	For $f\in \RR[\x,\y]_{2k_1, 2k_2}$ let $H_f$ be the quadratic form on $\RR[\x,\y]_{k_1, k_2}$ defined by
		$$H_f(g)=\langle f, g^2\rangle\quad\text{for all }g\in \RR[\x,\y]_{k_1, k_2}.$$
	Note that
		$$\|f\|_{\sq}=\|H_{f}\|_{\infty}.$$
	Here $\|H_{f}\|_{\infty}$ stands for the supremum norm of 	$H_f$ on the unit sphere $S_\cU$.
	Let $\widehat\mu$ be the $\SO(n)\times \SO(m)$-invariant probability measure on $S_\cU$.
	The $L^{2p}$ norm of $H_f$ for a positive integer
	$p$ is defined by
		$$\|H_f\|_{2p}:=\left(\int_{S_{\cU}} H_f^{2p}(g)\dd\widehat\mu\right)^{\frac{1}{2p}}.$$
	We will bound $\|H_{f}\|_{\infty}$ by a $L^{2p}$ norm of $H_f$ for $p$ high enough.
	Note that
	$H_f$ is a form of degree 2 in the vector space $\cU$ of dimension $D_\cU$. By the proof of
	\cite[Theorem 4.2]{Blek1} we have that
		$$\|H_{f}\|_{\infty}\leq 2\sqrt{3}\|H_{f}\|_{2D_\cU}.$$
	It suffices to estimate
		$$A=\int_{S_\cM} \|H_{f}\|_{2D_\cU}\dd\widetilde\mu=
			\int_{S_\cM} \left(\int_{S_\cU} \langle f,g^2\rangle^{2D_\cU} \dd\widehat{\mu}(g)\right)^{\frac{1}{2D_\cU}}\dd\widetilde\mu(f).$$
	Applying H\"older's inequality and interchanging the order of integration we obtain
		\begin{equation}\label{A-estimate}
			A\leq \left(\int_{S_\cU} \int_{S_\cM} \langle f,g^2\rangle^{2D_\cU}
				\dd\widetilde\mu(f)\dd\widehat{\mu}(g)\right)^{\frac{1}{2D_\cU}}.
		\end{equation}
	We estimate  the inner integral as follows:
		\begin{eqnarray}
			\int_{S_{\cM}} \left\langle f, g^2\right\rangle^{2 D_{\cU}}\; \dd\widetilde\mu(f)
			&=&	\int_{S_{\cM}} \left\langle f, \pr_\cM (g^2)\right\rangle^{2 D_{\cU}}\; \dd\widetilde\mu(f)
				\nonumber\\
			&=& \left\| \pr_\cM (g^2)\right\|_2^{2D_\cU}
				\int_{S_{\cM}} \left\langle f, \frac{\pr_\cM(g^2)}{ \left\| \pr_\cM(g^2)\right\|_2}\right\rangle^{2 D_{\cU}}\; \dd\widetilde\mu(f) \nonumber \\
			&\leq& \left\|g^2\right\|_2^{2D_\cU}
				\int_{S_{\cM}} \left\langle f, \frac{\pr_\cM(g^2)}{ \left\| \pr_\cM(g^2)\right\|_2}\right\rangle^{2 D_{\cU}}\; \dd\widetilde\mu(f) \label{inner}
		\end{eqnarray}
	where $\pr_\cM(g^2)$ denotes the projection of $g^2$ into $\cM$.
	Observe that
		$$\left\| g\right\|_2=1.$$
	By Lemma \ref{bask-prod} used for $g^2$ it follows that
		$$\left\| g^2\right\|_2\leq 4^{2k_1+2k_2}.$$
	Using this in (\ref{inner}) we get
		\begin{equation}\label{rev-hold-est-1}
			\int_{S_{\cM}} \left\langle f, g^2\right\rangle^{2 D_{\cU}}\; \dd\widetilde\mu(f) \leq
			4^{4(k_1+k_2)D_{\cU}} \int_{S_{\cM}} \left\langle f,\frac{\pr_\cM(g^2)}{ \left\| \pr_\cM(g^2)\right\|}
			\right\rangle^{2 D_{\cU}}\; \dd\widetilde\mu(f).
		\end{equation}
	As in \cite[p.\ 376]{Blek1} we estimate
		\begin{equation}\label{Blek-est}
			\int_{S_{\cM}} \left\langle f, \frac{\pr_\cM(g^2)}{ \left\| \pr_\cM(g^2)\right\|_2}
			\right\rangle^{2 D_{\cU}}\; \dd\widetilde\mu(f)
			\leq \left(\sqrt{\frac{2D_{\cU}}{D_{\cM}}}\right)^{2D_{\cU}}.
		\end{equation}
	Combining (\ref{rev-hold-est-1}) and (\ref{Blek-est}) we see
		\begin{equation}\label{Blek-est-2}
			\int_{S_{\cM}} \left\langle f, g^2\right\rangle^{2 D_{\cU}}\; \dd\widetilde\mu(f) \leq
			4^{4(k_1+k_2)D_{\cU}} \left(\sqrt{\frac{2D_{\cU}}{D_{\cM}}}\right)^{2D_{\cU}}.
		\end{equation}
	Using (\ref{Blek-est-2}) in (\ref{A-estimate}) we obtain
		$$A\leq 4^{2(k_1+k_2)} \cdot\sqrt{\frac{2D_{\cU}}{D_{\cM}}}\cdot
			\int_{S_{\cU}}1\;\dd\widehat{\mu}(g)=4^{2(k_1+k_2)} \cdot\sqrt{\frac{2D_{\cU}}{D_{\cM}}}.$$
	To prove the Claim it remains to establish
		\begin{equation}\label{ostanek}
			\sqrt{\frac{2D_{\cU}}{D_{\cM}}} \leq \frac{e_{2k_1,2k_2}}{2\sqrt{3}\cdot 4^{2(k_1+k_2)}}.
		\end{equation}
	The dimensions $D_{\cU}$, $D_\cM$ are easily verified to be
		\begin{eqnarray*}
			D_{\cU} &=& \dim \RR[\x,\y]_{k_1,k_2}=
				\binom{n+k_1-1}{k_1} \binom{m+k_2-1}{k_2},\\
			D_{\cM} &=& \dim \RR[\x,\y]_{2k_1,2k_2}-1=
			 	\binom{n+2k_1-1}{2k_1} \binom{m+2k_2-1}{2k_2}-1.
		\end{eqnarray*}
	We distinguish two cases depending on $k_1$, $k_2$.\\

	\noindent \textbf{Case 1:} $k_1=k_2=1$.
	Observe that
		 \begin{eqnarray*}
		 \frac{2D_{\cU}}{D_{\cM} } &=&
		 \frac{2^3 \cdot nm}{n(n+1)m(m+1)-4} = \frac{2^3}{(n+1)(m+1)-\frac{4}{nm}}\\
			   &\underbrace{<}_{n,m\geq 3}&  \frac{2^3}{(n+1)(m+1)-1}= \frac{2^3}{nm+n+m}.
		\end{eqnarray*}
	\noindent \textbf{Case 2:} $k_1>1$ or $k_2>1$.\\

	Note that
		$$D_{\cM }>  \Big(\binom{n+2k_1-1}{2k_1}-1\Big)\cdot\Big(\binom{m+2k_2-1}{2k_2}-1\Big).$$
	Hence
		\begin{equation} \label{2D_U/D_M}
			 \sqrt{\frac{2D_{\cU}}{D_{\cM} } }
			 <	\sqrt{2\frac{\binom{n+k_1-1}{k_1}}{\binom{n+2k_1-1}{2k_1}-1} \frac{\binom{m+k_2-1}{k_2}}{\binom{m+2k_2-1}{2k_2}-1}}.
		\end{equation}
	Using (\ref{2D_U/D_M}) together with the estimates
		\begin{equation} \label{2D_U/D_M-est}
			\frac{ \binom{n+k_1-1}{k_1} }{\binom{n+2k_1-1}{2k_1}-1}<
			\frac{(2k_1)!}{k_1!} n^{-k_1} \quad \text{(resp.}\quad
			\frac{\binom{m+k_2-1}{k_2}}{\binom{m+2k_2-1}{2k_2}-1}<
			\frac{(2k_2)!}{k_2!} m^{-k_2}),
		\end{equation}
	which we prove below,
	it follows that
			 $$\sqrt{\frac{2D_{\cU}}{D_{\cM} } }
			 < \sqrt{2\frac{(2k_1)! (2k_2)!}{k_1! k_2!}} n^{-\frac{k_1}{2}} m^{-\frac{k_2}{2}}.$$
		This proves \eqref{ostanek} and establishes the upper bound in Theorem \ref{squares}.
	Hence it only remains to prove (\ref{2D_U/D_M-est}). We have
		\begin{eqnarray*}
			\frac{ \binom{n+k_1-1}{k_1} }{\binom{n+2k_1-1}{2k_1}-1} 
			&\underbrace{\leq}_{\binom{n+2k_1-1}{2k_1}\geq \binom{n+1}{2}}&
			\frac{\binom{n+k_1-1}{k_1} }{\frac{n^2+n-2}{n^2+n} \binom{n+2k_1-1}{2k_1}}=
			\frac{(2k_1)!\cdot n(n+1)\cdot (n+k_1-1)!}{(k_1)! \cdot (n-1)(n+2)\cdot (n+2k_1-1)! }\\
			&\underbrace{\leq}_{k_1\geq 1}&
			\frac{(2k_1)!\cdot n(n+1)}{(k_1)! \cdot (n-1)(n+2)(n+1)n^{k_1-1}}
			\underbrace{<}_{n\geq 3}
			\frac{(2k_1)!}{(k_1)!} n^{-k_1},
		\end{eqnarray*}
	which proves (\ref{2D_U/D_M-est}).  \hfill\qedsymbol

\subsubsection{Proof of the lower bound in Theorem \ref{squares}}
	Let $B_{\sq}$ be the unit ball of the $\|\;\|_{\sq}$ norm
		$$B_{\sq}=\left\{f\in \cM\colon \|f\|_{\sq}\leq 1\right\}=
			\left\{f\in \cM\colon \max_{g\in S_{\cU}}
				\left|\left\langle f,g^2\right\rangle\right| \leq 1\right\},$$
	where $S_{\cU}$ stands for the unit sphere in $\cU:=\RR[\x,\y]_{k_1,k_2}$ equipped with the $L^2$ norm.
	By the Claim in the proof of the upper bound of Theorem \ref{squares}, we have
		\begin{equation} \label{sq-norm-est}
			\int_{S_{M}} \left\|f\right\|_{\sq}\dd\widetilde\mu \leq e_{2k_1,2k_2}.
		\end{equation}
	By \cite[p.\ 91]{14},
		\begin{equation*} \label{expression-gauge-2}
			\left( \frac{\Vol B_{\sq}}{\Vol B_{\cM}} \right)^{\frac{1}{D_{\cM}}}=
			\left( \int_{S_{\cM}}G_{B_{\sq}}^{-D_{\cM}} \dd\widetilde\mu   \right)^{\frac{1}{D_{\cM}}},
		\end{equation*}
	where $G_{B_{\sq}}$ is the gauge of $B_{\sq}$ in $\cM$.
	Observe that 
		$$G_{B_{\sq}}(f)=\left\|f\right\|_{\sq}.$$
	\noindent By H\"older's inequality we have
			$$\left( \int_{S_{\cM}}G_{B_{\sq}}^{-D_{\cM}} \dd\widetilde\mu   \right)^{\frac{1}{D_{\cM}}}\geq
				\int_{S_\cM} G_{B_{\sq}}^{-1}\dd\widetilde\mu,$$
	and so
		$$\left(\frac{\Vol B_{\sq}}{\Vol B_\cM}\right)^{\frac{1}{D_\cM}}\geq
			\int_{S_\cM} G_{B_{\sq}}^{-1}\dd\widetilde\mu.$$
	By Jensen's inequality
	(applied to the convex function $y=\frac{1}{x}$ on $\RR_{>0}$),
		$$\int_{S_\cM} G_{B_{\sq}}^{-1}\dd\widetilde\mu \geq \left(\int_{S_\cM} G_{B_{\sq}}\dd\widetilde\mu\right)^{-1}.$$
	Therefore using (\ref{sq-norm-est}) we have
		\begin{equation} \label{ball-estimate}
			\left( \frac{\Vol B_{\sq}}{\Vol B_\cM} \right)^{\frac{1}{D_\cM}}\geq
			e_{2k_1,2k_2}^{-1}.
		\end{equation}
	
	Let $\left(\Sq'\right)^\circ$ be the polar dual of $\Sq':=\left(\Sq^{(n,m)}_{(2k_1,2k_2)}\right)'$.
	 By definition,
		$$\left(\Sq'\right)^\circ
			=\left\{f\in \cM\colon \langle f,g\rangle\leq 1\quad\text{for all }g\in \Sq'\right\}.$$
		
	\noindent \textbf{Claim:} $B_{\sq}=\left(\Sq'\right)^\circ \bigcap -\left(\Sq'\right)^\circ.$\\
			
	First we prove the inclusion $(\subseteq)$ in the Claim. Let us take 
	$f\in B_{\sq}$ and show that $f\in \left(\Sq'\right)^\circ \bigcap -\left(\Sq'\right)^\circ$.
	By definition of $\left(\Sq'\right)^\circ$ we have to prove that
		\begin{equation}\label{to-be-proved}
			\left|\langle f,h\rangle\right|\leq 1\quad\text{for all }h\in \Sq'.
		\end{equation}
	By assumption $f\in B_{\sq}$ we have
		\begin{equation}\label{def-B-sq}
			\left|\left\langle f,g^2\right\rangle\right| \leq 1\quad\text{for all }
			g\in S_{\RR[\x,\y]_{k_1,k_2}}.
		\end{equation}
	Notice that every $h=\sum_{i=1}^k h_i^2\in \left(\Sq'\right)^\circ$ can be written as
		$h=\sum_{i=1}^k \lambda_i \left(\frac{h_i}{\sqrt{\lambda_i}}\right)^2$, where
	$\lambda_i=\int_T h_i^2\dd\sigma$ and $\sum_{i=1}^k \lambda_i=1$.
	Since $\frac{h_i}{\sqrt{\lambda_i}}\in S_{\RR[\x,\y]_{k_1,k_2}}$,
	it follows by (\ref{def-B-sq}) that 
	$\left|\left\langle f,\frac{h_i}{\sqrt{\lambda_i}}\right\rangle\right|\leq 1$.
	Hence $\left|\left\langle f,h\right\rangle\right|\leq 1$.
	This proves (\ref{to-be-proved}).
		
		The inclusion $(\supseteq)$ in the Claim is trivial (since the definition (\ref{def-B-sq}) is 
	a special case of the definition (\ref{to-be-proved})).\\
	
	Let $\Sq^\ast$ be the dual cone of $\Sq$
	in the $L^2$ inner product,
		$$\Sq^\ast=\{f\in \RR[\x,\y]_{2k_1,2k_2}\colon
			\langle f, g\rangle\geq 0 \quad \text{for all }g\in \Sq\},$$
	and let $\widetilde{\Sq^\ast}$ be the set
		\begin{eqnarray*}
			\widetilde{\Sq^\ast}
			&=& \left\{ f\in \RR[\x,\y]_{2k_1,2k_2}\colon f+(\sum_i x_i^2)^{k_1}(\sum_j y_j^2)^{k_2}\in
				\Sq^\ast \bigcap\; \mathcal{H}^{(n,m)}_{(2k_1,2k_2)}\right\}\\
			&=& \left\{ f\in \cM\colon \langle f+(\sum_i x_i^2)^{k_1}(\sum_j y_j^2)^{k_2},g\rangle\geq 0
				\quad \text{for all }g\in \Sq\right\}\\
			&=&   \left\{ f\in \cM\colon \langle f+(\sum_i x_i^2)^{k_1}(\sum_j y_j^2)^{k_2},g\rangle\geq 0
				\quad \text{for all }g\in \Sq'\right\}\\
			&=&   \left\{ f\in \cM\colon \langle f,g\rangle\geq -1
				\quad \text{for all }g\in \Sq'\right\}\\
			&=&   \left\{ f\in \cM\colon \langle -f,g\rangle\leq 1
				\quad \text{for all }g\in \Sq'\right\}\\
			&=& -\left(\Sq'\right)^\circ,
		\end{eqnarray*}
	where the second equality follows by definitions of $\Sq^\ast$ and $\mathcal{H}^{(n,m)}_{(2k_1,2k_2)}$,
	the third by homogeneity of the inner product and the forth by
		$\left\langle (\sum_i x_i^2)^{k_1}(\sum_j y_j^2)^{k_2},g\right\rangle=1$ for $g\in \Sq'$.
	Using $\widetilde{\Sq^\ast}=-\left(\Sq'\right)^\circ$
	together with (\ref{ball-estimate}) and the Claim we get
		\begin{equation}\label{estimate-5}
			\left(\frac{\Vol \widetilde{\Sq^\ast}}{\Vol B_{\cM}}\right)^{\frac{1}{D_{\cM}}}
			\geq e_{2k_1,2k_2}^{-1}.
		\end{equation}

	Since $(\sum_i x_i^2)^{k_1}(\sum_j y_j^2)^{k_2}$ is in the interior of $\Sq_{(2k_1,2k_2)}^{(n,m)}$ and for all non-zero
	$f\in \Sq_{(2k_1,2k_2)}^{(n,m)}$ we have $\int_T f\dd\sigma>0$,
	it follows by Lemma \ref{cone-L} that
		\begin{equation}\label{using-cone-L}
			\left(\frac{\Vol \widetilde{\Sq_d^\ast}}{\Vol \widetilde{\Sq^\ast}} \right)^{\frac{1}{D_{\cM}}} \geq \frac{k_1!}{(\frac{n}{2}+2k_1)^{k_1}}\frac{k_2!}{(\frac{m}{2}+2k_2)^{k_2}},
		\end{equation}
	where
		$$\widetilde{\Sq_d^\ast}
			= \left\{ f\in \RR[\x,\y]_{2k_1,2k_2}\colon f+(\sum_i x_i^2)^{k_1}(\sum_j y_j^2)^{k_2}\in \Sq_d^\ast\; \bigcap\; \mathcal{H}^{(n,m)}_{(2k_1,2k_2)}\right\},$$
	and $\Sq_d^\ast\subset \RR[\x,\y]_{2k_1,2k_2}$ is the dual cone in the apolar inner product of $\Sq_{(2k_1,2k_2)}^{(n,m)}$.
	Combining (\ref{estimate-5}) and  (\ref{using-cone-L}) we obtain
		$$\left(\frac{\Vol \widetilde{\Sq_d^\ast}}{\Vol B_\cM}\right)^{\frac{1}{D_{\cM}}} \geq
			e_{2k_1,2k_2}^{-1}
			\frac{k_1!}{(\frac{n}{2}+2k_1)^{k_1}}\frac{k_2!}{(\frac{m}{2}+2k_2)^{k_2}}.
			$$
	By Lemma \ref{lemma-sq},  $\widetilde{\Sq_d^\ast}\subseteq \widetilde{\Sq_{(2k_1,2k_2)}^{(n,m)}}$ and the lower bound of Theorem \ref{squares} is proved.
\hfill\qedsymbol

\subsection{Extension of the results to symmetric multiforms}\label{ext-to-sym-multiforms}

Let $\FF\in \{\RR,\CC\}$ and $\FF[\z,\overline{\z},\w,\overline{\w}]$ be the vector space of polynomials over $\FF$ in the complex variables $\z:=(z_1,\ldots,z_n)$ and
$\w:=(w_1,\ldots,w_m)$, equipped with conjugation as the involution $\ast$ (in case $\FF=\RR$ the involution is trivial on the coefficients). 
Let $\symm\FF[\z,\overline{\z},\w,\overline{\w}]_{1,1,1,1}$ be the real subspace of symmetric multiforms of 
multidegree $(1,1,1,1)$, i.e., 
	$$\symm\FF[\z,\overline{\z},\w,\overline{\w}]_{1,1,1,1}:=
		\left\{\sum_{i,j=1}^n\sum_{k,\ell=1}^m a_{ijk\ell} \overline z_i z_j \overline{w_k} w_\ell\colon a_{ijk\ell}\in \FF,\ 
		a_{ijk\ell}=\overline{a_{ji\ell k}} \text{ for all }i,j,k,\ell\right\}.$$
\begin{remark}\label{dim-remark}
	It is easy to check that the real dimension of $\symm\FF[\z,\overline{\z},\w,\overline{\w}]_{1,1,1,1}$ is $n^2m^2$ for $\FF=\CC$ and 
	$\frac{1}{2}nm(nm+1)$ for $\FF=\RR$.
\end{remark}
Let $\FF[\z,\w]$ stand for the vector subspace of $\FF[\z,\overline{\z},\w,\overline{\w}]$ of polynomials in $\z,\w$, 
and $\FF[\z,\w]_{1,1}$ for the subspace of 
$\FF[\z,\w]$ of bilinear polynomials, i.e., polynomials from 
$\FF[\z,\w]$ that are linear in 
$\z$ and $\w$.
Let
\begin{align*}
	\Pos_{\FF} & = \left\{ f\in \symm\FF[\z,\overline{\z},\w,\overline{\w}]_{1,1,1,1}\colon f(\z,\w)\geq 0\quad \text{for all }(\z,\w)\in\CC^n\times \CC^m \right\},\\
	\Sq_{\FF} &= \left\{f\in \symm\FF[\z,\overline{\z},\w,\overline{\w}]_{1,1,1,1} \colon f=\sum_{r}f_r^\ast f_r\quad \text{for some }f_r\in \FF[\z,\w]_{1,1}\right\},
\end{align*}
be the cone of nonnegative multiforms and the cone of sum of hermitian squares multiforms, respectively. 
Let $\symm \CC[\z,\overline \z,\w,\overline \w]_{1,1}$ stand for the real subspace of $\CC[\z,\overline{\z},\w,\overline{\w}]$ of symmetric bilinear polynomials in 
$(\z,\overline\z)$ and $(\w,\overline\w)$, i.e.,
$$\symm \CC[\z,\overline \z,\w,\overline \w]_{1,1}:=\left\{  \sum_{i=1}^n\sum_{j=1}^m \left(a_{ij}z_iw_j + \overline{a_{ij}}\overline{z_i}\overline{w_j}+b_{ij}\overline{z_i}w_j + \overline{b_{ij}} z_i\overline{w_j}\right) \colon a_{ij}, b_{ij}\in \CC\right\}.$$

\begin{proposition}\label{sos-prop}
	We have
		\begin{equation*}
			\Sq_{\FF} 
				\subseteq\left\{f\in \symm\FF[\z,\overline{\z},\w,\overline{\w}]_{1,1,1,1} \colon f=\sum_{r}f_r^2\quad \text{for some }f_r\in \symm \CC[\z,\overline{\z}, \w,
					\overline{\w}]_{1,1}\right\}.
		\end{equation*}
\end{proposition}

\begin{proof}
	The proposition follows by the equality
		$$f^\ast f=f_{\re}^2+f_{\im}^2,$$
	where $f\in \FF[\z,\w]_{1,1}$ and $f_{\re}:=\frac{f+f^\ast}{2}, f_{\im}:=\frac{f-f^\ast}{2i}$ belong to $\symm \CC[\z,\overline \z,\w,\overline \w]_{1,1}$.
\end{proof}

Now we introduce new real variables $\x:=(x_1,\ldots,x_{2n})$ and $\y:=(y_1,\ldots,y_{2m})$ such that 
	$$z_j=x_j+i\cdot x_{n+j},\quad w_k=y_k+i\cdot y_{m+k} \quad\text{for}\quad j=1,\ldots,n, \; k=1,\ldots,m.$$
Under this identification the real vector space $\symm\FF[\z,\overline{\z},\w,\overline{\w}]_{1,1,1,1}$ becomes a subspace of $\RR[\x, \y]_{2,2}$ which we denote by $\cC_\FF$. 
We write $\Pos_{\cC_\FF}\subset\cC_\FF$ and $\Sq_{\cC_\FF}\subset\cC_\FF$ for the images of sets $\Pos_{\FF}$ and $\Sq_{\FF}$, respectively. 
Let $\Pos^{(2n,2m)}_{(2,2)}$ and $\Sq^{(2n,2m)}_{(2,2)}$ be defined as in (\ref{pos-def}) and (\ref{sos-def}), respectively.

\begin{proposition} \label{inclusions}
	We have 
		$$\Pos_{\cC_\FF}=\Pos^{(2n,2m)}_{(2,2)}\cap\;\cC_\FF,
			\quad \Sq_{\cC_\FF}\subseteq \Sq^{(2n,2m)}_{(2,2)}\cap\;\cC_\FF.$$
\end{proposition}


\begin{proof}
	The equality for $\Pos_{\cC_\FF}$ is clear. 
	The set $\symm \CC[\z,\overline{\z}, \w,\overline{\w}]_{1,1}$ maps bijectively to $\RR[\x,\y]_{1,1}$.
	(Clearly $\symm \CC[\z,\overline{\z}, \w,\overline{\w}]_{1,1}$ maps to  $\RR[\x,\y]_{1,1}$ and by expressing
		$$x_j=\frac{z_j+\overline{z_j}}{2}, \quad x_{n+j}=\frac{z_j-\overline{z_j}}{2i}, \quad y_k=\frac{w_k+\overline{w_k}}{2},\quad y_{m+k}=\frac{w_k-\overline{w_{k}}}{2i}$$
	for $j=1,\ldots,n$, $k=1,\ldots,m$, we see that each element from $\RR[\x,\y]_{1,1}$ comes from an element of $\symm \CC[\z,\overline{\z}, \w,\overline{\w}]_{1,1}$.)
	Therefore the set
		$$\{f\in \symm\FF[\z,\overline{\z},\w,\overline{\w}]_{1,1,1,1} \colon f=\sum_{r}f_r^2\quad \text{for some }f_r\in \symm \CC[\z,\overline{\z}, \w,
					\overline{\w}]_{1,1}\}$$
	maps bijectively to $\Sq^{(2n,2m)}_{(2,2)}\cap\;\cC_\FF$.
	Thus by Proposition \ref{sos-prop} the inclusion $\Sq_{\cC_\FF}\subseteq \symm \FF[\z,\overline{\z}, \w,\overline{\w}]_{1,1,1,1}$ follows.
\end{proof}

Recall the definitions of the product measure $\sigma$ from Subsection \ref{sec-main-res} and the set $\cM^{(n,m)}_{(2,2)}$ from (\ref{M-def}) and replace $(n,m)$ with $(2n,2m)$. We define the vector subspace $\cM_{\cC_\FF}$ of $\cM^{(2n,2m)}_{(2,2)}$ by
	\begin{align*}
		\cM_{\cC_\FF}&:=\cM^{(2n,2m)}_{(2,2)}\cap \cC_\FF,
	\end{align*}
and its sections $\widetilde{\Pos}_{\cC_\FF}$, $\widetilde{\Sq}_{\cC_\FF}$ by 
	\begin{align*}
		\widetilde{\Pos}_{\cC_\FF}&:=\left\{f\in \cM_{\cC_\FF}\colon f+(\sum_{i=1}^{2n}x_i^2)(\sum_{j=1}^{2m} y_j^2)\in \Pos_{\cC_\FF}\right\},\\
		\widetilde{\Sq}_{\cC_\FF}&:=\left\{f\in \cM_{\cC_\FF}\colon f+(\sum_{i=1}^{2n}x_i^2)(\sum_{j=1}^{2m} y_j^2)\in \Sq_{\cC_\FF}\right\}.
	\end{align*}
The subspace $\cM_{\cC_\FF}$ is a Hilbert subspace of $\RR[\x,\y]_{2,2}$ equipped with the $L^2(\sigma)$ inner product and we write $D_{\cM_{\cC_\FF}}$ for its dimension;
so it is isomorphic to $\RR^{\cM_{\cC_\FF}}$ as a Hilbert space.
Let $S_{\cM_{\cC_\FF}}$, $B_{\cM_{\cC_\FF}}$ be the unit sphere and the unit ball in $\cM_{\cC_\FF}$, respectively. 
Let $\mu$ be the (unique w.r.t.\ unitary isomorphism) pushforward of the Lebesgue measure on $\RR^{D_{\cM_{\cC_\FF}}}$ to $\cM_{\cC_\FF}$ (cf.\ Lemma \ref{unique-pushforward}).

The bounds for the volume of the set $\widetilde{\Pos}_{\cC_\FF}$ are as follows.

\begin{theorem}\label{lower-bound-F}
	For 
	integers $n,m\ge 3$ we have:
		$$3^{3} \cdot 10^{-\frac{20}{9}}\cdot 2^{-\frac 12} \max(n,m)^{-\frac{1}{2}}\leq \left(\frac{\Vol \widetilde{\Pos}_{\cC_\FF}}{\Vol B_{\cM_{\cC_\FF}}}\right)^{\frac{1}{D_{\cM_{\cC_\FF}}}}\leq
		2\left(\min\left(\frac{1}{1+n},  \frac{1}{1+m}\right)\right)^{\frac{1}{2}}.$$
\end{theorem}

We regard the vector space $\symm\FF[\z,\overline{\z},\w,\overline{\w}]_{1,1,1,1}$ as a module over the group $G$ where $G=\SU(n)\times \SU(m)$ is the product of special unitary groups if $\FF=\CC$ and $G=\SO(n)\times \SO(m)$ is the product of special orhogonal groups if $\FF=\RR$, with the actions given by  
$$(A,B)\cdot f(\z,\overline{\z},\w,\overline{\w}):= f(A^{-1}\z, \overline{A^{-1}}\overline{\z} ,B^{-1}\w, \overline{B^{-1}} \overline{\w})$$
where $(A,B)\in G$, $f\in \symm\FF[\z,\overline{\z},\w,\overline{\w}]_{1,1,1,1}$ and conjugations over $A^{-1}$, $B^{-1}$ are trivial if $\FF=\RR$.

\begin{proof}[Proof of Theorem \ref{lower-bound-F}]
	The proofs of both bounds are analogous to the proofs of the corresponding bounds in Theorem \ref{lower-bound} with some minor changes:
	\begin{enumerate}
		\item Since $\cC_{\FF}$ is a subspace in $\RR[\x,\y]_{2,2}$ where $\x=(x_1,\ldots,x_{2n})$, $\y:=(y_1,\ldots,y_{2m})$, we work with twice as many variables as in Theorem \ref{lower-bound}.
		\item In the proof of the lower bound there is a slight change in the part where we estimate 
			$\int_{S_{\cM_{\cC_{\FF}}}} \left\| f\right\|_{2k_0}\; \dd\widetilde\mu$. Namely, we use the fact that the elements in
			$S_{\cM_{\cC_\FF}}$ correspond to restrictions of linear functionals in $S_{\widetilde{V}_2}$ where $\widetilde{V}_2$ is a vector subspace of $V_2$ 
			(that is identified with $\cM_{\cC_\FF}$). On replacing $V_2$ with $\widetilde V_2$, the equality (\ref{identifikacija}) remains true and 
			also the rest of the proof is the same.
		\item	 In the proof of the upper bound the validity of the inequality (\ref{Blaschke}) for $\widetilde{\Pos}_{\cC_\CC}$ (resp.\ $\widetilde{\Pos}_{\cC_\RR}$) follows since
			$\widetilde{\Pos}_{\cC_\CC}$ (resp.\ $\widetilde{\Pos}_{\cC_\RR}$) 
			is invariant under the action of $\SU(n)\times \SU(m)$ (resp.\ $\SO(n)\times \SO(m)$) and since the origin is the only fixed point under this action.\qedhere
	\end{enumerate}
\end{proof}

We next present the upper bound for the volume of the set $\widetilde{\Sq}_{\cC_\FF}$.

\begin{theorem} \label{squares-F}
	For integers $n,m\geq 2$ we have
		$$\left(\frac{\Vol \widetilde{\Sq}_{\cC_\FF}}{\Vol B_{\cM_{\cC_\FF}}}\right)^{\frac{1}{D_{\cM_{\cC_\FF}}}}
			\leq 2^{10+\frac{3-\dim_{\RR}\FF}{2}}\sqrt{3} \cdot \frac{1}{\sqrt{nm-1}}.$$
\end{theorem}

\begin{proof}
	The proof is analogous to the proof of the upper bound in Theorem \ref{squares} with some minor changes:
	\begin{enumerate}
		\item Since $\cC_{\FF}$ is a subspace in $\RR[\x,\y]_{2,2}$ where $\x=(x_1,\ldots,x_{2n})$, $\y:=(y_1,\ldots,y_{2m})$, we work with twice as many variables as in Theorem \ref{lower-bound}.
		\item Since  $L_{\widetilde{\Sq}_{\cC_\FF}}(f)=\max_{g\in \widetilde{\Sq}_{\cC_\FF}}\langle f,g\rangle$ and 
			$\widetilde{\Sq}_{\cC_\FF}\subset \widetilde{\Sq}:=\widetilde{\Sq}^{(2n,2m)}_{(2,2)}$, it is true that
			$\displaystyle L_{\widetilde{\Sq}_{\cC_\FF}}(f)\leq \max_{g\in \widetilde{\Sq}}\langle f,g\rangle$.
			Now the inequality $\displaystyle L_{\widetilde{\Sq_{\cC_\FF}}}(f)\leq \left\|f\right\|_{\sq}$ 		
			is established in the same way as the inequality (\ref{width-2}) and everything up to the equality (\ref{ostanek}) remains the same.
			The estimate of $\sqrt{\frac{2D_{\cU}}{D_{\cM_{\cC_\FF}}}}$ becomes
				$$\sqrt{\frac{2D_{\cU}}{D_{\cM_{\cC_\FF}}}}\leq 
				\left\{\begin{array}{rl}
					\sqrt{\frac{8nm}{n^2m^2-1}}=\sqrt{\frac{8}{nm-\frac{1}{nm}}}< \frac{2\sqrt{2}}{\sqrt{nm-1}},& \text{if }\FF=\CC,\\
					\sqrt{\frac{16nm}{nm(nm+1)-2}}=\sqrt{\frac{16}{nm+1-\frac{2}{nm}}}< \frac{4}{\sqrt{nm-1}},& \text{if }\FF=\RR.
				\end{array}\right.$$
			For the first inequality we used Remark \ref{dim-remark}.\qedhere
	\end{enumerate}
\end{proof}

%
%

\section{Positive maps and biforms} \label{pos-maps-biforms}

In this section we connect linear maps on matrices with biforms, thus translating
the question of comparing the size of the cone of completely positive maps with the size of the cone of positive maps to the question of comparing the size of the cone of sums of squares biforms with the size of the cone of positive biforms.

	We denote by $\cL(\Sym_n,\Sym_m)$ the vector space of all linear maps
	from $\Sym_n$ to $\Sym_m$.
		There is a linear bijection $\Gamma$ between linear maps $\cL(\Sym_n,\Sym_m)$ and biforms $\RR[\x,\y]_{2, 2}$ of bidegree $(2,2)$
	given by
		\begin{equation} \label{identification}
			\Gamma:\cL(\Sym_n,\Sym_m)\to \RR[\x,\y]_{2, 2} ,\quad
			\Phi\mapsto p_\Phi(\x,\y):=\y^\ast \Phi(\x \x^\ast)\y.
		\end{equation}
	Thus $\Gamma$
translates
between
 properties of linear maps from $\cL(\Sym_n,\Sym_m)$ and the corresponding properties of biforms
from $\RR[\x,\y]_{2, 2}$. Positivity (resp.\ complete positivity) of a map $\Phi$ corresponds to nonnegativity (resp.\ being a sum of squares) of the polynomial $p_{\Phi}$:

\begin{proposition} \label{map-poly}
	Let $\Phi:\Sym_n\to \Sym_m$ be a linear map.
	Then
		\begin{enumerate}[label={\rm(\arabic*)}]
			\item $\Phi$ is positive iff $p_{\Phi}$ is nonnegative;
			\item \label{cp-sos} $\Phi$ is completely positive iff $p_{\Phi}$ is a sum of squares.
		\end{enumerate}
\end{proposition}

\begin{proof}
		The implication ($\Rightarrow$) of (1) is trivial. For the implication ($\Leftarrow$) observe that any positive semidefinite matrix $X\in \Sym_n$ can be written as the sum
	$X=\sum_{i=1}^k v_i v_i^\ast$ where $v_i\in \RR^n$ for each $i$. Hence $y^\ast\Phi(X)y=\sum_{i=1}^k y^\ast p_{\Phi}(v_iv_i^\ast)y$ is positive for every $y\in \RR^m$.

		To prove the implication ($\Rightarrow$) of (2) first invoke the Arveson's extension theorem
\cite[Theorem 7.5]{PAU}  to extend $\Phi$ to a completely positive map
$\widetilde \Phi: M_n\to M_m$ and then the Stinespring's representation theorem (see \cite[Theorem 4.1]{PAU} or \cite[Theorem 1]{Cho75}) to represent $\widetilde \Phi$ in the form $X\mapsto \sum_{i=1}^\ell V_i^\ast X V_i$ for some $\ell\in \NN$ where
	$\sum_{i=1}^{\ell}V_i^\ast V_i$ is a bounded operator of norm
	$\left\|\Phi\right\|$. (The proofs of the real finite-dimensional versions of \cite[Theorem 7.5]{PAU} and \cite[Theorem 4.1]{PAU} can be found, for example, in \cite[\S3.1]{HKM1}.) Hence
		$$p_\Phi(\x,\y)= \sum_{i=1}^\ell \y^\ast V_i^\ast \x  \x^\ast V_i  \y=\sum_{i=1}^\ell q_i^2(\x,\y),$$
	where $q_i(\x,\y)=\x^\ast V_i \y$ for each $i$.

		It remains to prove the implication ($\Leftarrow$) of (2). It suffices to prove that there is an extension of $\Phi$ to a completely positive map
	$\widetilde \Phi:M_n\to M_m$.
	Since $p_\Phi(\x,\y)$ is a sum of squares, it is of the form
		\begin{eqnarray*}
			p_\Phi(\x,\y)
			&=& \sum_{i=1}^\ell q_{i}(\x,\y)^2
			= \sum_{i=1}^{\ell} \left(\sum_{j=1}^{m}\sum_{k=1}^{n} q_{ijk} x_{k}y_{j} \right)^2=
				\sum_{i=1}^{\ell} (\y^\ast (q_{ijk})_{jk}\x)^2\\
			&=&	\sum_{i=1}^{\ell} \y^\ast (q_{ijk})_{jk}\x \x^\ast(q_{ijk})_{jk}^\ast \y,
		\end{eqnarray*}
	where $q_{i}(\x,\y)=\sum_{j=1}^{m}\sum_{k=1}^{n} q_{ijk} x_{k}y_{j}\in \RR[\x,\y]$.
	From $p_\Phi(\x,\y):=\y^\ast \Phi(\x \x^\ast)\y$ it follows that
		$$\Phi(\x \x^\ast)=\sum_{i=1}^\ell (q_{ijk})_{jk}\x \x^\ast(q_{ijk})_{jk}^\ast.$$
	Hence the map $\widetilde \Phi:M_n\to M_m$ defined by
		$$\widetilde \Phi(X)=\sum_{i=1}^\ell (q_{ijk})_{jk} X (q_{ijk})_{jk}^\ast \quad \text{for all }
			X\in M_n(\RR)$$
	is a completely positive extension of $\Phi$.
\end{proof}

%
Let $\Pos(\Sym_n,\Sym_m)$ and $\CP(\Sym_n, \Sym_m)$ denote the cone of positive maps
	and the cone of completely positive maps from $\Sym_n$ to $\Sym_m$, respectively.
By Proposition \ref{map-poly}, comparing the cones $\Pos(\Sym_n,\Sym_m)$ and $\CP(\Sym_n, \Sym_m)$
is equivalent to comparing the cones $\Pos^{(n,m)}_{(2,2)}$ and $\Sq^{(n,m)}_{(2,2)}$.

\subsection{Comparing the volumes of  $\widetilde{\Pos}^{(n,m)}_{(2,2)}$ and 
$\widetilde{\Sq}_{(2,2)}^{(n,m)}$}

In this subsection we obtain bounds on the ratio between the volumes of the sets
$\widetilde{\Sq}_{(2,2)}^{(n,m)}$ and $\widetilde{\Pos}^{(n,m)}_{(2,2)}$.
By Theorem \ref{n,m>3-intro} the sets are the same if and only if $n\leq 2$ or $m\leq 2$.
Here is the main result of this subsection.

\begin{theorem} \label{cor-1}
	For integers $n,m\geq 3$ we have
		$$  \frac{3\cdot \sqrt{3}}{2^{10} \cdot 7^2 \cdot \sqrt{\min(n,m)}}<
		\left(\frac{\Vol \widetilde{\Sq}_{(2,2)}^{(n,m)} }{\Vol \widetilde{\Pos}^{(n,m)}_{(2,2)}}\right)^{\frac{1}{D_{\cM}}} < 
		\frac{2^{12}\cdot 5^2 \cdot 6^{\frac{1}{2}} \cdot 10^{\frac{2}{9}}}{3^3\cdot \sqrt{\min(n,m)+1}},$$
	where $D_\cM=\binom{n+1}{2}\binom{m+1}{2}-1=
		\frac{(n+1)n(m+1)m-4}{4}.$ In particular, 
		$$\left(\frac{\Vol \widetilde{\Sq}_{(2,2)}^{(n,m)} }{\Vol \widetilde{\Pos}^{(n,m)}_{(2,2)}}\right)^{\frac{1}{D_{\cM}}} =\Theta\left(\min(n,m)^{-\frac{1}{2}}\right).$$
\end{theorem}

\begin{proof}
	 We first prove the upper bound.
Combining the lower bound in Theorem \ref{lower-bound} with the upper bound in Theorem \ref{squares} we have
		$$\left(\frac{\Vol \widetilde{\Sq}_{(2,2)}^{(n,m)}   }{\Vol
			\widetilde{\Pos}^{(n,m)}_{(2,2)} }\right)^{\frac{1}{D_{\cM}}} \leq
			 		\frac{2^{12}\cdot 5^2 \cdot 6^{\frac{1}{2}} \cdot 10^{\frac{2}{9}}\cdot \sqrt{\max(n,m)}}{3^3 \cdot \sqrt{nm+n+m}}.$$
	Observe that
		$$\frac{\sqrt{\max(n,m)}}{\sqrt{nm+n+m}} =\frac{1 }{\sqrt{\min(m,n)+1+\frac{\min(m,n)}{\max(m,n)}}}<\frac{1}{\sqrt{\min(n,m)+1}}.$$

	It remains to prove the lower bound.
	Use the lower bound from Theorem \ref{squares} and the upper bound in Theorem \ref{lower-bound} to obtain
		$$\frac{\sqrt{nm+n+m}}{2^8\sqrt{6}(n+4)(m+4)}\frac{\sqrt{2+\max(n,m)}}{2\sqrt{2}} \leq
			\left(\frac{\Vol
				\widetilde{\Sq}_{(2,2)}^{(n,m)}}{\Vol
				\widetilde{\Pos}_{ (2,2) }^{(n,m)}}\right)^{\frac{1}{D_{\cM}}}.$$
	Note that
		$$\frac{\sqrt{(nm+n+m)(2+\max(n,m))}}{2^{10}\sqrt{3}(n+4)(m+4)}=
			\frac{\sqrt{(1+\frac{1}{m}+\frac{1}{n})(\frac{2}{nm}+\frac{1}{\min(n,m)})}}{2^{10}\sqrt{3}
			(1+\frac{4}{n})(1+\frac{4}{m})}
			> \frac{\sqrt{\frac{1}{\min(n,m)}}}{\frac{2^{10} 7^2}{3\sqrt{3}}},$$
	where the estimate in the denominator follows by
		\begin{equation*}
				(1+\frac{4}{n})(1+\frac{4}{m})
					\underbrace{\leq}_{n,m\geq 3}\frac{7^2}{3^2}.
		\end{equation*}
	This concludes the proof of Theorem \ref{cor-1}.
\end{proof}

\subsection{Comparing the volumes of cones of positive and completely positive maps}\label{ssec:thm12}

We define the \textbf{probability} $p_{n,m}$ that a randomly chosen positive map
$\Phi:\mathbb{S}_n\to \mathbb{S}_m$ is completely positive to be the ratio between the volumes of the sections $\widetilde{\Sq}^{(n,m)}_{(2,2)}$ and $\widetilde{\Pos}^{(n,m)}_{(2,2)}$ in $\cM$, i.e.,
	$$p_{n,m}=
		\frac{\Vol \widetilde{\Sq}^{(n,m)}_{(2,2)} }{\Vol \widetilde{\Pos}^{(n,m)}_{(2,2)}}.$$

\begin{corollary}\label{asimptotika}
	For $n,m \in \NN$, $n\geq 3, m\geq 3$, the probability $p_{n,m}$ that a random positive map
	$\Phi:\mathbb{S}_n\to \mathbb{S}_m$ is completely
	positive, is bounded by
		$$\left(\frac{3\sqrt{3}}{2^{10} 7^2 \sqrt{\min(n,m)}}\right)^{D_\cM}<
			p_{n,m}< \left(  \frac{2^{12}\cdot 5^2 \cdot 6^{\frac{1}{2}} \cdot 10^{\frac{2}{9}}}{3^3\cdot \sqrt{\min(n,m)+1}}  \right)^{D_\cM},$$
	where $D_\cM=\binom{n+1}{2}\binom{m+1}{2}-1=
		\frac{(n+1)n(m+1)m-4}{4}.$
	In particular, if $\min(n,m)\geq \frac{2^{25}\cdot 5^4\cdot 10^{\frac{4}{9}}}{3^5}$, then
		$$\lim_{\max(n,m)\to\infty}
			p_{n,m}=0.$$
\end{corollary}

\begin{proof}
	Corollary \ref{asimptotika} follows by the definition of $p_{n,m}$
	and Theorem \ref{cor-1}.
\end{proof}

The hyperplane $\cH^{(n,m)}_{(2,2)}$ corresponds, under our identification, to linear maps 
$\Phi:\mathbb{S}_n\to \mathbb{S}_m$ satisfying $\Tr(\Phi(I_n))=nm$ by the following proposition.

\begin{proposition} \label{hyperplane-equiv}
	Let $\Phi:\mathbb{S}_n\to \mathbb{S}_m$ be a linear map and $p_\Phi$ the corresponding
	biform. Then $p_\Phi\in \cH^{(n,m)}_{(2,2)}$ if and only if $\Tr(\Phi(I_n))=nm$.
\end{proposition}

To prove Proposition \ref{hyperplane-equiv} we need the following lemma.

\begin{lemma} \label{pom-lem}
	Let $\widetilde{\sigma}$ be a normalized Lebesgue measure on $\Sym^{n-1}$.
	Then \[\displaystyle\int_{\Sym^{n-1}} xx^t \dd \widetilde{\sigma}(x)=\frac{1}{n}I_n.\]
\end{lemma}

\begin{proof}
	Since the measure $\widetilde{\sigma}$ is rotation invariant, it follows that
	for every orthogonal matrix $Q\in M_n(\RR)$ we have 
	$$
		\int_{\Sym^{n-1}} xx^t \dd  \widetilde{\sigma}(x) 
			= \int_{\Sym^{n-1}} (Qx)(Qx)^t \dd  \widetilde{\sigma}(x)
			= \int_{\Sym^{n-1}} (Qxx^t Q^t) \dd  \widetilde{\sigma}(x)
			= Q\left(\int_{\Sym^{n-1}} xx^t \dd  \widetilde{\sigma}(x)\right) Q^t,
	$$
	where the last equality follows by linearity of $Q$.
	Thus
		$$Q\left(\int_{\Sym^{n-1}} xx^t \dd  \widetilde{\sigma}(x)\right)=
			\left(\int_{\Sym^{n-1}} xx^t \dd  \widetilde{\sigma}(x)\right)Q.$$ 
	Since orthogonal matrices span the vector space $M_n(\RR)$, 
	$\displaystyle\int_{\Sym^{n-1}} xx^t \dd \widetilde{\sigma}(x)$  commutes with every matrix
	from $M_n(\RR)$. 
	Therefore 
		$$\int_{\Sym^{n-1}} xx^t \dd  \widetilde{\sigma}(x)=\alpha I_n$$
	for some $\alpha\in \RR$.
	Now 
		$$n\alpha=\Tr\left(\int_{\Sym^{n-1}} xx^t \dd  \widetilde{\sigma}(x)\right)=
			\int_{\Sym^{n-1}} \Tr(xx^t) \dd  \widetilde{\sigma}(x)=
			\int_{\Sym^{n-1}} \Tr(x^tx) \dd  \widetilde{\sigma}(x)=
			1,$$
	where the the second equality follows by $\Tr$ being linear and the
	last equality follows by $\sigma$ being normalized.
	This proves Lemma \ref{pom-lem}.
\end{proof}

\begin{proof}[Proof of Proposition \ref{hyperplane-equiv}]
	By definition, 
		$$p_\Phi\in \cH^{(n,m)}_{(2,2)}\quad \text{if and only if}\quad 
			\int_T p_\Phi(x,y) \dd\sigma=1.$$
	We have
		\begin{eqnarray*} 
			\int_T p_\Phi(x,y) \dd\sigma
			&=&	\int_T y^t\Phi(xx^t)y \dd\sigma=
				\int_{S^{n-1}} \left(\int_{S^{m-1}}
				\Tr\left(\Phi(xx^t)yy^t\right) \dd\sigma_2(y)\right)\dd\sigma_1(x)\\
			&=& \int_{S^{n-1}} \Tr\left(\Phi(xx^t) \int_{S^{m-1}} yy^t \dd\sigma_2(y)\right)\dd\sigma_1(x)\\
			&\underbrace{=}_{\text{Lemma } \ref{pom-lem}}& 
				\frac{1}{m} \int_{S^{n-1}} \Tr\left(\Phi(xx^t)\right)\dd\sigma_1(x)=
			\frac{1}{m} \Tr\left(\Phi\left(\int_{S^{n-1}} xx^t\dd\sigma_1(x)\right)\right)\\
			&\underbrace{=}_{\text{Lemma } \ref{pom-lem}}&  \frac{1}{nm} \Tr\left(\Phi(I_n)\right),
		\end{eqnarray*}
	where the third and the fifth equality follow by linearity of the maps $\Tr$ and $\Phi$.
	Therefore
		\[p_\Phi\in \cH^{(n,m)}_{(2,2)}\quad \text{if and only if}\quad 
			\Tr\left(\Phi(I_n)\right)=nm.\qedhere\] 
\end{proof}

\subsection{Extension of the results to all real or complex matrices}\label{ext-to-real-or-complex}

In this subsection we connect linear maps on the full matrix algebra over
$\FF$ where $\FF\in \{\RR,\CC\}$ with the subspace of real biforms. This connection will translate
the question of comparing the size of the cone of completely positive maps with the size of the cone of positive maps to the question of comparing the size of the cone of sums of squares biforms with the size of the cone of positive biforms on the subspace of biforms.

	We denote by $\cL(M_n(\FF),M_m(\FF))$ the vector space of all $\ast$-linear maps
	from $M_n(\FF)$ to $M_m(\FF)$. 
	Let $\Phi^\CC:M_n(\CC)\to M_m(\CC)$ stand for the complexification of a $\ast$-linear map $\Phi:M_n(\RR)\to M_m(\RR)$, i.e.,
		$$\Phi^{\CC}(A+iB):=\Phi(A)+i\Phi(B)$$
	where $A,B\in M_n(\RR).$
It is easy to check that $\Phi^\CC$ is $\ast$-linear. 
We write
	$$\cL^\CC(M_n(\RR),M_m(\RR)):=\left\{\Phi^{\CC}\mid \Phi\in \cL(M_n(\RR),M_m(\RR))\right\}$$ 
for the real vector subspace of $\cL(M_n(\CC),M_m(\CC))$ obtained by complexifying the maps from $\cL(M_n(\RR),M_m(\RR))$.
	
		There is a natural bijection $\Gamma$ between $\ast$-linear maps $\cL(M_n(\CC),M_m(\CC))$ and symmetric multiforms 
	$\symm \CC[\z,\overline{\z},\w,\overline{\w}]_{1,1,1,1}$
	given by
		\begin{equation} \label{identification-F}
			\Gamma:\cL(M_n(\CC),M_m(\CC))\to \symm\CC[\z,\overline{\z},\w,\overline{\w}]_{1,1,1,1} ,\quad
			\Phi\mapsto p_\Phi(\z,\w):=\w^\ast \Phi(\overline\z\; \overline{\z}^\ast)\w.
		\end{equation}
	Note that 
		\begin{equation}\label{cor-real-sym}
			\Gamma\left(\cL^\CC(M_n(\RR),M_m(\RR))\right)=\symm \RR[\z,\overline{\z},\w,\overline{\w}]_{1,1,1,1}.
		\end{equation}
	Thus  $\Gamma$ converts
	 properties of $\ast$-linear maps in $\cL(M_n(\CC),M_m(\CC))$ and $\cL^\CC(M_n(\RR),M_m(\RR))$ to corresponding properties of multiforms
	in $\symm\CC[\z,\overline{\z},\w,\overline{\w}]_{1,1,1,1}$ and $\symm\RR[\z,\overline{\z},\w,\overline{\w}]_{1,1,1,1}$, respectively. 
	Positivity (resp.\ complete positivity) of a map $\Phi$ corresponds to nonnegativity (resp.\ being a sum of hermitian squares) of the polynomial $p_{\Phi}$:

\begin{proposition} \label{map-poly-F}
	Let $\Phi:M_n(\FF)\to M_m(\FF)$ be a $\ast$-linear map.
	If $\FF=\CC$, then:
		\begin{enumerate}[label={\rm(\arabic*)}]	
					\item $\Phi$ is positive iff $p_{\Phi}$ is nonnegative;
					\item\label{cp-sos-F} $\Phi$ is completely positive iff 
						$p_{\Phi}=\sum_{r}q_r^\ast q_r$ is a sum of hermitian squares with $q_m\in \CC[\z, \w]_{1,1}$.
		\end{enumerate}
	If $\FF=\RR$, then:	 
				\begin{enumerate}[label={\rm(\arabic*)}]		
					\setcounter{enumi}{2}			
					\item\label{pos-F-R} $\Phi$ is positive iff $p_{\Phi^{\CC}}|_{\RR^n\times \RR^m}$ is nonnegative. 
					\item \label{cp-sos-F-R} 
						$\Phi$ is completely positive iff $p_{\Phi^\CC}=\sum_{r}q_r^\ast q_r$ is a sum of hermitian squares with $q_m\in \RR[\z, \w]_{1,1}$.
		\end{enumerate}
\end{proposition}

\begin{proof}
	The proof of Proposition \ref{map-poly-F} is analogous to the proof of Proposition \ref{map-poly}. 	
	Since in the case $\FF=\RR$, positivity of $\Phi$ is determined on real symmetric matrices, \ref{pos-F-R} is clear. 
	Since a $\ast$-linear map $\Phi$ is cp iff the Choi matrix $[\Phi(E_{ij})]_{i,j}$ is psd \cite[Theorem 3.14]{PAU} where $E_{ij}$ stand for the matrix units, 
	$\Phi$ in \ref{cp-sos-F-R} is cp iff $\Phi^\CC$ is cp. Hence \ref{cp-sos-F-R} follows from \ref{cp-sos-F}
	by noticing that since $p_{\Phi^\CC}$ belongs to $\symm \RR[\z,\overline{\z},\w,\overline{\w}]_{1,1,1,1}$,
	we can further replace each 
	$\displaystyle q_r=\sum_{j,k}a^{(r)}_{jk}z_jw_k$ in 
	$\displaystyle p_{\Phi^\CC}=\sum_{r}q_r^\ast q_r$ with 
	$\displaystyle q_{r,1}=\sum_{j,k} \frac{a^{(r)}_{jk}+\overline{a^{(r)}_{jk}}}{2} z_jw_k$ and 
	$\displaystyle q_{r,2}=\sum_{j,k} \frac{a^{(r)}_{jk}-\overline{a^{(r)}_{jk}}}{2i} z_jw_k$ such that  
	$\displaystyle \sum_r q_r^\ast q_r=\sum_r \left(q_{r,1}^\ast q_{r,1}+q_{r,2}^\ast q_{r,2}\right)$ and $q_{r,1},q_{r,2}\in \RR[\z,\w]_{1,1}$.
\end{proof}

Let $\Pos(M_n(\FF),M_m(\FF))$ and $\CP(M_n(\FF), M_m(\FF))$ denote the cone of positive maps
and the cone of completely positive maps from $M_n(\FF)$ to $M_m(\FF)$, respectively.
By Proposition \ref{map-poly-F}, comparing the cones $\Pos(M_n(\CC),M_m(\CC))$ and $\CP(M_n(\CC), M_m(\CC))$
is equivalent to comparing the cones $\Pos_\CC$ and $\Sq_\CC$, while comparing the cones $\Pos(M_n(\RR),M_m(\RR))$ and $\CP(M_n(\RR), M_m(\RR))$
is equivalent to comparing the cones 
	$$\mathcal{P}_\RR:=\left\{p\in \symm \RR[\z,\overline \z, \w,\overline\w]_{1,1,1,1}\colon p|_{\RR^n\times \RR^m}\geq 0 \right\}$$ 
and $\Sq_\RR$. 
Since $\Pos_\RR\subset\mathcal{P}_\RR$ the upper bound for the probability of a random map from $\Pos(M_n(\RR),M_m(\RR))$ belonging to $\CP(M_n(\RR),M_m(\RR))$ 
can be obtained by comparing the cones $\Pos_\RR$ and $\Sq_\RR$. By identifying $\symm \FF[\z,\overline{\z},\w,\overline{\w}]_{1,1,1,1}$ with a subspace
$\cC_\FF$ of $\RR[\x,\y]_{2,2}$ where $\x=(x_1,\ldots,x_{2n})$ and $\y=(y_1,\ldots,y_{2m})$, comparing the cones $\Pos_\FF$ and $\Sq_\FF$ is equivalent to comparing the cones 
$\Pos_{\cC_{\FF}}$ and $\Sq_{\cC_{\FF}}$. We also  write $\mathcal{P}_{\cC_{\RR}}\subset \cC_{\RR}$ for the image of the cone $\mathcal{P}_{\RR}$ under the identification between $\symm \RR[\z,\overline{\z},\w,\overline{\w}]_{1,1,1,1}$ and $\cC_{\RR}$.

We define the probability $p_{n,m}^{\FF}$ that a randomly chosen positive map
$\Phi:M_n(\FF)\to M_m(\FF)$ is completely positive to be the ratio 
	$$p_{n,m}^\FF=\left\{\begin{array}{cl}
		\frac{\Vol \widetilde{\Sq}_{\cC_\CC} }{\Vol \widetilde{\Pos}_{\cC_\CC}},& \text{if }\FF=\CC,\\ 
		\frac{\Vol \widetilde{\Sq}_{\cC_\RR} }{\Vol \widetilde{\mathcal P}_{\cC_{\RR}}},& \text{if }\FF=\RR, 
		\end{array} \right.$$
where $\widetilde{\mathcal P}_{\cC_\RR}:=\left\{f\in \cM_{\cC_\RR}\colon f+(\sum_{i=1}^{2n}x_i^2)(\sum_{j=1}^{2m} y_j^2)\in \mathcal P_{\cC_\RR}\right\}.$

\begin{corollary}\label{asimptotika-F}
	For $n,m \in \NN$, $n\geq 3, m\geq 3$, the probability $p_{n,m}^\FF$ that a random positive map
	$\Phi:M_n(\FF)\to M_m(\FF)$ is completely
	positive, is bounded by
		$$p_{n,m}^\FF\leq 
			\frac{\Vol \widetilde{\Sq}_{\cC_\FF} }{\Vol \widetilde{\Pos}_{\cC_\FF}}< 
			\left( \left( 2^{28-\dim_{\RR}\FF} \right)^{\frac{1}{2}}\cdot 3^{-\frac{5}{2}}\cdot 5^2\cdot 10^{\frac{2}{9}}\cdot
					\frac{1}{\sqrt{\min(n,m)-\frac{1}{2}}}  \right)^{D_{\cM_{\cC_\FF}}},$$
	where 
		$D_{\cM_{\cC_{\FF}}}=\left\{  
		\begin{array}{lc} 
			n^2m^2-1,&\; \text{if }\FF=\CC,\\
			\frac{nm(nm+1)}{2}-1,&\;  \text{if }\FF=\RR.
		  \end{array} \right.$
	In particular, if 
		$\min(n,m)\geq 
			\frac{\left(2^{28-\dim_{\RR}\FF} \right) \cdot 5^4 \cdot 10^{\frac{4}{9}}}{3^{5}}$, then 
		$$\lim_{\max(n,m)\to\infty}
			p_{n,m}^\FF=0.$$
\end{corollary}

\begin{proof}
	By combining the lower bound in Theorem \ref{lower-bound-F} with the upper bound in Theorem \ref{squares-F} 
	as in the proof of Theorem \ref{cor-1} 
	and observing that 
		$$\frac{\sqrt{\max(n,m)}}{\sqrt{nm-1}} =\frac{1 }{\sqrt{\min(m,n)-\frac{1}{\max(n,m)}}}<\frac{1}{\sqrt{\min(m,n)-\frac{1}{2}}},$$
	Corollary \ref{asimptotika-F} follows by the definition of $p_{n,m}^\FF$.\qedhere
\end{proof}

%

\section{Constructing positive maps that are not completely positive}\label{sec:algo}

By Proposition \ref{map-poly},
 each biform $f\in \RR[\x,\y]_{2, 2}$ that is positive but 	not a sum of squares yields
 an example of a  positive map $\Phi:\Sym_n\to \Sym_m$ that is not completely positive.
By Proposition \ref{ext-of-pos-to-F} below, all extensions of $\Phi$ to $M_n(\RR)$ and the complexification of the trivial extension are positive but not completely positive. 
In this section we specialize
the Blekherman-Smith-Velasco algorithm \cite[Procedure 3.3]{BSV} to produce many examples of positive biforms of bidegree (2,2) that are not sums of squares.

\subsection{Extending positive maps from real symmetric matrices to the full matrix algebra $M_n(\FF)$,  $\FF\in \{\RR,\CC\}$}

 Let $\mathbb{K}_n$ be the vector space of real antisymmetric $n\times n$ matrices, i.e.,
    $$\mathbb{K}_n=\left\{A\in M_n(\RR) \mid A^\ast=-A\right\}.$$
    The vector space $M_n(\RR)$ can be expressed as the direct sum
        $$M_n(\RR)=\Sym_n\oplus \mathbb{K}_n,$$
    and a $\ast$-linear map $\Phi:M_n(\RR)\to M_m(\RR)$  uniquely decomposes as a direct sum
        $$\Phi=\Phi|_{\Sym_n}\oplus \Phi|_{\mathbb{K}_n},$$
    where
        $$\Phi|_{\Sym_n}:\Sym_n\to\Sym_m\quad \text{and}\quad
            \Phi|_{\mathbb{K}_n}:\mathbb{K}_n\to \mathbb{K}_m$$
    are the restrictions of $\Phi$ to $\Sym_n$ and $\mathbb{K}_n$, respectively.
  Conversely, given linear maps  $\Phi:\Sym_n\to\Sym_m$ and $\Psi:\mathbb{K}_n\to\mathbb{K}_m$, the map
 $\Gamma:=\Phi\oplus\Psi:M_n(\RR)\to M_m(\RR)$ defined by $\Gamma(S+A):=\Phi(A)+\Psi(A)$ is readily seen to be $\ast$-linear.
 	
	Recall that the complexification $\Phi^\CC:M_n(\CC)\to M_m(\CC)$ of a $\ast$-linear map $\Phi:M_n(\RR)\to M_m(\RR)$ is defined by
		$$\Phi^{\CC}(A+iB):=\Phi(A)+i\Phi(B)$$
	where $A,B\in M_n(\RR).$

\begin{proposition}\label{ext-of-pos-to-F}
	Let $\Phi:\Sym_n\to\Sym_m$ be a positive but not completely positive map and $\Psi:\mathbb{K}_n\to\mathbb{K}_m$ a linear map.
	Then:
	\begin{enumerate}[label=\rm{(\arabic*)}]
		\item\label{ext-of-pos-to-F-1} The map $\Gamma:=\Phi\oplus \Psi:M_n(\RR)\to M_m(\RR)$ is positive but not completely positive.
		\item\label{ext-of-pos-to-F-2}	 Let $\mathbf{0}:\mathbb{K}_n\to\mathbb{K}_m$ be the trivial map, i.e., $\mathbf{0}(A)=0$ for all $A\in \mathbb{K}_n$.
			The map $(\Phi\oplus\mathbf 0)^\CC:M_n(\CC)\to M_m(\CC)$ is positive but not completely positive. 
	\end{enumerate}
\end{proposition}
 
 \begin{proof}
 		To prove \ref{ext-of-pos-to-F-1} it suffices to observe that $\Gamma$ is positive iff its restriction $\Gamma|_{\Sym_n}=\Phi$ is positive and that 
	$\Gamma$ being cp would imply that $\Phi$ is cp. 
	As in the proof of Proposition \ref{map-poly-F} note that $(\Phi\oplus\mathbf 0)^\CC$ is cp iff $\Phi\oplus\mathbf 0$ is cp. Thus to
	prove \ref{ext-of-pos-to-F-2} it only remains to show that 
	$(\Phi\oplus\mathbf 0)^\CC(X)$ is psd for all psd matrices $X\in M_n(\CC)$. Decompose a psd matrix $X$ as $X=X_{\re}+iX_{\im}$ where $X_{\re}, X_{\im}\in M_n(\RR)$.
	Since $X=X^\ast$, it follows that $X_{\re}\in \Sym_n$ and $X_{\im}\in \mathbb{K}_n$. 
	For all $v\in \RR^n$ we have that
		$v^\ast X_{\re} v=v^\ast Xv\geq 0.$
	Hence $X_{\re}$ is psd. Thus $(\Phi\oplus\mathbf 0)^\CC(X)=\Phi(X_{\re})$ is psd which concludes the proof of \ref{ext-of-pos-to-F-2}.
 \end{proof}

\subsection{Specialization of the Blekherman-Smith-Velasco algorithm}

To use \cite[Procedure 3.3]{BSV} we have to observe first that biquadratic forms are in bijective correspondence with quadratic forms on the Segre variety (see \cite[Example 5.6]{BSV}). Indeed, let
\begin{eqnarray*}
		 \sigma_{n,m}:\PP^{n-1}\times \PP^{m-1} &\to& \PP^{nm-1},\\
	([x_1:\ldots:x_n], [y_1:\ldots:y_m]) &\mapsto& [x_1y_1:x_1y_2:\ldots: x_1y_m:\ldots:x_ny_m].	\end{eqnarray*}
be the Segre embedding. Its image $\sigma _{n,m}(\PP^{n-1}\times \PP^{m-1})$ is the zero locus of the ideal $I_{n,m}\subseteq \RR[z_{11},z_{12},\ldots ,z_{1m},\ldots ,z_{nm}]$ generated by all $2\times 2$ minors of the matrix $\left(z_{ij}\right)_{i,j}$. Moreover, the ideal $I_{n,m}$ is radical and consists of all polynomials vanishing on $\sigma_{n,m}(\PP^{n-1}\times \PP^{m-1})$ \cite[p.\ 98]{Har92}. It is also well known that $\sigma_{n,m}(\PP^{n-1}\times \PP^{m-1})$ is smooth (being the determinantal variety of all $n\times m$ matrices of rank at most $1$) \cite[p.\ 184-185]{Har92} and that its degree equals $\binom{n+m-2}{n-1}$ \cite[p.\ 233]{Har92}. We write $V(I_{n,m})$ for the image of the Segre embedding $\sigma _{n,m}(\PP^{n-1}\times \PP^{m-1})$, i.e.,
$$V(I_{n,m})=
		\{[z_{11}:\ldots:z_{nm}]\in \PP^{nm-1}\colon f(\z)=0\;\text{for every }f\in I_{n,m}\},$$
where
$$\z=(z_{11},z_{12},\ldots, z_{1m}, \ldots,z_{nm}),$$
and $V_\RR(I_{n,m})$ for the subset of its real points.

Since $I_{n,m}$ is the homogeneous ideal of all polynomials that vanish on $V(I_{n,m})$, the quotient ring $\CC[\z]/{I_{n,m}}$ is the coordinate ring $\CC[V(I_{n,m})]$ of the variety $V(I_{n,m})$. Moreover, the Segre embedding $\sigma_{n,m}$ induces the injective ring homomorphism $\sigma_{n,m}^\#:\CC[\z]/I_{n,m}\to \CC[\x,\y]$ satisfying $\sigma_{n,m}^\#(z_{ij}+I_{n,m}) =x_iy_j$ for $1\le i\le n,1\le j\le m$. The restriction of $\sigma_{n,m}^\#$ to the real quadratic forms is then a (linear) bijective correspondence between quadratic forms from $\RR[\z]/{I_{n,m}}$ and biforms from $\RR[\x,\y]_{2, 2}$.

\begin{lemma} \label{prehod-na-zij}
		Let  $f\in \RR[\x,\y]_{2, 2}$ be a biform of bidegree (2,2). Then:
	\begin{enumerate}[label={\rm(\arabic*)}]
		\item\label{pt-1-prehod} If $f\in \RR[\x,\y]_{2, 2}$ is a sum of squares, then it is a sum of squares of biforms from $\RR[\x,\y]_{1, 1}$.
		\item\label{pt-2-prehod} The biform $f\in \RR[\x,\y]_{2, 2}$ is a sum of squares if and only if the quadratic form ${\sigma_{n,m}^\#}^{-1}(f)\in \RR[\z]/I_{n,m}$ is a sum of squares.
	\end{enumerate}
\end{lemma}

\begin{proof}
	First we prove \ref{pt-1-prehod}. We have
		\begin{equation}\label{pt-1-proof}
			f=\sum_{i=1}^{i_0} \left(\sum_{j=0}^{j_i}\sum_{k=0}^{k_i} f_{ijk} \right)^2,
		\end{equation}
	where $i_0\in \NN$, $j_i, k_i\in \NN\cup\{0\}$ and $f_{ijk}\in \RR[\x,\y]_{j,k}$ are biforms of bidegree $(j,k)$.
	Let $f_{jk}$ be the bihomogenous part of $f$ of bidegree $(j,k)$.
	Then
		$$f=f_{22}=\sum_{i=1}^{i_0} \sum_{j=0}^2 \sum_{k=0}^2 f_{ijk}f_{i(2-j)(2-k)}.$$
	Since $f_{j0}=f_{0k}=0$ for every $j,k\in  \NN\cup\{0\}$, it follows from (\ref{pt-1-proof}) that 
	$f_{ij0}=f_{j0k}=0$ for each $i,j,k$. 
	Hence
		\begin{equation}\label{exp-pt-1}
			f=\sum_{i=1}^{i_0}  f_{i11}^2,
		\end{equation}
	which proves \ref{pt-1-prehod}.
		
	To prove the implication $(\Rightarrow)$ of \ref{pt-2-prehod} note that all $f_{i11}$ from $(\ref{exp-pt-1})$
	are in the image of $\sigma_{n,m}^\#$. Hence 
		${\sigma_{n,m}^\#}^{-1}(f)=\sum_{i=1}^{i_0}{\sigma_{n,m}^\#}^{-1}(f_{i11})^2$ 
	is a sum of squares. 
	It remains to prove the implication 
	$(\Leftarrow)$ of \ref{pt-2-prehod}. Since $f$ is in the image of $\sigma_{n,m}^\#$ it follows from
		$${\sigma_{n,m}^\#}^{-1}(f)=\sum_{i=1}^{i_1} [h_i]^2,$$
	where $i_1\in \NN$ and $[h_i]$ is the equivalence class of $h_i\in \RR[\z]$ in $\RR[\z]/I_{n,m}$, that
		$$f=\sum_{i=1}^{i_1} \sigma_{n,m}^\#([h_i])^2$$
	which proves $(\Leftarrow)$ of \ref{pt-2-prehod}.
\end{proof}

 We write
	\begin{eqnarray*}
		\Pos(V_\RR(I_{n,m})) &=& \left\{f\in \RR[\z]/I_{n,m}\colon f(z)\geq 0\quad \text{for all } z \in V_\RR(I_{n,m})
			\right\},\\
		\Sq(V_\RR(I_{n,m})) &=& \{f\in \RR[\z]/I_{n,m}\colon f=\sum_{i} f_i^2\quad \text{for some }
			f_i \in \RR[\z]/I_{n,m} \},
	\end{eqnarray*}
for the cone of nonnegative polynomials and the cone of sums of squares from $\RR[\z]/I_{n,m}$, respectively.

 For $n>2, m>2$, \cite[Procedure 3.3]{BSV} is an explicit construction of nonnegative quadratic forms from $\RR[\z]/{I_{n,m}}$ that are not sums of squares forms from random input data.
We now present this procedure specialized to our context
of biquadratic biforms.

\subsection{Algorithm}

\begin{algorithm}\label{algo}\rm
	Let $n>2$, $m>2$,
		$$d=n+m-2=\dim\sigma_{n,m}(\PP^{n-1}\times \PP^{m-1}) \;\text{ and }\; e=(n-1)(m-1)=\codim \sigma_{n,m}(\PP^{n-1}\times \PP^{m-1}).$$
	 To obtain a quadratic form in $ \Pos(V_{\RR}(I_{n,m}))\setminus \Sq(V_{\RR}(I_{n,m}))$ proceed as follows:\\

	\begin{enumerate}[label={\rm Step \arabic*}]
	\item\label{it:1} Construction of linear forms $h_0,\ldots, h_d$.
	\begin{enumerate}[label={\rm Step 1.\arabic*}]
		\item\label{it:1.1} 
			Choose $e+1$ random points $x^{(i)}\in \RR^n$ and $y^{(i)} \in \RR^m$ and calculate their Kronecker tensor products
				$z^{(i)}=x^{(i)}\otimes y^{(i)}\in \RR^{nm}$.
		\item\label{it:1.2} 
			Choose $d$ random vectors $v_1,\ldots v_d \in \RR^{nm}$ from the kernel of the matrix 
			$$\begin{pmatrix} z^{(1)} & \ldots & z^{(e+1)} \end{pmatrix}^\ast.$$
			The corresponding linear forms $h_1,\ldots, h_d$ are
				$$h_j(\z)=v_j^\ast\cdot \z  \in \RR[\z] \quad \text{for } j=1,\ldots,d.$$
			If the number of points in the intersection
				\begin{equation*}
					\ker(
					\begin{pmatrix}
					v_1 & \ldots & v_d
					\end{pmatrix}^\ast)
					\bigcap V(I_{n,m})
				\end{equation*}
			is not equal to $\deg (V(I_{n,m}))=\binom{n+m-2}{n-1}$ or if the points in the intersection are not in linearly general position, then repeat \ref{it:1.1}.
		\item\label{it:1.3}
			Choose a random vector $v_0$ from the kernel of the matrix 
			$$\begin{pmatrix} z^{(1)} & \ldots & z^{(e)} \end{pmatrix}^\ast.$$
			(Note that we have omitted  $z^{(e+1)}$.)
			The corresponding linear form $h_0$ is
				$$h_0(\z)=v_0^\ast\cdot \z \in \RR[\z].$$
			If $h_0$ intersects  $h_1$, $\ldots$, $h_d$ in more than
			$e$ points on $V(I_{n,m})$, then repeat \ref{it:1.3}.
\end{enumerate}

\smallskip
		\noindent Let $\mathfrak a$ be the ideal in $\RR[\z]/{I_{n,m}}$  generated by $h_0,h_1,\ldots,h_{d}$.

		\smallskip
	\item\label{it:2} Construction of a quadratic form $f\in  \left(\RR[\z]/{I_{n,m}}\right)\setminus \mathfrak a^2$.
	\begin{enumerate}[label={\rm Step 2.\arabic*}]
		
		\item\label{it:2.1}
		Let $g_1(\z),\ldots ,g_{\binom{n}{2}\binom{m}{2}}(\z)$ be the generators of the ideal $I_{n,m}$, i.e., the $2\times 2$ minors $z_{ij}z_{kl}-z_{il}z_{kj}$ for $1\le i<k\le n,1\le j<l\le m$. For each $i=1,\ldots ,e$ compute a basis $\{w_1^{(i)},\ldots ,w_{d+1}^{(i)}\}\subseteq \RR^{nm}$ of the kernel of the matrix
		$$\begin{pmatrix}
		\nabla g_1(z^{(i)})^* \\ \vdots \\ \nabla g_{\binom{n}{2}\binom{m}{2}}(z^{(i)})^*
		\end{pmatrix}.$$
		(Note that this kernel is always $(d+1)$-dimensional, since the variety $V(I_{n,m})$ is $d$-dimensional (in $\PP^{nm-1}$) and smooth.)
		
		\item\label{it:2.2}
		
		Let $\e_i$ denote the $i$-th standard basis vector of the corresponding vector space, i.e., the vector with 1 on the $i$-th component and 0 elsewhere. Choose a random vector $v\in \RR^{n^2m^2}$ from the intersection of the 
		kernels of the matrices 
		$$\begin{pmatrix}
		z^{(i)} \otimes w_1^{(i)} & \cdots & z^{(i)}\otimes w_{d+1}^{(i)}
			\end{pmatrix}^*\quad \text{for }i=1,\ldots,e
		$$
		with the kernels of the matrices
		$$\begin{pmatrix}
		\e_i\otimes \e_j-\e_j\otimes \e_i
		\end{pmatrix}^*\quad \text{for } 1\le i<j\le nm.$$
		(The latter condition ensures $v$ is a symmetric tensor in $\RR^{nm}\otimes \RR^{nm}$. Note also that we have omitted the point $z^{(e+1)}$.)
		
		For $1\leq i,k\leq n$ and $1\leq j,l\leq m$	 denote
		$$E_{ijkl}=(\e_i\otimes \e_j)\otimes (\e_k\otimes \e_l)+(\e_k\otimes \e_l)\otimes (\e_i\otimes \e_j)\in \mathbb{R}^{n^2m^2}.$$
			If $v$ is in
				$$\Span\big(\left\{v_i\otimes v_j+v_j\otimes v_i\colon 0\leq i\leq j\leq d\right\}\bigcup
				\left\{E_{ijkl}-E_{ilkj};1\leq i<k\leq n,1\leq j<l\leq m\right\}\big),$$
			then repeat {\rm\ref{it:2.2}}.
			Otherwise the corresponding quadratic form $f$
					$$f(\z)=v^\ast\cdot (\z\otimes\z)\in \RR[\z]/I_{n,m},$$
			does not belong to $\mathfrak a^2.$
	\end{enumerate}

	\item\label{it:3} Construction of a quadratic form in $\mathbb{R}[\z]/I_{n,m}$
	 that is positive  but not a sum of squares.\\[1mm]
	 Calculate the greatest $\delta_0>0$ such that
				$\delta_0 f+\sum_{i=0}^dh_i^2$
			is nonnegative on $V_{\RR}(I_{n,m})$. Then for every $0<\delta<\delta_0$
			the quadratic form
				$$(\delta f+\sum_{i=0}^dh_i^2)(\z)$$
			is nonnegative on $V_{\RR}(I_{n,m})$ but is not a sum of squares.
	\end{enumerate}
\end{algorithm}

\subsection{Correctness of Algorithm \ref{algo}}
The main ingredient in the proof is the theory of minimal degree varieties as developed in \cite{BSV}.
Since the Segre variety $\sigma_{n,m}(\PP^{n-1}\times\PP^{m-1})$ is not of minimal degree for $n,m\geq 3$
\cite[Example 5.6]{BSV}, $\Sq(V_\RR(I_{n,m}))\subsetneq \Pos(V_\RR(I_{n,m}))$. Hence results of \cite[Section 3]{BSV} apply; their Procedure 3.3 adapted to our set-up is Algorithm \ref{algo}. While \ref{it:1} and \ref{it:3} follow immediately from the corresponding steps in \cite[Procedure 3.3]{BSV}, we note for \ref{it:2} that ``vanishing to the second order at $z^{(i)}$'' means $f(z^{(i)})=0$ and $\nabla f(z^{(i)})\in \Span\left\{\nabla g_j(z^{(i)})\colon 1\leq j\leq \binom{n}{2}\binom{m}{2}\right\}$. Moreover, the former step is redundant, as the relation
$$\nabla \left(f-\sum _{j=1}^{\binom{n}{2}\binom{m}{2}}\lambda_jg_j\right)(z^{(i)})=0$$
together with the well-known identity $2q(\z)=(\nabla q(\z))^*\z$ for any quadratic form $q$ immediately yields $f(z^{(i)})=0$, since $z^{(i)}\in V(I_{n,m})$. The quadratic form $\delta f+\sum _{i=0}^dh_i^2$ is never a sum of squares, since $f\not \in \mathfrak a^2$, while it is nonnegative on $V_{\RR}(I_{n,m})$ for sufficiently small $\delta >0$ by the positive definiteness of the Hessian of $\sum _{i=0}^dh_i^2$ at its real zeros $z^{(1)},\ldots ,z^{(e)}$, see the proof of the correctness of Procedure 3.3 in \cite{BSV}. We note that the verification in \ref{it:1.2} is computationally difficult, but since all steps in the algorithm are performed with random data, all the generic conditions from \cite[Procedure 3.3]{BSV} are satisfied with probability 1. Hence, Algorithm \ref{algo} works well with probability 1 without implementing verifications.

\subsection{Implementation and rationalization}\label{ssec:rat}

\ref{it:1} and  \ref{it:2} are easily implemented as
they only require linear algebra.
(The verification in \ref{it:1.2} can be performed
using Gr\"obner basis if $m,n$ are small, but is
``always'' satisfied with random input data.)
On the other hand, \ref{it:3} is computationally difficult; testing nonnegativity even of low degree polynomials is NP-hard, cf.~\cite{LNQY09}.
We thus employ a sum of squares relaxation technique
motivated by (the solution to) Hilbert's 17th problem \cite{BCR98}.
Consider the following polynomial optimization problem:
find the maximal $\delta_0$ such that
\begin{equation}\label{eq:poliSos}
\sigma_{n,m}^\#\Big(\delta_0 f+\sum_{i=0}^dh_i^2\Big) \Big( \sum_{j,k}\big( x_jy_k\big)^2\Big)^{\ell}
\qquad\text{is a sum of squares.}
\end{equation}

\def\de{\delta}
For a given $\ell\in\NN$ the condition
\eqref{eq:poliSos} can be converted
to a linear matrix inequality using  Gram matrices of
polynomials. Thus maximizing $\delta_0$ subject to this
constraint is a standard semidefinite programming problem (SDP) \cite{WSV00}. We start by solving \eqref{eq:poliSos} for $\ell=1$ using one of the standard solvers.
If the obtained maximum is  $\de_0=0$, then we increase $\ell$ and solve another SDP. 
We repeat this until we obtain a maximum $\de_0>0$.
In fact, in our numerical experiments this always
happened with $\ell=1$ already.

Any $\de_0>0$ for that \eqref{eq:poliSos}
 holds gives an example of a positive biquadratic
 biform that is not a sum of squares. Together
 with Proposition \ref{map-poly} this yields
 instances of positive but not completely positive maps.

\subsubsection{Rationalization}\label{sss:rat}
\ref{it:1} and \ref{it:2} can be performed over $\QQ$,
leading to rational forms $h_j,f$. But
in \ref{it:3} of the algorithm
we are using SDPs,
so the output $\delta_0$ will be floating point.
Pick a positive rational $\delta<\delta_0$. 
We now explain how tools from polynomial optimization (\cite{PP08,CKP15}) can be used to provide 
an exact, symbolic certificate of positivity
for the produced form $\delta f+\sum_{i=0}^dh_i^2$ by computing a
positive semidefinite
rational Gram matrix $G$ for
$\sigma_{n,m}^\#\Big(\delta_0f+\sum_{i=0}^dh_i^2\Big) \Big( \sum_{j,k}\big( x_jy_k\big)^2\Big)^{\ell}$.
That is,  letting $p=\sigma_{n,m}^\#\Big(\delta_0f+\sum_{i=0}^dh_i^2\Big)$,
\begin{equation}\label{eq:gramG}
p \Big( \sum_{j,k}\big( x_jy_k\big)^2\Big)^{\ell} = W^* G W
\end{equation}
where $W=W(x,y)$ is the bihomogeneous vector $(x^I y^J)_{|I|=|J|=\ell+1}$.
Since $p(x^{(i)},y^{(i)})=0$ for $i\leq e$, each positive
semidefinite $G$ satisfying \eqref{eq:gramG} will have
at least an $e$-dimensional nullspace. Let $P$
be a change of basis matrix containing 
the vectors $W(x^{(i)},y^{(i)})$, $i\leq e$,
as the first  $e$
 columns and a (rational) basis for the orthogonal complement as its remaining columns.
With respect to this decomposition, write
\[
P^*GP = \begin{bmatrix} \check G_{11} & \check G_{12}\\
\check G_{12}^* & \check G_{22}\end{bmatrix}.
\]
By construction, we want $\check G_{11}$ and $\check G_{12}$ to be equal to $0$. Solve these linear equations and use them in $\check G_{22}$ to produce $\check G$. Then run
a SDP to solve $\check G\succeq0$. Use the  trivial objective function, since under a strict feasibility assumption the interior point
methods (which all state-of-the-art SDP solvers use) yield solutions in the relative interior of the optimal face, leading to solutions of maximal rank \cite{LSZ98}. 
If the output of the SDP is a full rank floating point $\check G$,
simply use a 
rationalization that is fine enough to yield 
 a positive semidefinite matrix (cf.~\cite{PP08}).

\subsection{Example}
In this subsection we give an explicit example
of a positive map that is
not completely positive built off Algorithm \ref{algo}.
Let 
\begin{multline*}
p_\Phi(x,y)=
104 x_1^2 y_1^2+283 x_1^2 y_2^2+18 x_1^2 y_3^2-310 x_1^2 y_1
   y_2+18 x_1^2 y_1 y_3+4 x_1^2 y_2 y_3+310 x_1 x_2 y_1^2 \\ -18
   x_1 x_3 y_1^2-16 x_1 x_2 y_2^2+52 x_1 x_3 y_2^2+4 x_1 x_2
   y_3^2-26 x_1 x_3 y_3^2-610 x_1 x_2 y_1 y_2-44 x_1 x_3 y_1
   y_2  \\ +36 x_1 x_2 y_1 y_3-200 x_1 x_3 y_1 y_3-44 x_1 x_2 y_2
   y_3+322 x_1 x_3 y_2 y_3+285 x_2^2 y_1^2+16 x_3^2 y_1^2+4
   x_2 x_3 y_1^2\\ +63 x_2^2 y_2^2+9 x_3^2 y_2^2+20 x_2 x_3
   y_2^2+7 x_2^2 y_3^2+125 x_3^2 y_3^2-20 x_2 x_3 y_3^2+16
   x_2^2 y_1 y_2+4 x_3^2 y_1 y_2-60 x_2 x_3 y_1 y_2\\
   +52 x_2^2
   y_1 y_3+26 x_3^2 y_1 y_3-330 x_2 x_3 y_1 y_3-20 x_2^2 y_2
   y_3+20 x_3^2 y_2 y_3-100 x_2 x_3 y_2 y_3.   \end{multline*}
The corresponding linear map $\Phi:\Sym_3\to\Sym_3$
is as follows:
\[
\Phi(E_{11})=
\begin{bmatrix}
 104 & -155 & 9 \\
 -155 & 283 & 2 \\
 9 & 2 & 18 \\
 \end{bmatrix},
 \quad
 \Phi(E_{22})=
\begin{bmatrix}
  285 & 8 & 26 \\
 8 & 63 & -10 \\
 26 & -10 & 7 \\
 \end{bmatrix},
 \]
 \[
 \Phi(E_{33})=
\begin{bmatrix}
  16 & 2 & 13 \\
 2 & 9 & 10 \\
 13 & 10 & 125 \\
 \end{bmatrix},
\quad
\Phi(E_{12}+E_{21})=
\begin{bmatrix}
  310 & -305 & 18 \\
 -305 & -16 & -22 \\
 18 & -22 & 4 \\
 \end{bmatrix},
 \]
 \[
\Phi(E_{13}+E_{31})=
 \begin{bmatrix}
 -18 & -22 & -100 \\
 -22 & 52 & 161 \\
 -100 & 161 & -26 \\
 \end{bmatrix},
\quad
 \Phi(E_{23}+E_{32})=
\begin{bmatrix}
   4 & -30 & -165 \\
 -30 & 20 & -50 \\
 -165 & -50 & -20 \\
  \end{bmatrix}.
\]
We claim that $p_\Phi$ is nonnegative but not a sum of squares. Equivalently, $\Phi$ is positive but not cp.
We will establish this by explaining how this example was produced using Algorithm \ref{algo}.

Start with the points
\[
\begin{bmatrix} x^{(1)} & y^{(1)} \\
x^{(2)} & y^{(2)}\\
x^{(3)} & y^{(3)}\\
x^{(4)} & y^{(4)}\\
x^{(5)} & y^{(5)}
\end{bmatrix}=
\left[
\begin{array}{rrr|rrr}
1 & 1 & -1 & 1 & 1 & -1 \\
 1 & -1 & 1 & 1 & -1 & 1 \\
 -1 & 1 & 1 & -1 & 1 & 1 \\
 1 & 1 & 1 & 1 & 1 & 1 \\
 2 & -3 & 3 & -2 & 0 & 2 \\
\end{array}
\right],
\]
where each $x^{(i)},y^{(i)}\in\RR^3$. Find some random
linear forms $h_j$ from \ref{it:1}, e.g., using
\[
\begin{bmatrix} v_0^* \\
v_1^* \\
v_2^* \\
v_3^* \\
v_4^* \\
\end{bmatrix}=
\left[
\begin{array}{rrrrrrrrr}
-2 & 2 & -1 & -2 & 0 & 0 & 1 & 0 & 2 \\
 0 & 2 & 3 & -2 & 3 & 0 & -3 & 0 & -3 \\
 -3 & 7 & 0 & -7 & -3 & 1 & 0 & -1 & 6 \\
 9 & -14 & 0 & 14 & -3 & 2 & 0 & -2 & -6 \\
 0 & 6 & 0 & -6 & -6 & 0 & 0 & 0 & 6 \\
\end{array}
\right].
\]
Finally, a random 
quadratic form (in $\z$) $f$ satisfying the conditions described in \ref{it:2} is
\begin{multline*}
\sigma_{3,3}^\#(f)= 5 x_1^2 y_1^2-3 x_1^2 y_2^2+4 x_1^2 y_3^2-4 x_1^2 y_1 y_2+7
   x_1^2 y_1 y_3-2 x_1^2 y_2 y_3+4 x_1 x_2 y_1^2-7 x_1 x_3
   y_1^2+x_1 x_2 y_2^2\\ +5 x_1 x_3 y_2^2+2 x_1 x_2 y_3^2-2 x_1
   x_3 y_3^2+2 x_1 x_2 y_1 y_2-3 x_1 x_3 y_1 y_2+7 x_1 x_2 y_1
   y_3-14 x_1 x_3 y_1 y_3-10 x_1 x_2 y_2 y_3\\
   +x_1 x_3 y_2 y_3-2
   x_2^2 y_1^2+3 x_3^2 y_1^2-2 x_2 x_3 y_1^2+2 x_3^2 y_2^2+x_2
   x_3 y_2^2+x_2^2 y_3^2+2 x_3^2 y_3^2-4 x_2 x_3 y_3^2-x_2^2
   y_1 y_2\\ +2 x_3^2 y_1 y_2+5 x_2^2 y_1 y_3+2 x_3^2 y_1 y_3-5
   x_2 x_3 y_1 y_3-x_2^2 y_2 y_3+4 x_3^2 y_2 y_3.
   \end{multline*}
Next run the SDP maximizing $\delta_0$
subject to ``$\sigma_{3,3}^\#(\sum_{i=0}^4 h_i^2+\delta_0 f)\sum_{j,k} (x_jy_k)^2$
is a sum of squares''. The optimal objective value is
$\delta_0\approx3.41628$. Choosing $\delta=2$, let
\[
p=\sigma_{3,3}^\#\Big(\sum_{i=0}^4 h_i^2+2 f\Big).
\]
Then $p=p_\Phi$. As explained above, $p$ is not a sum of squares, whence $\Phi$ is not completely positive. 
Alternately, a SDP can be used to compute an explicit
example of a linear functional positive on sum of squares and negative on $p$.

Finally, we used the rationalization procedure
described in Subsection \ref{sss:rat} above
to prove $p$ is nonnegative (with $\ell=1$).
We provide a Mathematica notebook\footnote{see \url{https://www.math.auckland.ac.nz/~igorklep/} or the arXiv source of this manuscript} 
where the interested reader can verify the calculations.

\appendix

\section{Sums of products of even powers of linear forms}

In \cite{Blek1} Blekherman also estimated the volume of the section of the cone of sums of even powers of linear forms. The tools developed in this article can be used to extend his result to the cone of sums of products of even powers of linear forms in different sets of variables. The main result of this appendix, Theorem \ref{lin-forms} below, provides bounds for the volume of the section of this cone.

Let $\Lf^{(n,m)}_{(2k_1,2k_2)}$ stand for the cone generated by the products of the form $\ell(\x)^{2k_1}\ell^\prime(\y)^{2k_2}$ where $\ell(\x)$ and $\ell^\prime(\y)$
are linear forms in $\x:=(x_1,\ldots,x_n)$ and $\y:=(y_1,\ldots,y_m)$, respectively, i.e., 
	$$\Lf^{(n,m)}_{(2k_1,2k_2)}:=\left\{ f\in \RR[\x,\y]_{2k_1,2k_2}\colon f= \sum_{i} \ell_i^{2k_1}{\ell_i^\prime}^{2k_2}
		\quad \text{with }\ell_i\in \RR[\x]_1,\; \ell^\prime_i\in \RR[\y]_1\right\},$$
where $\RR[\x]_1$ and $\RR[\y]_1$ stand for the vector spaces of linear forms in $\x$ and $\y$, respectively.

Recall the definitions of the product measure $\sigma$ from Subsection \ref{sec-main-res} and the vector space $\cM:=\cM^{(n,m)}_{(2,2)}$ from (\ref{M-def}).
Equip $\cM$ with the $L^2(\sigma)$ inner product and let $B_\cM$ be the unit ball in $\cM$.
Write $D_{\cM}$ for the dimension of $\cM$ and let $\mu$ be the (unique w.r.t.\ unitary isomorphism) pushforward of the Lebesgue measure on $\RR^{D_{\cM}}$ to $\cM$
(see Lemma \ref{unique-pushforward}).
Let $\widetilde{\Lf}^{(n,m)}_{(2k_1,2k_2)}$ be the set
	\begin{align*}
		\widetilde{\Lf}^{(n,m)}_{(2k_1,2k_2)}&:=\left\{f\in \cM \colon f+(\sum_{i=1}^{n}x_i^2)^{k_1}(\sum_{j=1}^{m} y_j^2)^{k_2}\in \Lf^{(n,m)}_{(2k_1,2k_2)}\right\}.
	\end{align*}
The bounds for the volume of the set $\widetilde{\Lf}^{(n,m)}_{(2k_1,2k_2)}$ are as follows.

\begin{theorem}\label{lin-forms}
	For $n,m\in \NN$ we have:
		$$
			h_{2k_1,2k_2}\leq
			\left(\frac{\Vol \widetilde{\Lf}^{(n,m)}_{(2k_1,2k_2)} }{\Vol B_{\cM} }\right)^{\frac{1}{D_{\cM}}}
			\leq j_{2k_1,2k_2},$$
	where
	\begin{align*}
		h_{2k_1,2k_2}&=\frac{1}{2}\max\left(\frac{2k_1^2+n}{2k_1^2},\frac{2k_2^2+m}{2k_2^2}\right)^{\frac{1}{2}}\frac{k_1!k_2!}{(\frac{n}{2}+2k_1)^{k_1} (\frac{m}{2}+2k_2)^{k_2}},\\
		j_{2k_1,2k_2}&=\frac{1}{c_{2k_1,2k_2}}\left(\frac{k_1!k_2!}{(\frac{n}{2}+2k_1)^{k_1} (\frac{m}{2}+2k_2)^{k_2}}\right)^{\alpha_{2k_1,2k_2}},
	\end{align*}
	and
	\begin{align*}
		c_{2k_1,2k_2}&=\ds
			\left\{\begin{array}{lr}
				3^{3}\cdot 10^{-\frac{20}{9}} \max(n,m)^{-\frac{1}{2}},& \text{if } k_1=k_2=1,\\[1mm]
				\exp(-3) \left(2\lceil\max(n,m)\ln(2\max(k_1,k_2)+1)\rceil\right)^{-\frac{1}{2}} ,& \text{otherwise,}
			\end{array}\right.\\
		\alpha_{2k_1,2k_2}&=
			1-\left(\frac{2k_1-1}{n+2k_1-1}\right)^2-\left(\frac{2k_2-1}{m+2k_2-1}\right)^2+
			\left(\frac{2k_1}{n+2k_1-2} \frac{2k_2}{m+2k_2-2}\right)^2.
	\end{align*}
\end{theorem}

The proof of Theorem \ref{lin-forms} closely follows the proof of \cite[Theorem 7.1]{Blek1} which gives volume bounds for the cone generated by $2k$-th powers of linear forms in $\x$. 
We will need the following lemma.

\begin{lemma}\label{closed-cones}
	The sets $\Lf^{(n,m)}_{(2k_1,2k_2)}$ and $\Pos^{(n,m)}_{(2k_1,2k_2)}$ are closed in the apolar inner product on $\RR[\x,\y]_{2k_1,2k_2}$.
\end{lemma}

Let $S^{n-1}$ be the unit sphere in $\RR^n$. For a point $v:=(v_1,\ldots,v_n)\in S^{n-1}$, we denote by $v^{2k}$ the form
		$v^{2k}:=(v_1x_1+\ldots+v_nx_n)^{2k}.$

\begin{proof}[Proof of Lemma \ref{closed-cones}]
		First we will prove that the set  $\Pos:=\Pos^{(n,m)}_{(2k_1,2k_2)}$ is closed. 
	Let $\{p_i\}_{i\in \NN}$ be a sequence from $\Pos:=\Pos^{(n,m)}_{(2k_1,2k_2)}$ converging to some element $p\in \RR[\x,\y]_{2k_1,2k_2}$.
	We have to prove that $p\in \Pos$. For every $u\in S^{n-1}$, $v\in S^{m-1}$ we have that
		$$\frac{1}{(2k_1)!(2k_2)! }\langle p-p_i,u^{2k_1}\otimes v^{2k_2}\rangle_d=(p-p_i)(u,v).$$
	Therefore $\displaystyle p(u,v)=\lim_{i\to\infty}p_i(u,v)$ 
	and hence $p(u,v)\geq 0$ for every $u\in S^{n-1}$, $v\in S^{m-1}$. This proves that $p\in \Pos$ and $\Pos$ is closed.
	
		It remains to prove that the set $\Lf:=\Lf^{(n,m)}_{(2k_1,2k_2)}$ is closed.
		Let $\{\ell_i\}_{i\in \NN}$ be a sequence from $\Lf$ converging to some element $\ell\in \RR[\x,\y]_{2k_1,2k_2}$. 
		We have to prove that $\ell\in \Lf$.  By Caratheodory's theorem \cite[Proposition 2.3]{Rez92} we may assume that each $\ell_i$ is of the form 
			$$\ell_i=\displaystyle\sum_{p=1}^r\left((\sum_{j=1}^n a_{ipj} x_j)^{2k_1}(\sum_{k=1}^m b_{ipk} y_k)^{2k_2}\right),$$
		where $r:=\dim \RR[\x,\y]_{2k_1,2k_2}$, $a_{ipj}\in \RR$, $b_{ipk}\in \RR$ for all $i,p,j,k$ and $\sum_{j=1}^n |a_{ipj}|^2\neq 0$,
		$\sum_{k=1}^m |b_{ipk}|^2\neq 0$ for all $i,p$.
		For all $i,p$ we define
			$$M_{ip}:=\max(|b_{ip1}|,\ldots,|b_{ipm}|).$$
		Note that $M_{ip}>0$.
		For each $p$ there exists $k^{(p)}\in\{1,\ldots,m\}$ such that
		$|b_{ipk^{(p)}}|=M_{ip}$ for infinitely many $i\in \NN$. Passing to subsequences we may assume that
		$|b_{ipk^{(p)}}|=M_{ip}$ for all $p$ and $i\in\NN$. 
		We have
			 $$\ell_i=\sum_{p=1}^r\left((\sum_{j=1}^n M_{ip}^{\frac{2k_2}{2k_1}}a_{ipj} x_j)^{2k_1}(\sum_{k=1}^m \frac{b_{ipk}}{M_{ip}} y_k)^{2k_2}\right)=:
			 	\sum_{p=1}^r\left((\sum_{j=1}^n \tilde{a}_{ipj} x_j)^{2k_1}(\sum_{k=1}^m \tilde{b}_{ipk} y_k)^{2k_2}\right).$$
		Note that for all $p,k$ the sequences  $\{|\tilde{b}_{ipk}|\}_{i\in \NN}$ are bounded by 1 and hence the sequences $\{\tilde{b}_{ipk}\}_{i\in \NN}$ 
		have convergent subsequences. Passing to subsequences we may assume that
		all the sequences $\{\tilde{b}_{ipk}\}_{i\in \NN}$ are convergent; we write $b_{pk}$ for their limits. 
		Let $\e_i$ (resp.\ $\f_j$) denote the $i$-th (resp.\ $j$-th) standard basis vector of $\RR^n$ (resp.\ $\RR^m$), i.e., the vector with 1 on the $i$-th (resp.\ $j$-th) component and 0 elsewhere.
		Note that 
			$$\frac{1}{(2k_1)!(2k_2)!} \langle \ell-\ell_i,\e_j^{2k_1}\otimes \f_{k}^{2k_2}\rangle_d=
				(\ell-\ell_i)(\e_j,\f_{k})=\ell(\e_j,\f_{k})-\sum_{p=1}^r \tilde a_{ipj}^{2k_1}\tilde b_{ipk}^{2k_2}.$$
		Since $\ell_i$ converges to $\ell$ in the apolar inner product, it follows that 
		$$\displaystyle \ell(\e_j,\f_{k})=\lim_{i\to\infty}\sum_{p=1}^r \tilde a_{ipj}^{2k_1}\tilde b_{ipk}^{2k_2}.$$
		Therefore for all $p,j,k$ the sequences $\{   \tilde a_{ipj}^{2k_1}\tilde b_{ipk}^{2k_2}  \}_{i\in \NN}$ are bounded above 
		and hence have convergent subsequences. Passing to subsequences we may assume that all the sequences $\{   \tilde a_{ipj}^{2k_1}\tilde b_{ipk}^{2k_2}  \}_{i\in \NN}$ 
		are convergent.
		Now since $|\tilde b_{ipk^{(p)}}|=1$ for all $i,p$, it follows that for each $p,j$ the sequence
		$\{   \tilde a_{ipj}^{2k_1} \}_{i\in \NN}=\{   \tilde a_{ipj}^{2k_1}\tilde b_{ipk^{(p)}}^{2k_2}  \}_{i\in \NN}$ is convergent, and hence the bounded sequence $\{ \tilde a_{ipj}\}_{i\in \NN}$ can have at most two accumulation points. Passing to subsequences we may assume that all the squences $\{\tilde a_{ipj}\}_{i\in \NN}$ are convergent; we denote the limits by 
		$a_{pj}$.
		Then $$\ell=\displaystyle\sum_{p=1}^r\left((\sum_{j=1}^n a_{pj} x_j)^{2k_1}(\sum_{k=1}^m b_{pk} y_k)^{2k_2}\right)\in \Lf,$$ which concludes the proof of the lemma.
\end{proof}

\begin{proof}[Proof of Theorem \ref{lin-forms}]
	We write $\Lf:=\Lf^{(n,m)}_{(2k_1,2k_2)}$ and $\Pos:=\Pos^{(n,m)}_{(2k_1,2k_2)}$.
	By Lemma \ref{closed-cones}, $\Lf$ and $\Pos$ are closed in the apolar inner product on $\RR[\x,\y]_{2k_1,2k_2}$.
	Since
		$$\left\langle f, u^{2k_1}\otimes v^{2k_2}\right\rangle_d = (2k_1)!(2k_2)! f(u,v)\quad \text{for all } f\in \RR[\x,\y]_{2k_1,2k_2},\; u\in \Sym^{n-1},\; v\in \Sym^{m-1},$$
	we have that
		$$ \Lf_d^\ast=\Pos\quad\text{and}\quad \Lf=\Pos_d^\ast, $$
	where $\Lf_d^\ast$ (resp.\ $\Pos_d^\ast$) is the dual to the cone $\Lf$ (resp.\ $\Pos$) in the apolar inner product.
	In particular, 
		\begin{equation}\label{duality-1}
			\widetilde{\Lf}=\widetilde{\Pos}_d^\ast.
		\end{equation}
	Let $\Pos^\ast\subseteq \RR[\x,\y]_{2k_1,2k_2}$ and
	$\widetilde{\Pos}^\circ\subseteq \cM$ be the dual cone of $\Pos$ and the polar dual of ${\widetilde \Pos}$ in the $L^2(\sigma)$ inner product, respectively.	
	By an analogous reasoning as for the equality $(\Sq ')^\circ=-\widetilde{\Sq}^\ast$ in the proof of the lower bound in Theorem \ref{squares}, we conclude that
		\begin{equation}\label{duality-2}
			\widetilde{\Pos}^\circ=-\widetilde{\Pos}^\ast.
		\end{equation}
	By (\ref{duality-1}) and (\ref{duality-2}) we have that
		$$\left(\frac{\Vol \widetilde{\Lf} }{\Vol B_{\cM} }\right)^{\frac{1}{D_{\cM}}}=
			\left(\frac{\Vol \widetilde{\Pos}_d^\ast }{\Vol B_{\cM} }\right)^{\frac{1}{D_{\cM}}}=
			\left(\frac{\Vol \widetilde{\Pos}_d^\ast }{\Vol  \widetilde{\Pos}^\ast }\right)^{\frac{1}{D_{\cM}}}
			\left(\frac{\Vol \widetilde{\Pos}^\circ }{\Vol B_{\cM} }\right)^{\frac{1}{D_{\cM}}}.
		$$
	Since $\Pos$ has $(\sum_{i=1}^n x_i^2)^{k_1}\otimes (\sum_{j=1}^m y_j^2)^{k_2}$
	as an interior point and $\int_T f \dd\sigma>0$ for all non-zero $f\in \Pos$, we can estimate 
	$\left(\frac{\Vol \widetilde{\Pos}_d^\ast }{\Vol  \widetilde{\Pos}^\ast }\right)^{\frac{1}{D_{\cM}}}$ by Lemma \ref{cone-L} and obtain
		\begin{equation}\label{1est}
			\frac{k_1!k_2!}{(\frac{n}{2}+2k_1)^{k_1}(\frac{m}{2}+2k_2)^{k_2}} 
			\leq  \left(\frac{\Vol \widetilde{\Pos}_d^\ast }{\Vol  \widetilde{\Pos}^\ast }\right)^{\frac{1}{D_{\cM}}}
			\leq  \left(\frac{k_1!k_2!}{(\frac{n}{2}+2k_1)^{k_1}(\frac{m}{2}+2k_2)^{k_2}}\right)^{\alpha_{2k_1,2k_2}},
		\end{equation}
	where 
	$\alpha_{2k_1,2k_2}$ is defined as in the statement of the theorem.
	Using the lower bound in the estimate (\ref{1est}) together with the estimate (\ref{estimate-10}) proves
	the lower bound in Theorem \ref{lin-forms}.
	By the estimate (\ref{Blaschke}) and the equality (\ref{duality-2}) we have that
		$$\left( \frac{\Vol \widetilde{\Pos^\ast}}{\Vol B_{\cM}}\right)^{\frac{1}{D_{\cM}}} \leq \left(\frac{B_{\cM}}{\Vol \widetilde{\Pos}}\right)^{\frac{1}{D_{\cM}}}.$$
	Using the lower bound in Theorem \ref{psd-intro} and the upper bound in the estimate (\ref{1est})
	proves the upper bound in Theorem \ref{lin-forms}.
\end{proof}

\end{document}